\documentclass[twoside,11pt]{article}

\usepackage{blindtext}

\usepackage{jmlr2e}

\usepackage{xurl}
\usepackage{amsmath,amssymb,amsfonts,mathabx}
\usepackage{bbm}
\usepackage{tikz-cd} %
\usepackage{booktabs}
\usepackage{enumerate}
\usepackage{multirow}

\DeclareMathOperator*{\argmin}{argmin} 
 
\DeclareMathOperator{\Nrv}{Nrv} 
 
\DeclareMathOperator{\im}{im} 
\DeclareMathOperator{\Dgm}{Dgm} 
\DeclareMathOperator{\rank}{rank}

\DeclareMathOperator{\Lk}{Lk}
\newcommand{\GH}{\text{GH}}

\usepackage{lastpage}

\ShortHeadings{Analysing Multiscale Clusterings with Persistent Homology}{Schindler and Barahona}
\firstpageno{1}

\begin{document}

\title{Analysing Multiscale Clusterings with Persistent Homology}

\author{\name Juni Schindler \email juni.schindler19@imperial.ac.uk \\
       \addr Department of Mathematics, Imperial College London, London, UK\\
       Department of Mathematical Modeling and Machine Learning, University of 
       Zurich, Switzerland\AND
       \name Mauricio Barahona \email m.barahona@imperial.ac.uk \\
       \addr Department of Mathematics, Imperial College London, London, UK}

\maketitle

\begin{abstract}%
In data clustering, it is often desirable to find not just a single partition into clusters but a sequence of partitions that describes the data at different scales (or levels of coarseness). A natural problem then is to analyse and compare, as a whole, the (not necessarily hierarchical) sequences of partitions that underpin such multiscale descriptions. Using tools from topological data analysis, we introduce here the Multiscale Clustering Filtration (MCF), a filtration of abstract simplicial complexes that encodes arbitrary cluster assignments in a sequence of partitions across scales of increasing coarseness. We prove the stability of the MCF with respect to small changes in the sequence of partitions, and establish that its persistence diagram (PD)  provides an interpretable representation of the multiscale clustering: at dimension 0, it captures violations of the refinement order, and at higher dimensions, it tracks the emergence and resolution of higher-order inconsistencies between cluster assignments. We also introduce an equivalent construction of the MCF via a nerve complex filtration, and show that this nerve-based MCF has computational advantages in certain settings, aiding scalability. Further, we show that in the hierarchical case, the MCF reduces to a Vietoris-Rips filtration of an ultrametric space, and we establish that pairwise measures, such as conditional entropy or zigzag persistence, are fully determined by a graph representation of the Sankey diagram and thus fail to describe the higher-order inconsistencies across scales that MCF captures. Using synthetic and real-world data from stochastic block models, single-cell transcriptomics and temporal social interactions, we illustrate how the MCF PD serves to characterise and compare multiscale clusterings beyond pairwise or hierarchical baseline methods.
\end{abstract}


\section{Introduction}

Data clustering,
whereby groups of similar data points (or `clusters')  in a dataset are found in an unsupervised manner, has found widespread applications across  disciplines~\citep{jainDataClusteringReview1999, luxburgClusteringScienceArt2012, schindlerCommunityVagueOperator2023}. Often, a single partition of a dataset into clusters does not provide an appropriate description, specifically when the dataset has intrinsic structure at several levels of resolution (or coarseness)~\citep{delvenneStabilityGraphCommunities2010,schaubMarkovDynamicsZooming2012, schaubEncodingDynamicsMultiscale2012}. Examples include grouping cells into cell types and subtypes based on similar patterns of gene expression measured through single-cell transcriptomics~\citep{hoekzemaMultiscaleMethodsSignal2022, venkatMultiscaleGeometricTopological2023}; extracting patterns in human mobility data at different spatial scales~\citep{schindlerMultiscaleMobilityPatterns2023, andrisHumannetworkRegionsEffective2023}; or finding groups of documents that fall under finer and broader thematic categories~\citep{altuncuFreeTextClusters2019, grootendorstBERTopicNeuralTopic2022}. In such cases, it is desirable to find \emph{multiscale clusterings}, a sequence of partitions at multiple levels of resolution parametrised by a scale parameter $t$ that captures different descriptions of the data. 
Classically, such descriptions have emerged through variants of hierarchical clustering where $t$ corresponds to the depth of the dendrogram~\citep{carlssonCharacterizationStabilityConvergence2010,carlssonAxiomaticConstructionHierarchical2013,cohen-addadHierarchicalClusteringObjective2019}. Yet imposing a strict hierarchy on the data does not always capture the complexity of relationships across levels of resolution, and can be restrictive in many applications. Therefore alternative formulations that generate \emph{non-hierarchical} multiscale clusterings %
have been proposed, %
e.g., in diffusion-based clustering, where $t$ is the increasing time horizon of a random walker~\citep{azranSpectralMethodsAutomatic2006,lambiotteLaplacianDynamicsMultiscale2009,delvenneStabilityGraphCommunities2010, lambiotteRandomWalksMarkov2014}; in single-cell clustering, where $t$ corresponds to the granularity of cell groupings~\citep{kiselevSC3ConsensusClustering2017,stuartComprehensiveIntegrationSingleCell2019}; or in temporal clustering, where $t$ corresponds to physical time~\citep{rosvallMappingChangeLarge2010,bovetFlowStabilityDynamic2022}.

A natural problem then is to analyse and characterise non-hierarchical sequences of partitions, viewed as whole descriptions of the data, and to compare such multiscale clusterings. Methods to analyse hierarchical sequences of partitions are well established in the literature; in particular, the correspondence between dendrograms and ultrametric spaces has proved useful for measuring the similarity of hierarchical sequences of partitions~\citep{carlssonCharacterizationStabilityConvergence2010,carlssonAxiomaticConstructionHierarchical2013}. In contrast, the study of non-hierarchical sequences of partitions %
has received less attention. 
Here, we address this problem from the perspective of topological data analysis (TDA)~\citep{carlssonTopologyData2009,deyComputationalTopologyData2022} and introduce %
the \emph{Multiscale Clustering Filtration} (MCF), which naturally encodes the topology of cluster assignments across the scale $t$ in a (non-hierarchical) sequence of partitions. The MCF allows us
to take into account the whole sequence of partitions in an integrated manner. We then use persistent homology (PH)~\citep{edelsbrunnerTopologicalPersistenceSimplification2002, otterRoadmapComputationPersistent2017} to track the emergence and resolution of \emph{conflicts} between cluster assignments %
in such non-hierarchical sequences of partitions.

\subsection{Problem Definition}\label{sec:problem_def}

Let $X$ be a finite set with $N:=\lvert X\rvert$ elements. A \emph{partition} $\pi=\{C_1, \dots, C_c\}$ of $X$ into $\lvert\pi\rvert:=c$ clusters is a collection of non-empty subsets $C_i \subseteq X$, with size $\lvert C_i\rvert$, such that $C_i\cap C_j=\emptyset$ for $i\neq j$ and $X=\bigcup_{i=1}^cC_i$~\citep{brualdiIntroductoryCombinatorics2010}. 
The partition $\pi$ induces an equivalence relation $\sim_\pi$ on $X$, where $x\sim_\pi y$ if $x,y\in X$  are in the same cluster of $\pi$, and the equivalence classes of $\sim_\pi$ are the 
clusters~$C_i$. 

Let $\Pi_X$ denote the set of partitions of $X$.
We say that $\pi\in \Pi_X$ is a \emph{refinement} of $\pi'\in \Pi_X$, denoted by $\pi\le\pi'$, if every cluster in $\pi$ is contained in a cluster of $\pi'$. This makes $(\Pi_X,\le)$ a finite partially ordered set (or \emph{poset})~\citep{birkhoffLatticeTheory1967} with minimum $\hat 0=\{\{x\} \;|\; x\in X\}$, the partition of singletons, and maximum $\hat 1=\{X\}$, the partition with a single cluster $X$.

Here we are interested in \emph{multiscale clusterings}, sequences of partitions which can be represented through a piecewise-constant %
function, as follows.
\begin{definition}[Sequence of partitions]\label{def:sequence_of_partitions}
A \emph{sequence of partitions} is a piecewise-constant function
$\theta\colon[t_1,\infty)\rightarrow\Pi_X, \, t\mapsto \theta(t)$ so that each \emph{scale} $t\ge t_1$ is assigned a partition of $X$. In particular, the function $\theta$ has $M \geq 2$  \emph{change points} $t_1 < t_2 < \dots < t_M$ such that $\theta(t)=\theta(t_m)$ for $t\in [t_m,t_{m+1})$, for $m=1,\dots, M-1$, and $\theta(t)=\theta(t_M)$ for $t\in[t_M,\infty)$.
\end{definition}
The sequence $\theta$ is called \emph{hierarchical} if $\theta(s) \le  \theta(t), \forall s\le t$, and \emph{non-hierarchical} otherwise.

Our goal in this paper is to define an invariant for multiscale clusterings $\theta$, including non-hierarchical sequences of partitions, which takes into account memory effects across the scale $t$. 
The invariant should be an interpretable measure that captures non-hierarchy and higher-order inconsistencies across $t$ and provides a 
representation of the multiscale clustering $\theta$ that can be used for downstream machine learning tasks.

\begin{remark}
\label{rem:not_doing_clustering}
    We are not concerned here with the problem of computing the (non-hierarchical) sequence of partitions $\theta$ from a dataset, for which a plethora of methods exists in the literature. Rather, we take $\theta$ as a given, and provide a methodology to analyse its structure. Hence, the method is applicable to any chosen multiscale clustering method.
\end{remark}

\begin{figure}[htb!]
    \centering
    \includegraphics[width=\textwidth]{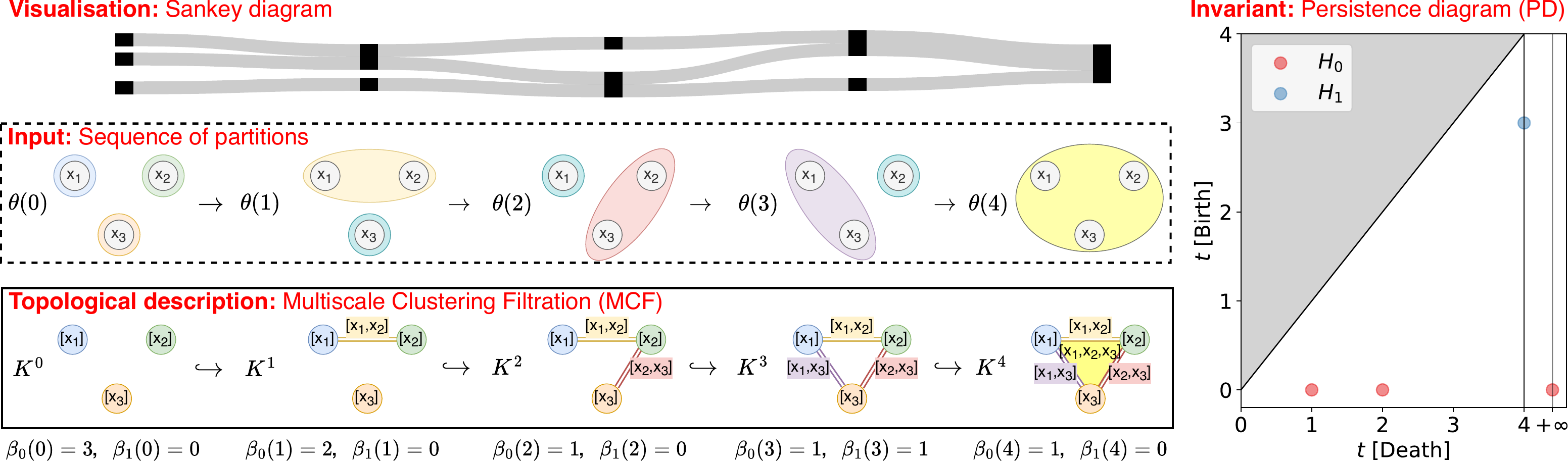}
    \vspace{-0.7cm}
    \caption{Running example to illustrate MCF. 
    Our input is a sequence of partitions $\theta$, a multiscale clustering of the set $X=\{x_1, x_2, x_3\}$ across the scale $t$, here defined for $t \geq t_1=0$ with $M=5$ change points $t_m=m-1$, $m=1,\ldots,5$. The sequence $\theta$ is sometimes visualised with the aid of a Sankey diagram. Our method encodes $\theta$ as the Multiscale Clustering Filtration (MCF) $\mathcal{M}_\theta:=(K^t)_{t\ge t_1}$, a filtration of abstract simplicial complexes that codifies cluster intersection patterns across $t$ (Section~\ref{sec:MCF_construction}).
    The persistent homology of the MCF captures the violation of the refinement order at dimension $k=0$ (Section~\ref{Sec:MCF_zero_dim}) and higher-order inconsistencies at dimension $k\ge 1$ (Section~\ref{sec:higher_dim_ph}). The stable persistence diagram (PD) of the MCF, which is a complete invariant of the PH, serves as a summary of the multiscale structure of $\theta$ 
    (Section~\ref{sec:mcf_stability}), and can be used as a feature map (i.e., a representation) for downstream learning tasks 
    (Section~\ref{sec:NumericalExperiments}).
    }
    \vspace{-0.4cm}
    \label{fig:MCF_toy_example}
\end{figure}

\subsection{Contributions and Outline}

We introduce a well-defined filtration of abstract simplicial complexes $\mathcal{M}_\theta$, denoted the \emph{Multiscale Clustering Filtration} (MCF), which naturally encodes the topological crossing patterns of cluster assignments in a sequence of partitions $\theta$ (see Figure~\ref{fig:MCF_toy_example} for an illustration).
We show that the persistence diagram (PD) of the MCF is stable under small changes in the joint first-merge times of $\theta$ (Section~\ref{sec:mcf_stability}) and exploit the computable characteristics of the MCF to characterise different aspects of non-hierarchy in multiscale clusterings. In particular: (i) the 0-dimensional~PH of $\mathcal{M}_\theta$ measures violations of the refinement order in $\theta$ (Section~\ref{Sec:MCF_zero_dim}), and (ii) the birth and death times in the $k$-dimensional PH for $k\ge 1$ correspond to the emergence and resolution of higher-order inconsistencies between cluster assignments in $\theta$ (Section~\ref{sec:higher_dim_ph}). In contrast, baseline methods such as ultrametrics~\citep{carlssonCharacterizationStabilityConvergence2010} are not applicable to non-hierarchical sequences of partitions and information-based measures~\citep{meilaComparingClusteringsVariation2003} are restricted to pairwise comparisons of partitions that do not capture higher-order inconsistencies across multiple scales. To broaden the theoretical and practical foundations of the MCF, we provide an equivalent construction via a nerve complex filtration (Section~\ref{Sec:NerveComplex}), which can be advantageous for sequences of partitions $\theta$ with large cluster sizes. We also show that, for hierarchical $\theta$, the MCF coincides with the Vietoris-Rips (VR) filtration of an ultrametric space (Section~\ref{sec:CAG_hierarchical}), and that a simple zigzag filtration defined on $\theta$ does not capture higher-order inconsistencies across more than two partitions (Section~\ref{sec:zigzag_mcf}).

We demonstrate the applicability of our method across various domains and show that the MCF PDs are interpretable feature maps useful for downstream learning tasks on (non-hierarchical) sequences of partitions. Our numerical experiments include: (i) comparing the degree of non-hierarchy and recovering the planted ground-truth structure in sequences of partitions obtained from stochastic block models (Section~\ref{sec:sbms}); (ii) measuring the degree of non-hierarchy in the multiscale structure of synthetic and real single-cell data (Section~\ref{sec:single_cell}); and (iii) analysing the temporal multiscale structure of social interactions in a primary school dataset (Section~\ref{sec:primary_school}). Our experiments illustrate that the higher-order inconsistencies in multiscale clusterings captured by MCF, and missed by current baseline methods, provide improved power for downstream learning tasks. We make available a \texttt{Python} implementation of MCF based on the \texttt{GUDHI} software~\citep{mariaGudhiLibrarySimplicial2014} and also release the code for our experiments.\footnote{\url{https://github.com/barahona-research-group/MCF}}

\subsection{Related Work}\label{sec:related_work} %

A hierarchical sequence of partitions with $\theta(t_1=0)=\hat{0}$ and $\theta(t_M)=\hat 1$ is also called a \emph{dendrogram} and can be represented by a rooted acyclic tree. \citet{carlssonCharacterizationStabilityConvergence2010} showed that there is a one-to-one correspondence between dendrograms and \emph{ultrametrics}, where the distance between two points in the dendrogram is defined as the first-merge time. Dendrograms can then be compared by computing the \emph{Gromov-Hausdorff distance} between the corresponding ultrametric spaces~\citep{memoliGromovHausdorffDistanceUltrametric2023}. However, this approach is not applicable to non-hierarchical sequences, since the first-merge times do not satisfy the strong triangle inequality in general (see Section~\ref{sec:CAG_hierarchical}). %

A non-hierarchical sequence of partitions $\theta$ can be visualised by its \emph{Sankey diagram} $S(\theta)$, an $M$-layered graph that allows for (non-trivial) crossings across layers~\citep{zarateOptimalSankeyDiagrams2018}, see Appendix~\ref{appx:Sankey} for details. 
However, while Sankey diagrams are a useful visualisation tool, they only encode the relations between neighbouring partitions and miss higher-order inconsistencies between more than two partitions (see Section~\ref{sec:zigzag_mcf}). 

In the context of single-cell clustering, \citet{pengCellTypeHierarchy2021} developed the \emph{Multi-resolution Reconciled tree} (\texttt{MRtree}) algorithm, which extracts a tree from the (non-hierarchical) Sankey diagram. The algorithm involves an NP-hard optimisation of the \emph{Hamming distance} between the Sankey diagram and its closest dendrogram. 
However, \texttt{MRtree} is not designed for comparing the non-hierarchical structure of sequences of partitions. %

The problem of comparing two partitions (rather than comparing sequences of partitions) is well studied in the literature. Examples of measures include chance-corrected combinatorial measures, such as the \emph{adjusted Rand index} (ARI) 
~\citep{hubertComparingPartitions1985}, and information-based measures~\citep{meilaComparingClusteringsVariation2003}, such as the \emph{conditional entropy} (CE), the \emph{variation of information} (VI), a metric on $\Pi_X$, 
or its normalised version, the \emph{normalised variation of information} (NVI)~\citep{kraskovHierarchicalClusteringBased2003}. While these measures can be used to compare consecutive partitions, higher-order effects that arise through cluster assignment conflicts across more than two partitions are missed by such pairwise measures. %

Finally, \emph{formigrams} were recently proposed by \citet{kimFormigramsClusteringSummaries2018} to analyse and compare the evolution of connected components in temporal graphs. In particular, the patterns of merging and splitting from one time step to the next are analysed with 0-dimensional \emph{zigzag persistence}~\citep{carlssonZigzagPersistence2010}. However, this framework is not directly applicable to arbitrary sequences of partitions $\theta$, as it relies on the comparability of consecutive partitions in the partition lattice~\citep{kimAnalysisDynamicGraphs2020}. Moreover, zigzag persistence cannot capture higher-order inconsistencies across more than two time steps (see Section~\ref{sec:zigzag_mcf}). %

We compare MCF to these works numerically in Section~\ref{sec:NumericalExperiments} and theoretically in Section~\ref{sec:discussion}.

\section{The Multiscale Clustering Filtration (MCF)}\label{sec:Methods}

We provide a table of the main mathematical symbols used in this paper in Appendix~\ref{appx:table_of_symbols}. For mathematical background on PH used throughout this paper, see Appendix~\ref{sec:background_ph}.

\subsection{Construction of the MCF}\label{sec:MCF_construction}

We begin by defining the \emph{Multiscale Clustering Filtration} (MCF), which is the central object of our paper.
The MCF takes as input a (non-hierarchical) sequence of partitions $\theta$ (see Definition~\ref{def:sequence_of_partitions}) and encodes the topological structure of its cluster overlaps across the scale $t$ in a filtration $(K^t)_{t\ge t_1}$ of abstract simplicial complexes. For background on simplicial complexes and filtrations we refer to Appendices~\ref{background:complex} and \ref{background:filtration}.

\begin{definition}[Multiscale Clustering Filtration]\label{def:MCF_Construction}
   Given  a sequence of partitions \linebreak
   $\theta\colon[t_1,\infty)\rightarrow\Pi_X, \, t\mapsto\theta(t)$, the \emph{Multiscale Clustering Filtration} (MCF) $\mathcal{M}_\theta:=(K^t_\theta)_{t\ge t_1}$ is a filtration of abstract simplicial complexes 
    \begin{equation}\label{eq:MCF_construction}
        K^t_\theta := \bigcup_{t_1 \le s \leq t} \bigcup_{C \in\theta(s)} \Delta C, %
    \end{equation}
    where $\Delta C:=2^C\setminus\{\emptyset\}$ is the $(\lvert C\rvert-1)$-dimensional solid simplex defined on cluster $C\subseteq X$.
\end{definition}
To define a simplicial complex from the sequence of partitions, we thus interpret each cluster $C$ as a solid simplex $\Delta C$, which by definition contains all its lower dimensional simplices. This construction resonates with the view of data clustering as information compression or coarse-graining~\citep{rosvallMapsRandomWalks2008,lambiotteRandomWalksMarkov2014}, and with clusters seen as equivalence classes~\citep{brualdiIntroductoryCombinatorics2010}. The complex $K^t_\theta$ then aggregates the information about all cluster assignments in $\theta$ up to scale $t$ %
by taking a union over solid simplices. We show in Sections~\ref{Sec:MCF_zero_dim} and \ref{sec:higher_dim_ph} that $\mathcal{M}_\theta$ can be used to measure and characterise different types of non-hierarchy in $\theta$. Henceforth, when $\theta$ is understood from the context, we drop the $\theta$ subindex and write  $\mathcal{M}$ and $K^t$ for notational simplicity.

We first prove %
that the MCF is indeed a well-defined filtration.

\begin{proposition}\label{prop:MCF_is_filtration}
    The MCF $\mathcal{M}=(K^{t})_{t \ge t_1}$ is a filtration of abstract simplicial complexes and can only change 
    at the change points $t_1<t_2<\dots<t_M$. %
\end{proposition}

\begin{proof}
    A solid simplex $\Delta C$ is a simplicial complex. Since $K^t$ is the union of simplicial complexes, then it is a simplicial complex for every $t\ge t_1$, and  $K^{t} \subseteq K^{t^\prime}$ for $t\leq t^\prime$. Hence, $(K^{t})_{t\ge t_1}$ is a filtration. The statement about change points follows by definition: $K^t=K^{t_m}$ for all $t\in[t_m,t_{m+1})$ and $K^t=K^{t_M}$ for all $t\ge t_M$.
\end{proof}

We illustrate the MCF construction with a toy example, to which we will come back throughout the paper.

\begin{example}[Running example, 
Figure~\ref{fig:MCF_toy_example}]\label{ex:MCF_illustration}
    Let $X=\{x_1,x_2,x_3\}$ and consider the sequence of partitions $\theta$   
    with five change points such that $\theta(0)=\hat 0$, $\theta(1)=\{\{x_1,x_2\},\{x_3\}\}$, $\theta(2)=\{\{x_1\},\{x_2,x_3\}\}$, $\theta(3)=\{\{x_1,x_3\},\{x_2\}\}$ and $\theta(4)=\hat 1$. The MCF $(K^t)_{0\le t \le 4}$ is given by the abstract simplicial complexes $K^0=\{[x_1],[x_2],[x_3]\}$, $K^1=\{[x_1],[x_2],[x_3],[x_1,x_2]\}$, $K^2=\{[x_1],[x_2],[x_3],[x_1,x_2],[x_2,x_3]\}$, $K^3=\{[x_1],[x_2],[x_3],$ $[x_1,x_2],$ $[x_2,x_3],$ $[x_1,x_3]\}$ and $K^4=\{[x_1],[x_2],[x_3],[x_1,x_2],[x_2,x_3],$
    $[x_1,x_3],[x_1,x_2,x_3]\}=2^X\setminus\{\emptyset\}$. 
\end{example}

This toy example highlights several aspects of the MCF, which we expand in the remarks below. First, the example shows the importance of the ordering of partitions in~$\theta$.

\begin{remark}%
\label{rem:ordering}
To illustrate the role played by the ordering of the partitions, consider the effect of swapping $\theta(4)$ and $\theta(0)$ in Example~\ref{ex:MCF_illustration}. This would yield $K^t=2^X\setminus\{\emptyset\}$ for $t\ge 0$, and the filtration would be immediately saturated incorporating no further information  from other partitions.  MCF is thus designed to encode sequences of partitions 
with an ordering that reflects a notion of coarse-graining across scales from fine to coarse, as usually obtained in hierarchical and non-hierarchical multiscale clustering algorithms 
~\citep{lambiotteLaplacianDynamicsMultiscale2009,delvenneStabilityGraphCommunities2010, lambiotteRandomWalksMarkov2014,schaubMarkovDynamicsZooming2012}. A simple heuristic is to order the partitions in decreasing order of their number of clusters, 
or, alternatively, $\theta$ could be re-ordered based on intrinsic properties of the MCF (Remark~\ref{rem:Ordering_PH}).
When $\theta$ is not globally coarse-graining and re-ordering the sequence is not appropriate (e.g., when $t$ is a physical time), one can apply MCF to a rolling window of subsequences of $\theta$, as in Section~\ref{sec:primary_school}. 
\end{remark}

Second, Example~\ref{ex:MCF_illustration} also highlights why other popular filtrations used in TDA are not applicable to studying non-hierarchical sequences of partitions.

\begin{remark}\label{rem:2-determined}
As clusters can merge and then split in non-hierarchical sequences of partitions, the MCF is generally not \emph{2-determined}, i.e., the 1-simplices do not determine the 2-simplices in $K^t$. In
Example~\ref{ex:MCF_illustration}, although every pair of the set $\{x_1,x_2,x_3\}$ is a 1-simplex in $K^3$, the 2-simplex $[x_1,x_2,x_3]$ is not included in $K^3$. This means that 2-determined filtrations, such as the VR filtration (Equation~\ref{eq:VR_point_cloud}) 
or a  clique complex filtration (Equation~\ref{eq:clique_filtration}), are not appropriate for non-hierarchical sequences of partitions.
Indeed, we show in Section~\ref{sec:CAG_hierarchical} that the MCF does reduce to a VR filtration when $\theta$ %
is hierarchical.
\end{remark}

Third, our toy example serves to illustrate how the size of the MCF grows with the size of the problem. In particular, the dimension of $\mathcal{M}$ is given by the dimension of its largest simplex, which depends on the largest cluster in $\theta$.

\begin{remark}\label{rem:mcf_dimension}
    The dimension $\dim K$ of $K:=\bigcup_{t\ge t_1}K^t=K^{t_M}$ is given by the largest cluster in the sequence of partitions $\theta$: $\dim K=\max_{t\ge t_1, \,C\in\theta(t)}\lvert C\rvert-1$. An input containing large clusters %
    can thus lead to a high number of simplices in $K$, since faces of simplices must be included. 
    However, %
    in Section~\ref{Sec:NerveComplex} we introduce an equivalent construction of the MCF based on nerve complexes~\citep{matousekUsingBorsukUlamTheorem2003}   which can be computationally advantageous when $\dim K$ is large, see Remark~\ref{rem:computational_complexity} and Section~\ref{sec:computational_efficiency}.
\end{remark}

\subsection{Stability of the MCF}\label{sec:mcf_stability}

The persistent homology of a filtration can be summarised by the $k$-dimensional persistence diagrams, $\Dgm_k(\mathcal{M})$ for $k \ge 0$. For background on PH and PDs, see Appendices~\ref{background:persistant_homology} and \ref{background:PDs}. In this section we show that the MCF is a \emph{stable} filtration, i.e., that small changes in $\theta$ lead to small changes in the MCF PDs.
To establish a distance between non-hierarchical sequences of partitions, we first define the \emph{joint first-merge times} of $\theta$.

\begin{definition}[Joint first-merge time]\label{def:joint-first-merge-times}
    Given $\theta$ defined on the set $X$, the \emph{joint first-merge time} of elements $x_0,\dots, x_k\in X$ is the function $D_{\theta,k}:X^{k+1}\rightarrow \mathbb{R}\cup\{\infty\}$ given by: 
    \begin{equation}\label{eq:joint-first-merge-times}
        D_{\theta,k}(x_0,\dots,x_k) = \min \left\{t\ge t_1 |\; \exists C \in\theta(t): x_0,\dots,x_k\in C\right\},
    \end{equation}
    with $\min\emptyset = +\infty$. For $k=1$, this function particularises to the $N \times N$ \emph{matrix of pairwise first-merge times} $D_\theta:=D_{\theta, 1}$. %
    We study $D_\theta$ in detail in Section~\ref{sec:CAG_hierarchical}. 
    
\end{definition}

\begin{remark}
    The function $D_{\theta,k}$ is symmetric under permutation of the elements $x_i$. 
    Furthermore, $D_{\theta,0}(x_k)=D_{\theta,1}(x_k,x_k)=t_1$, and $D_{\theta,k}(x_0,\dots,x_k)=D_{\theta,k+1}(x_0,\dots,x_k,x_k)$. Hence, we can always retrieve $D_{\theta,k}$ from $D_{\theta,k+1}$ for $k\ge 0$ by evaluating its partial diagonal. However, we cannot obtain $D_{\theta,k+1}$ from $D_{\theta,k}$ for $k\ge 1$ because elements of $X$ that have merged in $\theta$ can split again if $\theta$ is non-hierarchical.
    Example~\ref{ex:running_clique_complex} shows that $D_{\theta,2}$ cannot be obtained from $D_{\theta,1}$ in our running 
    Example~\ref{ex:MCF_illustration}.
\end{remark}

The next proposition states the stability of the MCF PDs with respect to the Wasserstein distance (Equation~\ref{eq:Wasserstein_distance} in  Appendix~\ref{background:PD_distance}). %
\begin{proposition}\label{prop:mcf_stability}
    Consider two sequences of partitions $\theta\colon[t_1,\infty)\rightarrow \Pi_X$ and $\eta:[\tilde{t}_1,\infty)\rightarrow \Pi_X$ with corresponding MCF filtrations $\mathcal{M}_\theta=(K^{t}_\theta)_{t\ge t_1}$ and $\mathcal{M}_\eta=(\tilde{K}^{t}_\eta)_{t\ge \tilde{t}_1}$ such that $K:=\bigcup_{t\ge t_1}K^t_\theta = \bigcup_{t\ge \tilde{t}_1}K^t_\eta$. Then we get for $k \ge 0$ that
    \begin{equation}\label{eq:stability_MCF}
        d_{W,q}\left(\Dgm_k(\mathcal{M_\theta}),\Dgm_k(\mathcal{M}_\eta)\right)^q \; \le\; \frac{1}{(k+1)!}\; \|D_{\theta,k+1}-D_{\eta,k+1}\|_q^q,
    \end{equation}
    where $d_{W,q}$ is the $q$-th Wasserstein distance in Equation~\eqref{eq:Wasserstein_distance} with $q\ge 1$ and we use the convention $\infty - \infty := 0$.
\end{proposition}
\begin{proof}
    See Appendix~\ref{appx:proofs_stability} for a proof, which relies on results 
    by~\citet{skrabaWassersteinStabilityPersistence2025}. 
\end{proof}

The assumption of a common complex $K$ is automatically satisfied if both $\theta$ and $\eta$ eventually reach $\hat 1$. Our stability result shows that the $k$-dimensional MCF PD remains stable under small changes in the joint first-merge times $D_{\theta,k+1}$. In particular, we have stability with respect to small shifts in the change points of $\theta$ (e.g., $\eta(t):=\theta(t-\epsilon)$ for $t\ge t_1+\epsilon$ and $\epsilon>0$ small) or changes to cluster assignments in $\theta$ that leave $K$ unchanged and only lead to small differences in the joint first-merge times. %

Proposition~\ref{prop:mcf_stability} can be sharpened for dimension $k=0$, as follows. 
\begin{corollary}%
For $k=0$, the inequality~\eqref{eq:stability_MCF} in Proposition~\ref{prop:mcf_stability} becomes  
\begin{equation} \label{eq:stability_0_dim}
d_{W,q}\left(\Dgm_0(\mathcal{M}_\theta),\Dgm_0(\mathcal{M}_\eta)\right)^q \; \le\; \frac{1}{2} \left(N \, |t_1-\tilde{t}_1|^q+\|D_\theta-D_{\eta}\|_q^q\right).
\end{equation}
\end{corollary}

\begin{proof}
This follows from Equation~\eqref{eq:skraba_stability} (with $f$, $\tilde f$ and $g$ defined as in Appendix~\ref{appx:proofs_stability}) and  
the fact that $f([x])=t_1$ for all $x\in X$,  $D_\theta[x,y]=f(g(x,y))$ for all $[x,y]\in K$, and $D_\theta[x,y]=\infty$ for all $x,y\in X$ with $[x,y]\notin K$ (and similarly for $\tilde{f}$).
\end{proof}
Therefore, the 0-dimensional PDs of $\theta$ and $\eta$ are close when $|t_1-\tilde{t}_1|$ is small %
and the difference of pairwise first-merge times %
is small for all pairs of elements $x,y\in X$.

Our stability results show that 
the Wasserstein distance between MCF PDs can be used to compare full sequences of partitions, as illustrated through 
numerical experiments in Section~\ref{sec:NumericalExperiments}. 

\begin{remark}
The inequalities~\eqref{eq:stability_MCF} and \eqref{eq:stability_0_dim} are related to the stability theorem for single-linkage hierarchical clustering by \citet{carlssonCharacterizationStabilityConvergence2010}, which is based on the Gromov-Hausdorff distance between the ultrametric spaces associated with two \emph{hierarchical} sequences of partitions. Our MCF stability result extends to non-hierarchical sequences of partitions and 
the properties of $D_\theta$ and the stability of MCF for the hierarchical case are studied in detail in Section~\ref{sec:CAG_hierarchical}.
\end{remark}

We now study the persistent homology of the MCF at different dimensions and show that the PD reveals the violations of the refinement order and the emergence and resolution of conflicts between cluster assignments in non-hierarchical sequences of partitions. 

\subsection{0-dimensional PH of the MCF %
}\label{Sec:MCF_zero_dim}

In this section, we show that the 0-dimensional PH of the MCF, which captures the connected components in the filtration, allows us to characterise the loss of hierarchy in the sequence of partitions as a result of violations of the refinement order.
As all vertices in the MCF are born at $t_1$, we focus on the 0-dimensional homology groups $H_0^t$ of $K^t$, $t\ge t_1$ (see Appendix~\ref{background:homology} for background on simplicial homology). 
\begin{definition}[Fractured partition]\label{def:fractured}
Let $\theta\colon[t_1,\infty)\rightarrow \Pi_X, t \mapsto \theta(t)$ be a sequence of partitions.
    We say that the partition $\theta(t)$ is \emph{non-fractured} if  for all $t_1\le s\le t$ the partitions $\theta(s)$ are refinements of $\theta(t)$, i.e., $\theta(s)\le\theta(t)$. Otherwise, $\theta(t)$ is called \emph{fractured}.
\end{definition}
Clearly, $\theta$ is hierarchical if and only if partitions $\theta(t)$ are non-fractured for all $t\ge t_1$. %

We now show that the level of hierarchy in $\theta$ can be quantified by comparing the 0-dimensional Betti number $\beta_0(t)$ of $K^t$ to the number of clusters $\lvert\theta(t)\rvert$.
\begin{proposition}\label{prop:0-dim-Betti}
    For each $t\ge t_1$, the 0-th Betti number 
    $\beta_0(t)$ of the MCF fulfils the following properties:
    \begin{enumerate}
        \item $\beta_0(t)\le \min_{t_1 \le s\le t} \lvert\theta(s)\rvert$
        \item $\beta_0(t)=\lvert\theta(t)\rvert$ $\iff$ $\theta(t)$ is non-fractured
    \end{enumerate}
\end{proposition}

\begin{proof}
    1) By definition, $\beta_0(t)$ equals the number of connected components in $K^t$. The complex $K^t$ contains the solid simplices of the clusters of partitions $\theta(s), \, s\le t$. 
    Hence, $K^t$ has at most $\min_{s\le t} \lvert\theta(s)\rvert$ connected components, where $\lvert\theta(s)\rvert$ is the number of clusters in $\theta(s)$. 
    2) ``$\Longleftarrow$'' If $\theta(t)$ is non-fractured, then the clusters of $\theta(s)$ are nested within the clusters of $\theta(t)$ for all $s \le t$. Hence, the maximally disjoint simplices of $K^t$ are given by the solid simplices corresponding to the clusters of $\theta(t)$, implying $\beta_0(t) = \lvert\theta(t)\rvert$.
    ``$\Longrightarrow$'' Proof by the contrapositive. If $\theta(t)$ is fractured, then $\exists \,  s < t$ and $x,y \in X$ such that $x\sim_{\theta(s)}y$ but $x\not\sim_{\theta(t)}y$. 
 This means that: (i) $\exists \, C\in\theta(s)$ with $x,y\in C$; hence, $[x,y]\in\Delta C\subseteq K^s$; and (ii) $\exists \,  C',C''\in\theta(t)$ with $x\in C'$ and $y\in C''$ and $C'\cap C''=\emptyset$,  due to $x\not\sim_{\theta(t)}y$.
Clearly, $x,y$ are path-connected in $K^s$ and, due to $K^s\subseteq K^t$, they are path-connected in $K^t$; hence, $\Delta C'$ and $\Delta C''$ are connected in $K^t$ via the 1-simplex $[x,y]\in K^{t}$. %
    It follows that the number of clusters in $\theta(t)$ is larger than the number of connected components in $K^t$: $\beta_0(t) < \lvert\theta(t)\rvert$. 
\end{proof}
The number of clusters $\lvert\theta(t)\rvert$ is thus an upper bound for the Betti curve $\beta_0(t)$ and this motivates the following definitions.
\begin{definition}[0-conflict]\label{rem:0-conflicts}
    We say that $\theta$ has a 0-conflict at scale $t$ if $\theta(t)$ is fractured. 
\end{definition}

\begin{remark}
    By Proposition~\ref{prop:0-dim-Betti}, $\theta$ has a 0-conflict at $t$ iff the total number of 0-dimensional generators in $\Dgm_0(\mathcal{M})$ still alive at scale $t$ is smaller than the number of clusters $\lvert\theta(t)\rvert$.
\end{remark}
\begin{definition}[Persistent 0-conflict]\label{def:persistent_0_conflict}
    For $t\ge t_1$, we define the \emph{persistent 0-conflict} as
    \begin{equation}\label{eq:persistent_hierachy}
        0 \le c_0(t):=\lvert\theta(t)\rvert-\beta_0(t).
    \end{equation}
\end{definition}
Here, $c_0(t)$ is a piecewise-constant right-continuous function with $c_0(t_1)=0$ by construction.  %
High values of $c_0(t)$ indicate a high level of non-hierarchy in the sequence of partitions. %
\begin{remark}\label{rem:betti_0_join}
    The connected components of $K^t$ are the clusters of $\bigvee_{s\leq t}\theta(s)$, the join in the partition lattice $(\Pi_X,\le)$~\citep{birkhoffLatticeTheory1967}, which is the finest partition coarser than all $\theta(s)$, $s\le t$. Hence,  $\beta_0(t)=\lvert\bigvee_{s\leq t}\theta(s)\rvert$. 
    Furthermore, replacing $\theta(t)$ by a coarser partition $\pi\ge\theta(t)$ yields a non-fractured partition if and only if $\pi\ge \bigvee_{s\leq t}\theta(s)$. %
    Hence, $c_0(t)$ is the minimal number of merges of clusters in $\theta(t)$ needed to make it non-fractured.
\end{remark}
We can use $c_0(t)$ to state a necessary and sufficient condition for a sequence of partitions to be hierarchical. 
\begin{corollary}\label{cor:persistent_0_conflict_hierarchy}
    $c_0(t)=0, \, \forall t\ge t_1$ $\iff$ %
    $\theta$ %
    is hierarchical.
\end{corollary}
\begin{proof}
    By Proposition~\ref{prop:0-dim-Betti}, $c_0(t)= 0$ for $t\ge t_1$ is equivalent to $\theta(t)$ being non-fractured for all $t\ge t_1$, which in turn is equivalent to $\theta$ being hierarchical. %
\end{proof}
The results above show that $c_0(t)$ measures to what degree the sequence of partitions $\theta(t)$ deviates from a perfectly hierarchical sequence. This motivates the following normalised measure of average 0-conflict.
\begin{definition}[Average 0-conflict]\label{def:average_0_conflict}
The \emph{average 0-conflict} $\bar c_0$ of a sequence of partitions $\theta$ is defined as:
\begin{equation}\label{eq:average_0_conflict}
    0 \le \bar c_0 := \frac{1}{T-t_1} \int_{t_1}^{T}\frac{c_0(t)}{\lvert \theta(t)\rvert}\;dt  %
    \le 1-\frac{1}{N},
\end{equation}
where $T:=t_M+\frac{t_M-t_1}{M-1}$ for $M\ge 2$.
\end{definition}
From Definition~\ref{def:persistent_0_conflict} and Corollary~\ref{cor:persistent_0_conflict_hierarchy}, a hierarchical sequence yields $\bar c_0=0$ and a non-hierarchical sequence has $\bar c_0>0$. 
We illustrate these definitions with our running example.

\begin{example}%
\label{Ex:running_hierarchy}
    Let $(K^t)_{0\le t \le 4}$ be the MCF in Example~\ref{ex:MCF_illustration}. The persistent 0-conflict is: $c_0(0)=c_0(1)=0$, $c_0(2)=c_0(3)=1$ and $c_0(4)=0$, where the non-zero persistent 0-conflicts at $t=2$ and $t=3$ indicate a violation of hierarchy because partitions $\theta(2)$ and $\theta(3)$ are fractured. In both cases, one merge of clusters is needed to obtain the join $\hat{1}$. This leads to a positive average 0-conflict $\bar c_0=0.20>0$, which quantifies the degree of non-hierarchy of $\theta$. As discussed in Remark~\ref{rem:ordering}, the ordering of the sequence is crucial and swapping  $\theta(4)$ and $\theta(0)$ would lead to an increased  $\bar c_0=0.43$, indicating additional fracturedness.
\end{example}
This observation suggests that the ordering of the partitions in the sequence could be optimised to minimise the average 0-conflict. 
\begin{remark}\label{rem:Ordering_PH}
   Let $\theta_\kappa$ be the sequence where the partitions of $\theta$ are permuted according to $\kappa$, i.e., $\theta_\kappa(t_i):=\theta(t_{\kappa(i)})$. 
    A maximally hierarchical ordering $\kappa^*$ can be achieved by minimising the average 0-conflict, i.e., $\kappa^*:=\argmin_\kappa \bar c_0(\theta_\kappa)$. This is a combinatorial optimisation problem and future work will aim to develop efficient optimisation heuristics.
\end{remark}

\subsection{Higher-dimensional PH of the MCF %
}\label{sec:higher_dim_ph}

In this section, we study the MCF PH for dimensions $k\ge 1$, which captures non-bounding cycles or voids in the filtration. For simplicity, we assume that the PH is computed over the two-element field $\mathbbm{Z}_2$ but our theoretical results hold for a general field. We first illustrate with our toy example how the higher-dimensional PH tracks the emergence and resolution of higher-order cluster assignment conflicts across the sequence of partitions $\theta$.

\begin{example}%
\label{ex:1_conflict}
    In %
    Example~\ref{ex:MCF_illustration}, $x_1$, $x_2$ and $x_3$ are in a pairwise conflict at $t=3$ because each pair of elements has been assigned to a common cluster but all three elements have never been assigned to the same cluster up to $t=3$. This means that $K^3$ contains the 1-simplices $[x_1,x_2]$, $[x_2,x_3]$ and $[x_3,x_1]$ but not the 2-simplex $[x_1,x_2,x_3]$. Hence, the 1-chain $[x_1,x_2]+[x_2,x_3]+[x_3,x_1]$ is a non-bounding 1-cycle that corresponds to the generator of the 1-dimensional homology group $H^3_1=\mathbbm{Z}_2$. We call this type of conflict, which is apparent in the 1-dimensional PH, a \emph{1-conflict}. This 1-conflict is resolved at index $t=4$ when $x_1$, $x_2$ and $x_3$ are assigned to the same cluster in partition $\theta(4)$. Then the simplex $[x_1,x_2,x_3]$ is added to $K^4$ such that there are no more non-bounding 1-cycles and $H^4_1=0$. 
\end{example}

This example motivates our interpretation of the cycle-, boundary- and homology groups of the MCF PH in terms of cluster assignment conflicts. We refer to Appendix~\ref{background:homology} for precise definitions from simplicial homology.
\begin{definition}[$k$-conflicts]\label{def:k-conflicts}
    For $t\ge t_1$ and %
    $k \ge 1$, we interpret the elements of the cycle group $Z_k^t$ as \emph{potential $k$-conflicts} and the elements of the boundary group $B_k^t$ as \emph{resolved $k$-conflicts}. We further interpret the classes of the PH group $H^{t,p}_k$ (Equation~\ref{eq:persistent_homology_group}), $p\ge0$, as equivalence classes of \emph{true $k$-conflicts} that have not been resolved until filtration index $t+p$, with birth and death times of true conflicts corresponding to the %
    emergence and resolution of the conflict. 
\end{definition}

The total number of unresolved true $k$-conflicts at index $t$ for $k\ge 1$ is given by $\beta_k(t)$, which we define as our unnormalised measure of \emph{persistent $k$-conflict}, and  in analogy with Definition~\ref{def:average_0_conflict} we also define the average $k$-conflict,  $\bar c_k$.
\begin{definition}[Persistent $k$-conflict and average $k$-conflict]\label{def:average_k_conflict}
    For $k\ge 1$, we define the \emph{persistent $k$-conflict} $c_k(t)$ of $\theta$  as the Betti curve:  
    \[c_k(t):=\beta_k(t),\] 
and the \emph{average $k$-conflict} is then:
\begin{equation}\label{eq:average_k_conflict}
    0\le \bar c_k := \frac{1}{T-t_1} \int_{t_1}^{T}\beta_k(t)\;dt, %
\end{equation}
with $T$ as in Definition~\ref{def:average_0_conflict}.
\end{definition}
It is instructive to rewrite $\bar c_k$  in terms of the product of the total number of $k$-conflicts and their average lifetime in the filtration.  
Recall that the number of $k$-dimensional homology classes with birth time $s$ and death time $t> s$ (with $t=\infty$ included) is given by  $\mu_k^{s,t}\ge 0$, the multiplicity of the point $(s,t)$ in the $k$-dimensional PD,  $\Dgm_k(\mathcal{M})$,
and by the Fundamental Lemma of PH (Equation~\ref{eq:Fundamental_Lemma_PH}), we have $\beta_k(t)=\sum_{s\le t < t'}\mu_k^{s,t'}$. (See Appendix~\ref{background:PDs} for detailed background.) We then have the following rewriting of $\bar c_k$.

\begin{remark}\label{rem:k_conflict_count_vs_lifetime}
For $k \ge 1$, let $N_k:=\sum_{s < t}\mu_k^{s,t}$ denote the total number of $k$-conflicts, and let $\bar l_k:=\frac{1}{N_k}\sum _{s < t}\mu^{s,t}_k (\min(t,T)-s)$ denote their average lifetime truncated at $T$ (with $\bar l_k:=0$ if $N_k=0$ and $T$ as in Definition~\ref{def:average_0_conflict}).
    Then we have that
    \begin{equation}\label{eq:total_persistence}
         \int_{t_1}^{T}\beta_k(t)dt = N_k\bar l_k \quad \text{and hence} \quad \bar c_k = 
         \frac{N_k \bar l_k}{T-t_1}.
    \end{equation}
\end{remark}
\begin{remark} \label{rem:total persistence} The sum 
$\sum _{s < t}\mu^{s,t}_k (\min(t,T)-s)$ is 
    also called the \emph{degree-1 total persistence} of $\Dgm_k(\mathcal{M})$ \citep{cohen-steinerLipschitzFunctionsHave2010} with death times truncated at $T$.
\end{remark}
Equation~\ref{eq:total_persistence} indicates that, in addition to considering the average $k$-conflict, $\bar c_k$, it can be useful to look separately at the number of $k$-conflicts, $N_k$, and their average lifetime, $\bar l_k$. We will exploit this separation in Section~\ref{sec:real_single_cell}.

Next, we discuss the conditions under which we can expect to see $k$-conflicts 
in a sequence of partitions $\theta$. First, we establish that only elements $x\in X$ that lie in the intersection of non-nested clusters of $\theta$ contribute to the higher-dimensional PH of the MCF. This means that $k$-conflicts can only emerge in non-hierarchical sequences of partitions. 

Recall that two subsets $C, C'\subseteq X$ are \emph{nested} if one of the three sets $C \setminus C'$, $C' \setminus C$ or $C \cap C'$ is empty, see~\citet[Definition 2.12]{korteCombinatorialOptimizationTheory2012}.
\begin{definition}[Non-nested intersections]
    For a sequence of partitions $\theta$ %
    the \emph{set of non-nested intersections} $\bar{X}\subseteq X$ is
    \begin{equation}\label{eq:non_nested_intersections}
        \bar{X}:=\left\{x\in X\; |\; x \in C\cap C^\prime \; \text{for non-nested} \; C\in\theta(t), C^\prime\in\theta(t^\prime)  \right\}.
    \end{equation}
\end{definition}
Note that $\bar{X}=\emptyset$ if $\theta$ is hierarchical. We now define the restriction of a sequence of partitions $\theta$ and the MCF to a subset of $X$.

\begin{definition}[Restricted MCF]
    For a subset $Y\subseteq X$ we define the \emph{restricted sequence of partitions} $\theta|_Y(t)$ for $t\ge t_1$ as:
        $\theta|_Y(t):=\{C\cap Y\; |\; C \in \theta(t), C\cap Y\neq\emptyset\}$.
    We also define the \emph{restricted MCF} $\mathcal{M}|_Y=(K|_Y^t)_{t\ge t_1}$ as $\mathcal{M}|_Y:=\mathcal{M}_{\theta|_Y}$.
\end{definition}

The following proposition shows that we can restrict the MCF to the set of non-nested intersections $\bar{X}$ without changing the higher-dimensional PH.
\begin{proposition}\label{prop:restrict_true_overlaps}
    Let $\bar{X}\subseteq X$ be the set of non-nested intersections for a sequence of partitions $\theta$. 
    Then for all $k \ge 1$, $t\ge t_1$ and $p\ge 0$ we have
        $H^p_k(K^t) \cong H^p_k(K|_{\bar{X}}^t)$,
    where $\cong$ denotes group isomorphism.
\end{proposition}
\begin{proof}
    See Appendix~\ref{appx:proofs_higher_dim_ph}.
\end{proof}
\begin{remark}
    When $\theta$ has few violations of hierarchy, $\bar{X}$ is much smaller than $X$ and restricting the computation of the higher-dimensional MCF PH to the set of non-nested intersections $\bar{X}$ can be computationally beneficial. Computing $\bar{X}$ requires computing the intersections of all cluster pairs and checking the non-nested condition, which scales with the square of the number of distinct clusters.
\end{remark}

As a corollary of Proposition~\ref{prop:restrict_true_overlaps} we get that hierarchical 
sequences of partitions lead to trivial higher-dimensional PH groups. 
This also means that $k$-conflicts for $k\ge 1$ always imply 0-conflicts.
\begin{corollary}\label{cor:higher_dim_ph_0_when_hierarchical}
For all $k \ge 1$, $t\ge t_1$, $p\ge 0$ we have:
\begin{enumerate}
    \item If %
    all clusters in $\theta$ are nested (in particular, if $\theta$ is hierarchical), then $H^{t,p}_k=0$. %
    \item If $\bar c_k>0$, then $\bar c_0>0$ and $\theta$ is non-hierarchical.
\end{enumerate}
\end{corollary}
\begin{proof} 1) follows from Proposition~\ref{prop:restrict_true_overlaps}. 2) follows from 1) together with Corollary~\ref{cor:persistent_0_conflict_hierarchy}.
\end{proof}
\begin{remark}
    While a $k$-conflict for $k \ge 1$ always leads to a 0-conflict, a $(k+1)$-conflict generally does not lead to a $k$-conflict for $k\ge 1$. The reason is that $k$-conflicts are defined as $k$-dimensional non-bounding cycles and are counted by $\beta_k(t)$. In particular, if $k\ge 1$, then $\beta_{k+1}(t) > 0$ does not generally imply $\beta_{k}(t) >0$ for the MCF $K^t$.
\end{remark}

For non-hierarchical sequences of partitions, the $k$-conflicts not only live on the set of non-nested intersections $\bar X$ but also emerge from particular intersection patterns between non-nested clusters, as stated in the next proposition. 
\begin{proposition}\label{prop:kconflicts-intersections}
    If $\beta_k(t) > 0$ for some $t\ge t_1$ and $k\ge 1$, then there exist at least $k+1$ mutually non-nested clusters $C_0,\dots,C_{k}$ in the partitions  $\theta([t_1,t])$ 
    such that 
        $\bigcap_{i =0}^{k} C_i \neq \emptyset$.
\end{proposition}
\begin{proof}
See Appendix~\ref{appx:proofs_higher_dim_ph}.
\end{proof}
\begin{remark}\label{rem:conflict_parties}
Proposition~\ref{prop:kconflicts-intersections} shows that the minimum number of clusters involved in a $k$-conflict increases with the dimension $k$. Therefore, higher order $k$-conflicts require intersection patterns involving more ``conflict parties'' (i.e., clusters from different partitions). %
Future work will aim to further characterise the intersection patterns leading to $k$-conflicts.
\end{remark}

\subsection{Scale Selection within the MCF} \label{sec:mcf_scale_selection}

When applying multiscale clustering analysis, an important computational task in various domains is to capture intrinsic, robust scales in the data~\citep{azranSpectralMethodsAutomatic2006,schindlerMultiscaleMobilityPatterns2023,arefRobustMarkovStability2025}. 
In  hierarchical clustering, this corresponds to finding the optimal (long-lived) cuts in a dendrogram~\citep{jainDataClusteringUsers2005}. %
In this section, we describe how the MCF can be used for the selection of robust scales in a non-hierarchical sequence of partitions. 
Specifically, within MCF a robust scale is signalled by a $t^*$ that both: (i) resolves many %
cluster assignment conflicts, and (ii) is followed by a long interval over which the homology of the MCF does not change. We now make explicit these two ingredients.

To make the first ingredient (conflict resolution) more precise, let us consider again the multiplicities $\mu_k^{s,t}\ge 0$ of points $(s,t)$ in the $k$-dimensional MCF PD %
(see Appendix~\ref{background:PDs}). Noting that the only scales where the MCF can change  
are the change points $t_m$, we define:
\begin{itemize}
\item the \emph{number of births} of $k$-dimensional homology features: $b_k(t_m):=\sum_{\ell=m+1}^M\mu_k^{t_m,t_\ell}+\mu_k^{t_m,\infty}\ge0$ 
\item the \emph{number of deaths} of $k$-dimensional homology features: $d_k(t_m):=\sum_{\ell=1}^{m-1}\mu_k^{t_\ell,t_m}\ge 0$. 
\end{itemize}
For dimension $k=0$, the homology features are related to the number of connected components; hence, $b_0(t_1)=\lvert\theta(t_1)\rvert$ and $b_0(t_m)=0$ for $m\ge 2$, so that $d_0(t_m)$ is the decrease in the number of connected components of the MCF at $t_m$. 
For $k\ge 1$, $b_k(t_m)$ and $d_k(t_m)$ are, respectively, the numbers of $k$-conflicts created and resolved by partition $\theta (t_m)$.
A partition can %
both create and resolve $k$-conflicts
but, intuitively, a `good partition' resolves many $k$-conflicts while creating few new $k$-conflicts. This motivates our next definition.
\begin{definition}[Net $k$-conflict-resolving partition]
    Let $2\le m \le M$. We say that partition $\theta(t_m)$ is \emph{net 0-conflict-resolving} if $d_0(t_m)<\lvert\theta(t_{m-1})\rvert-\lvert\theta(t_{m})\rvert$, i.e., if the number of connected components of the MCF decreases by less than the number of clusters. For $k\ge 1$, we say that $\theta(t_m)$ is \emph{net k-conflict-resolving} if $d_k(t_m)>b_k(t_m)$, i.e., if it resolves more $k$-conflicts than it creates.
\end{definition}
We now show that net $k$-conflict resolution corresponds to a decrease in the persistent $k$-conflict curve (Definitions~\ref{def:persistent_0_conflict} and \ref{def:average_k_conflict}).
\begin{proposition}\label{prop:c_k_conflict_resolution}
    For all $k\ge 0$ and $2\le m\le M$, $\theta(t_m)$ is net $k$-conflict-resolving iff $c_k(t_m)<c_k(t_{m-1})$. More precisely, we have
    \begin{align*}
        c_0(t_m)-c_0(t_{m-1})&=d_0(t_m)-\left (\lvert\theta(t_{m-1})\rvert - \lvert\theta(t_{m})\rvert\right) < 0 ,\\
        c_k(t_m)-c_k(t_{m-1})&=b_k(t_m)-d_k(t_m) < 0, \quad k\ge 1.
    \end{align*}
\end{proposition}
\begin{proof}
     The proof follows directly from Definitions~\ref{def:persistent_0_conflict} and \ref{def:average_k_conflict} and the Fundamental Lemma of PH (Equation~\ref{eq:Fundamental_Lemma_PH}).
\end{proof}
The proposition characterises net conflict resolution by a negative discrete derivative of the persistent conflict curves $c_k$, which are defined in terms of the Betti curves of the MCF. 
Discrete derivatives of Betti curves have previously been used for time-series prediction with the Vietoris-Rips filtration~\citep{abayChainNetLearningBlockchain2019} but, to our knowledge, not for scale selection.
Therefore, as a first ingredient of our scale selection heuristic, we look for change points at which the persistent conflict curves decrease sharply, signalling conflict resolution.

As a second ingredient, we look for scales over which the homology of the MCF remains unchanged by evaluating its \emph{homological critical values}~\citep{cohen-steinerStabilityPersistenceDiagrams2005}.
\begin{definition}[Critical values]\label{def:critical_values_cdf}
    A change point $t_m$ of $\theta$ is a (homological) \emph{critical value} of the MCF %
    if $b_k(t_m)+d_k(t_m)>0$ for some $k \ge 0$. We denote the \emph{set of critical values} as $\mathrm{Crit}(\theta):=\{t_\ell\in\{t_1,\dots,t_M\} \; | \; t_\ell \; \text{is a critical value}\}$, and define the \emph{normalised cumulative count of critical values} 
    \begin{equation}\label{eq:critical_values_cdf}
        0 < \widehat{F}_\mathrm{Crit}(\theta; t) :=\frac{|\mathrm{Crit}(\theta)\cap [t_1,t]|}{|\mathrm{Crit}(\theta)|} \le 1, \quad t\ge t_1.
    \end{equation}
\end{definition}
 Note that $t_1$ is always a critical value because $b_0(t_1)>0$; hence, $1\le |\mathrm{Crit}(\theta)|\le M$  
 and the function $\widehat{F}_\mathrm{Crit}(\theta; t)$ is non-decreasing and right-continuous in $t$. We write $\widehat{F}_\mathrm{Crit}(t)$ if $\theta$ is clear from the context.
\begin{remark}\label{rem:properties_F_crit}
    From the Critical Value Lemma~\citep{cohen-steinerStabilityPersistenceDiagrams2005}, we have 
    $\widehat{F}_\mathrm{Crit}(t)=\widehat{F}_\mathrm{Crit}(t')$ if and only if the persistence module of the MCF is constant on $[t,t']$, i.e., for all $k\ge 0$ and  $t\le s \le s' \le t'$, the map $H_k(K^s)\to H_k(K^{s'})$ induced by the canonical inclusion is an isomorphism. 
\end{remark}
We thus look for %
long plateaux of $\widehat F_\mathrm{Crit}$, over which the MCF homology does not change. (See \citet{berryConsistentManifoldRepresentation2019} and \citet{lambiotteRandomWalksMarkov2014} for related ideas of robustness.)
Note that these plateaux also correspond to gaps 
in the MCF PD along both the birth- and death-axes for all $k\ge 0$. We now summarise our heuristic for scale selection. 

\begin{remark}[Heuristic for scale selection within MCF]\label{rem:heuristic_gaps}
To select robust scales in (non-hierarchical) sequences of partitions, we look for scales $t^* = t_m$, $m \geq 2$, such that: (i) none of the persistent conflict curves $c_k$, $k\ge 0$, increases at $t^*$, and $t^*$ is located at or shortly after a sharp decrease of at least one of them, signalling the resolution of conflicts (Proposition~\ref{prop:c_k_conflict_resolution}); and
(ii) $t^*$ initiates a long plateau of $\widehat{F}_\mathrm{Crit}$ (which also corresponds to gaps along both the birth- and death-axes of the MCF PDs),
such that the topological structure of the MCF remains unchanged over the plateau. %

\end{remark}

In the hierarchical case, this scale selection heuristic reduces to a well-known criterion for deciding the optimal number of clusters in hierarchical clustering.

\begin{remark}[Scale selection for hierarchical sequences]
\label{rem:lifetime-criterion}
For hierarchical sequences of partitions, our scale selection heuristic (Remark~\ref{rem:heuristic_gaps}) reduces to the \emph{lifetime criterion}, which selects the partitions that persist over the largest range of the dendrogram depth $t$~\citep{jainDataClusteringUsers2005}. Indeed, if $\theta$ is hierarchical, then $c_k(t) = 0$ for all $k \geq 0$ by Corollaries~\ref{cor:persistent_0_conflict_hierarchy} and \ref{cor:higher_dim_ph_0_when_hierarchical}. 
Hence, condition (i) imposes no constraint, as there are no conflicts to resolve.
Moreover, as $\beta_0(t) = |\theta(t)|$ (Proposition~\ref{prop:0-dim-Betti}), the
plateaux of $\widehat{F}_\mathrm{Crit}$ coincide with the lifetimes of the partitions, so selecting
the longest plateaux in (ii) is equivalent to the lifetime criterion.
\end{remark}

Below, we illustrate our heuristic for scale selection with numerical experiments on synthetic data with ground truth (Section~\ref{sec:sbms}) and on real-world single-cell data (Section~\ref{sec:real_single_cell}).

\section{Mathematical Links of MCF to Other Filtrations}\label{sec:alternative_filtrations}

\subsection{Equivalent Nerve Complex Construction of the MCF}\label{Sec:NerveComplex}

We define here an equivalent nerve complex construction of the MCF, inspired by the MAPPER construction~\citep{singhTopologicalMethodsAnalysis2007}. For background on nerve complexes see Appendix~\ref{background:complex}. Recall that $\theta$ is piecewise-constant with $M$ change points $t_1<t_2<\dots<t_M$.
\begin{definition}[Nerve-based MCF]\label{def:nerve-based MCF}
    Let $\mathcal{C}(m):=(C_\alpha)_{\alpha\in A(m)}$, $1\le m\le M$, be the family of clusters indexed over the multi-index set 
    $A(m):=\{(\ell,i)\; |\; 1 \le \ell \le m,\; i \le \lvert\theta(t_\ell)\rvert \}$
    such that $C_{(\ell,i)}$ is the $i$-th cluster in partition $\theta(t_\ell)$. 
    Then we define the \emph{nerve-based MCF} $\tilde{\mathcal{M}}_\theta:=(\tilde K^t)_{t \ge t_1}$ as a filtration of nerve complexes
    \begin{align*}
        \tilde K^t := \mathrm{Nrv}(\mathcal{C}(m))=
        \left\{ \emptyset\neq S\subseteq A(m) : \bigcap_{\alpha\in S}C_\alpha \neq \emptyset \right\},
    \end{align*}
    where $t\in[t_m,t_{m+1})$, $m=1,\dots,M-1$ or $t\ge t_m$ for $m=M$.
\end{definition}
The abstract simplicial complex $\tilde K^t$ records the intersection patterns of all clusters in $\theta$ up to scale $t$. %
Hence, the nerve-based MCF $\tilde{\mathcal{M}}_\theta$ provides a complementary perspective to the MCF $\mathcal{M}_\theta$: while the vertices of $\mathcal{M}_\theta$ correspond to elements in $X$ with homology generators indicating which elements contribute to conflicts, the vertices of $\tilde{\mathcal{M}}_\theta$ correspond to clusters in $\theta$ and the homology generators inform us about which clusters lead to a conflict. 
Again, if $\theta$ is clear from the context we write $\tilde{\mathcal{M}}$ for simplicity. 

\begin{figure}[!htb]
    \centering
    \includegraphics[width=\textwidth]{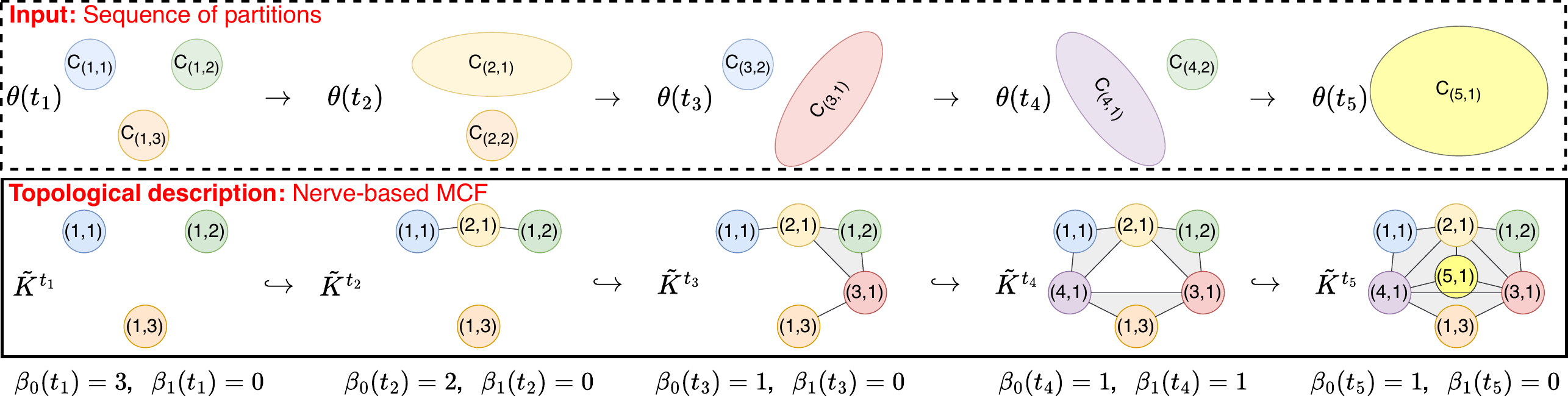}
    \vspace{-0.7cm}
    \caption{%
    Nerve-based MCF construction for our running example (Example~\ref{ex:MCF_illustration}). Top row:  %
    non-hierarchical sequence of partitions $\theta$ %
    with cluster indices as in Definition~\ref{def:nerve-based MCF}. Bottom row: nerve-based MCF $(\tilde K^t)_{t\ge t_1}$, where we simplify the visualisation without changing the homology (i.e., we omit vertices $(2,2)$, $(3,2)$ and $(4,2)$ of singleton clusters, and for $\tilde K^{t_5}$ we 
    omit 3-simplices). %
Because both the MCF and the nerve-based MCF have the same PH (Proposition~\ref{prop:equiv_MCF_nerve-based MCF}), the Betti numbers match those in Figure~\ref{fig:MCF_toy_example}. %
    }
    \label{fig:nerve-based MCF_illustration}
    \vspace{-0.3cm}
\end{figure}

We illustrate the nerve-based MCF in
Figure~\ref{fig:nerve-based MCF_illustration} using our running Example~\ref{ex:MCF_illustration}.
It is no coincidence that the nerve-based MCF construction leads to the same Betti numbers as the MCF in Figure~\ref{fig:MCF_toy_example}. In fact, we can prove that both filtrations always lead to the same PH.

\begin{proposition}\label{prop:equiv_MCF_nerve-based MCF}
    Consider the MCF $\mathcal{M}=(K^t)_{t\ge t_1}$ and the nerve-based MCF $\tilde{\mathcal{M}}=(\tilde K^t)_{t \ge t_1}$ for a sequence of partitions $\theta$. Then   $H_k^p(\tilde K^t)\cong H_k^p(K^t)$  for any $k\ge 0$, $t\ge t_1$ and $p\ge 0$.%
\end{proposition}
\begin{proof}
    The idea of the proof is to use a functorial nerve theorem for simplicial covers~\citep[Theorem 4.8]{bauerUnifiedViewFunctorial2023}. See Appendix~\ref{appx:proofs_nerve} for a full proof.
\end{proof}
Proposition~\ref{prop:equiv_MCF_nerve-based MCF} shows that the element-centric perspective of the MCF and the cluster-centric perspective of the nerve-based MCF are essentially equivalent. We next observe that the dimension of the nerve-based MCF depends only on the number of change points of $\theta$. This has consequences for computational efficiency.
\begin{proposition}\label{prop:dim_nerve_mcf}
    For $t \in [t_m, t_{m+1})$, $m = 1, \dots, M-1$, we have
    $\dim \tilde K^t = m - 1$, and for $t \ge t_M$ we have
    $\dim \tilde K^t = M - 1$.
\end{proposition}
\begin{proof}
    Due to Proposition~\ref{prop:MCF_is_filtration}, it suffices to show $\dim \tilde K^{t_m} = m - 1$ for $m = 1, \dots, M$.
    For each $\ell \le m$ the element $x\in X$ is contained in exactly one cluster $C_{(\ell, i_\ell)} \in \theta(t_\ell)$, and so
    $\sigma_x := \{(\ell, i_\ell) \mid \ell \le m\}\subseteq A(m)$ fulfils $x \in \bigcap_{\alpha \in \sigma_x} C_\alpha \neq \emptyset$, which
    yields $\sigma_x \in \tilde K^{t_m}$ with $\dim \sigma_x = m - 1$. Conversely, the clusters of each partition $\theta(t_\ell)$ are pairwise
    disjoint, and so any $\sigma \in \tilde K^{t_m}$, whose clusters share a common element, contains at most one vertex $(\ell, i)$ per change point
    $\ell \le m$, which yields $\dim \sigma \le m - 1$. Hence, $\dim \tilde K^{t_m}=m-1$.
\end{proof}
\begin{remark}\label{rem:computational_complexity}
    From Proposition~\ref{prop:dim_nerve_mcf} we have that $\dim \tilde K=M-1$ for $\tilde K:=\tilde K^{t_M}$,  and from Remark~\ref{rem:mcf_dimension} we have that $\dim K = \max_{t \ge t_1, \,C\in\theta(t)}\lvert C\rvert-1$. Hence, it is computationally beneficial to use the nerve-based MCF when $M$ is smaller than the size of the largest cluster in $\theta$, whereas the standard MCF is computationally advantageous in the opposite regime. We quantify this trade-off numerically in Section~\ref{sec:computational_efficiency}.
\end{remark}

\subsection{MCF in the Special Case of Hierarchical Clustering}\label{sec:CAG_hierarchical}

It is illustrative to examine the MCF for the case of a hierarchical sequence of partitions $\theta$ with $\theta(t_1)=\hat 0$ and $\theta(t_M)=\hat 1$, and redefine $t_1=0$ without loss of generality.
Let us consider the matrix of pairwise first-merge times $D_\theta:=D_{\theta,1}$ defined in Equation~\eqref{eq:joint-first-merge-times}. To simplify our notation, if $\theta$ is clear from the context, we write $D$ instead of $D_\theta$.

Recall that the element $D_{xy}$ of matrix $D$ is the first scale $t$ at which two elements $x,y\in X$ become part of the same cluster in $\theta$. Therefore, the matrix $D$ is symmetric with $t_1$ on the diagonal, and fulfils the properties of a \emph{dissimilarity measure} ($D_{xx}\le D_{xy}=D_{yx}$, $\forall x,y\in X$)~\cite[Section 4.2.5]{chazalPersistenceStabilityGeometric2014}. 
In the case of hierarchical clustering, merged elements never separate and $D$ fulfils the strong triangle inequality ($D_{xz}\le\max(D_{xy},D_{yz}), \, \forall x,y,z\in X$) %
so that $D$ defines an ultrametric on $X$~\citep[Lemma 10]{carlssonCharacterizationStabilityConvergence2010}. In the case of a non-hierarchical sequence, however, $D$ does not even fulfil the standard triangle inequality in general because merged elements can separate again. 

Let us define a simplicial complex $L^t$, $t\ge t_1$, as the clique complex of the thresholded \emph{cluster assignment graph} (CAG) $G_t:=(X,E_t)$ with undirected, unweighted edges $E_t:=\left\{\{x,y\}, x,y\in X\; |\; D_{xy}\le t\right\}$, see Equation~\eqref{eq:clique_filtration}. Equivalently, $L^t$ is defined as the VR complex at scale $t$ of the dissimilarity space $(X,D_\theta)$, see Equation~\ref{eq:VR_point_cloud}; hence, the filtration $\mathcal{L}=(L^t)_{t\ge t_1}$ guarantees the stability of its PDs~\cite[Section 4.2.5]{chazalPersistenceStabilityGeometric2014}. However, note that $\mathcal{L}$ is 2-determined, in contrast to the MCF (Remark~\ref{rem:2-determined}), and not equivalent to the MCF $\mathcal{M}$ in general. This is illustrated by our running example.
\begin{example}%
\label{ex:running_clique_complex}
    For Example~\ref{ex:MCF_illustration}, the matrix of pairwise first-merge times is given by 
    \[D_\theta=\begin{pmatrix}0&1&3\\1&0&2\\3&2&0\end{pmatrix}.
    \]
    Hence, $K^t=L^t$ for $t\le 2$ but $K^3 \neq L^3 = 2^X\setminus\{\emptyset\}$ because $L^3$ is the clique complex of $G_3$, which is 
    the clique of three nodes. In particular, the 2-simplex $[x_1,x_2,x_3]$ is already contained in $L^3$ and so $H_1(L^3)=0$ whereas $H_1(K^3)=\mathbbm{Z}_2$. %
    Note that the joint first-merge times $D_{\theta,2}$ of triplets in $X$ (Equation~\ref{eq:joint-first-merge-times}) cannot be derived from the pairwise first-merge times $D_\theta$. 
\end{example}

Our running example suggests that $\mathcal{L}$ is less sensitive to higher-dimensional conflicts, but it is shown below that $\mathcal{L}$ has the same 0-dimensional PH as $\mathcal{M}$. This implies that we can compute the persistent 0-conflict $c_0(t)$ of the MCF (Equation~\ref{eq:persistent_hierachy}) also from $\mathcal{L}$.
\begin{proposition}
    For $t \ge t_1$ and $p\ge 0$ we have $H_0^p(L^t)\cong H_0^p(K^t)$ but $H_k^p(L^t) \cong H_k^p(K^t)$ does not hold for $k \ge 1$ in general.
\end{proposition}
\begin{proof}
    The 1-skeletons of $L^t$ and $K^t$ coincide and so $H_0^p(L^t)\cong H_0^p(K^t)$. %
    The second statement follows from 
    Example~\ref{ex:running_clique_complex}.
\end{proof}
However, if $\theta$ is hierarchical, $\mathcal{L}$ has the same PH as the MCF $\mathcal{M}$ because the two filtrations coincide in this case.
\begin{corollary}\label{cor:hierarchical_cag_equivalent}
    If the sequence of partitions $\theta$ is hierarchical, then $K^t=L^t$ for all $t\ge t_1$. %
\end{corollary}
\begin{proof}
    Since $\theta$ is hierarchical we have $\theta(s)\le\theta(t)$ for all $s\le t$ and so $K^t=\bigcup_{C\in\theta(t)}\Delta C$. Furthermore, $D_{xy}\le t$ iff $x\sim_{\theta(t)} y$ and so  $G_t$ is the disjoint union of cliques on the clusters of $\theta(t)$, and therefore $L^t=\bigcup_{C\in\theta(t)}\Delta C$.
\end{proof}
Therefore, if $\theta$ is hierarchical, the MCF is equivalent to the VR filtration of the ultrametric space $(X,D_\theta)$.

If we further assume that the hierarchical sequence of partitions was obtained from a finite metric space $(X,d_X)$ using single-linkage hierarchical clustering, we can use a stability theorem by \citet{carlssonCharacterizationStabilityConvergence2010} to relate the (only non-trivial) 0-dimensional PD of the MCF of the hierarchical $\theta$ to the underlying ultrametric space as follows. 
\begin{corollary}\label{cor:MCF_SLHC}
    Let $\theta\colon[0,\infty)\rightarrow \Pi_X$ and $\eta:[0,\infty)\rightarrow \Pi_Y$ be two sequences of partitions obtained with single-linkage clustering from the finite metric spaces $(X,d_X)$ and $(Y,d_Y)$, respectively. Then we obtain the following inequalities for the corresponding MCFs $\mathcal{M}_\theta$ and $\mathcal{M}_\eta$ and ultrametrics $D_\theta$ and $D_\eta$:
    \begin{equation*}
        \frac{1}{2}\,        d_{W,\infty}\left(\Dgm_0(\mathcal{M}_\theta),\Dgm_0(\mathcal{M}_\eta)\right) \; \le\; d_\GH ((X,D_\theta),(Y,D_\eta)) \; \le\; d_\GH((X,d_X),(Y,d_Y)),
    \end{equation*}
    with $d_{W,\infty}$ the bottleneck distance 
    (Equation~\ref{eq:bottleneck-distance}) 
    and $d_\GH$  the Gromov-Hausdorff distance.
\end{corollary}
\begin{proof}
    See Appendix~\ref{appx:proofs_hierarchical}.
\end{proof}

\subsection{%
Zigzag MCF Construction and its relation to the Sankey Diagram}\label{sec:zigzag_mcf}

Inspired by more restrictive zigzag filtrations proposed for the analysis of temporal data~\citep{kimAnalysisDynamicGraphs2020}, we next construct an MCF zigzag filtration $\mathcal{Z}_\theta$ of the
sequence of partitions $\theta$, and compare it to the full MCF. See Appendix~\ref{mcf_sec:background_zigzag} for background on zigzag %
PH. 
\begin{definition}[Zigzag MCF]
    Let $\theta$ be a non-hierarchical sequence of partitions with change points $t_1<\dots < t_M$. For each $m=1,\dots,M$, we define $Z^m_\theta:=\bigcup_{C\in\theta(t_m)}\Delta C$ and for $m<M$,  $Z^{m,m+1}_\theta:=Z^m_\theta\cup Z^{m+1}_\theta$ so that the \emph{zigzag MCF} $\mathcal{Z}_\theta$ is given by the diagram%
    \begin{equation}\label{eq:zigzag_mcf}
        Z^1_\theta \hookrightarrow Z^{1,2}_\theta \hookleftarrow Z^2_\theta \hookrightarrow \dots \hookleftarrow Z^M_\theta,
    \end{equation}
    where the arrows indicate inclusion maps.
\end{definition}
Therefore, for each consecutive pair of partitions $\theta(t_m)$ and $\theta(t_{m+1})$ in $\theta$, the zigzag MCF $\mathcal{Z}_\theta$ constructs one intermediate simplicial complex $Z_\theta^{m,m+1}$.
Note that if $\theta(t_m)$ and $\theta(t_{m+1})$ are comparable with respect to the refinement order, then $Z_\theta^{m,m+1}=Z_\theta^{m+1}$ if $\theta(t_m)\le \theta(t_{m+1})$, and $Z_\theta^{m,m+1}=Z_\theta^{m}$ if $\theta(t_{m+1})\le \theta(t_{m})$. 
As above, if $\theta$ is clear from the context, we will drop the subindex and write $\mathcal{Z}$, $Z^m$ and $Z^{m,m+1}$ for notational simplicity.

The PH of $\mathcal{Z}$ tracks the merging and splitting patterns of connected components at $k=0$ and the emergence of 1-conflicts for every pair of \emph{consecutive} partitions at $k=1$, and is summarised by the zigzag PDs (Appendix~\ref{mcf_sec:background_zigzag}).
However, the zigzag MCF $\mathcal{Z}$ 
cannot account for the topological structure captured by the full MCF $\mathcal{M}$. 
To make this more explicit, we first show that the homology groups of $\mathcal{Z}$ can be computed directly from the Sankey diagram $S(\theta)$ and that the zigzag PH is trivial for dimensions $k\ge 2$. %
\begin{proposition}\label{prop:mcf_zigzag_homology}
    For a sequence of partitions $\theta$, consider its Sankey diagram $S(\theta)$ and its zigzag MCF $\mathcal{Z}_\theta$. Let $S_m:=(V_m\uplus V_{m+1},E_m)$, $m=1,\dots, M-1$, be the bipartite graph that describes the $m$-th layer of $S(\theta)$. Here $V_m$ and $V_{m+1}$ are sets of vertices and each vertex is a cluster in, respectively, partitions $\theta(t_m)$ and $\theta(t_{m+1})$ connected in $S(\theta)$ by the edges in the set $E_m$. 
    Interpreting $S_m$ as a 1-dimensional simplicial complex, we then have:
        \begin{enumerate}[1.]
            \item for $k=0$ : $\dim H_0(Z^m)=\lvert V_m\rvert =\lvert \theta(t_m)\rvert$ and $H_0(Z^{m,m+1})\cong H_0(S_m)$,
            \item for $k=1$ : $H_1(Z^m)=0$ and  $H_1(Z^{m,m+1}) \cong H_1(S_m)$,
            \item for $k\ge 2$ :  $H_k(Z^m)=H_k(Z^{m,m+1})=0$. 
        \end{enumerate}
\end{proposition}
\begin{proof}
    See Appendix~\ref{appx:proofs_zigzag}.
\end{proof}
\begin{remark}\label{rem:zigzag_barcode_interpretation} The zigzag MCF construction is directly related to
graph-theoretic invariants of the Sankey bipartite graphs $S_m$:
\begin{itemize}
\item
    For $k=0$: 
    In the 0-dimensional %
    zigzag PD $\Dgm_0(\mathcal{Z})$,  
    the number of %
    points (counted with multiplicity) that are alive at the position $Z^m$  is 
    $\dim H_0(Z^m)=\lvert \theta(t_m)\rvert$, and that are alive at the position %
    $Z^{m,m+1}$ %
    is $\dim H_0(S_m)$, which corresponds to the number of connected components of the graph $S_m$. %
    Moreover, 
    $\dim H_0(S_m) = 
    \lvert \theta(t_m) \vee \theta(t_{m+1}) \rvert =
    \lvert \theta(t_m) \rvert + \lvert \theta(t_{m+1}) \rvert - \rank L_m$  
    , where $\vee$ indicates the join (see Remark~\ref{rem:betti_0_join}) and $L_m$ is the Laplacian of the graph $S_m$  
    ~\citep{chungSpectralGraphTheory1997}.
\item
    For $k=1$: 
    The 1-dimensional %
    zigzag PD $\Dgm_1(\mathcal{Z})$ consists of the points $(Z^{m,m+1},Z^{m+1})$, with multiplicity $\dim H_1(S_m)$ equal to the \emph{cyclomatic number} of $S_m$~\citep{bergeTheoryGraphs2001}, i.e., the number of independent cycles within $S_m$, which is given by
    $\dim H_1(S_m)=
    \lvert E_m\rvert-\lvert \theta(t_m) \rvert - \lvert \theta(t_{m+1}) \rvert +\dim H_0(S_m) 
    = \lvert \theta(t_m) \wedge \theta(t_{m+1}) \rvert - \rank L_m $, where $L_m$ is the graph Laplacian of $S_m$ and $\wedge$ stands for the meet in the partition lattice~\citep{birkhoffLatticeTheory1967}, i.e., the coarsest partition that is finer than the two partitions.
    \end{itemize}
\end{remark}
We remark that, while $\mathcal{Z}$ has non-trivial PH for $k=1$, $\mathcal{Z}$ is restricted to consecutive partitions in the sequence and only $\mathcal{M}$ can detect 1-conflicts that arise across more than two partitions. We illustrate this in the next example.
\begin{example}%
\label{ex:zigzag_mcf}
Let $X=\{x_1,x_2,x_3\}$ and consider two different sequences of partitions $\theta, \eta: [0,\infty)\rightarrow\Pi_X$ with five change points $t_1,\dots,t_5$ each. We define $\theta$ as in  Example~\ref{ex:MCF_illustration} and $\eta$ such that %
$\eta(t_m):=\theta(t_m)$ for $m=1,2,3,5$ but 
$\eta(t_4):=\theta(t_2)$. See Figure~\ref{fig:zigzag_barcode} for an illustration. Note that, although $\theta$ and $\eta$ differ at change point $t_4$, both sequences lead to the same Sankey diagram, which only depends on the relation between consecutive partitions.
Using the MCF, we can distinguish both sequences of partitions: a 1-conflict emerges in $\theta$ at $t_4$, whereas $\eta$ has no 1-conflict because the 1-dimensional PH of $\mathcal{M}_\eta$ is trivial. In contrast, $\theta$ and $\eta$ cannot be distinguished using the zigzag MCF. In particular, the 0-dimensional zigzag persistence modules of $\mathcal{Z}_\theta$ and $\mathcal{Z}_\eta$ are equivalent as the split and merge patterns between consecutive partitions look the same locally; hence, they have the same 0-dimensional %
PD (see Figure~\ref{fig:zigzag_barcode}). Furthermore, the 1-dimensional zigzag persistence of both $\mathcal{Z}_\theta$ and $\mathcal{Z}_\eta$ is trivial by Proposition~\ref{prop:mcf_zigzag_homology} because no layer in the Sankey diagram has a cycle, and instead
the 1-conflict in $\theta$ emerges from the interplay of three partitions. %
This demonstrates that the zigzag MCF fails to detect 1-conflicts that arise across more than two consecutive partitions. 
\end{example}

\begin{figure}[htb!]
    \centering
    \includegraphics[width=1\linewidth]{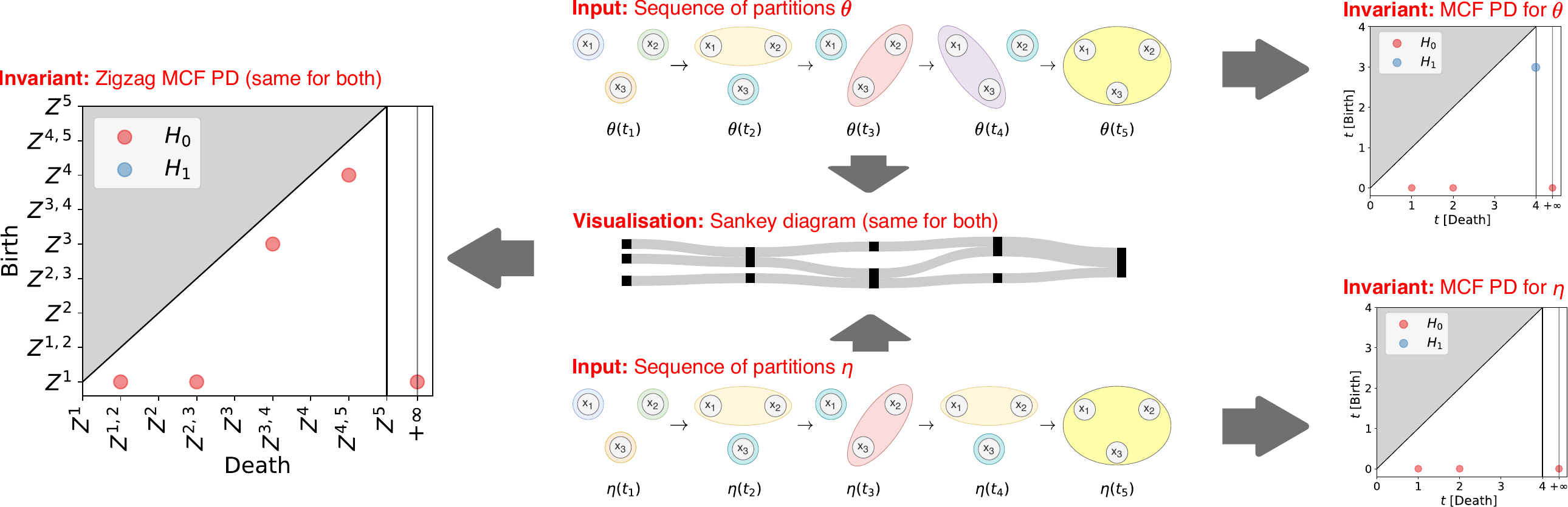}
    \vspace{-0.7cm}
    \caption{The sequence of partitions $\eta$  from Example~\ref{ex:zigzag_mcf} deviates at $t_4$ from $\theta$ in Example~\ref{ex:MCF_illustration}, yet they both have the same Sankey diagram. Consistent with Proposition~\ref{prop:zigzag_sankey}, the 0-dimensional PDs for the zigzag filtrations $\mathcal{Z}_\eta$ and $\mathcal{Z}_\theta$ are the same and the 1-dimensional PDs are empty. Hence, zigzag PDs cannot distinguish the two sequences. In contrast, the MCF PDs of $\theta$ %
    and $\eta$ differ in dimension 1. 
    }
    \label{fig:zigzag_barcode}
\end{figure}

The observation that $\theta$ and $\eta$ in Example~\ref{ex:zigzag_mcf} have the same Sankey diagram, and also the same zigzag PH, points us to a more general result, namely, that an isomorphism between Sankey diagrams (see %
Appendix~\ref{appx:Sankey} for a precise definition) leads to isomorphic zigzag persistence modules, even for sequences of partitions defined on different sets.
\begin{proposition}\label{prop:zigzag_sankey}
    Let $\theta \colon [t_1, \infty) \to \Pi_X$ and $\eta\colon [t_1, \infty) \to \Pi_Y$ be two sequences of
    partitions with change points $t_1, \dots, t_M$ whose Sankey diagrams
    $S(\theta)$ and $S(\eta)$ are isomorphic. Then the zigzag
    persistence modules of $\mathcal Z_\theta$ and
    $\mathcal Z_\eta$ are isomorphic in every dimension $k\ge 0$.
\end{proposition}
\begin{proof}
    See Appendix~\ref{appx:proofs_zigzag}.
\end{proof}
Proposition~\ref{prop:zigzag_sankey} shows that the zigzag PH of $\mathcal Z_\theta$ boils down to an invariant of the Sankey diagram $S(\theta)$ but does not contain more information. In contrast, the MCF goes beyond the Sankey diagram by aggregating the information from the full sequence of partitions, so it can also detect higher-order cluster inconsistencies arising across more than two partitions.

\section{Numerical Experiments}\label{sec:NumericalExperiments}

We demonstrate the applicability of MCF to multiscale clusterings from different domain areas including stochastic block models, single-cell data, and temporal social interaction data. Code for all numerical experiments and an implementation of the MCF with \texttt{GUDHI}~\citep{mariaGudhiLibrarySimplicial2014} are available at: \url{https://github.com/barahona-research-group/MCF}. All compute times reported in this section are based on a single CPU core.

\subsection{Application to Stochastic Block Models
}\label{sec:sbms}

In our first experiment, we apply MCF to multiscale clusterings generated from stochastic block models (SBMs). SBMs have become an important modelling tool in different areas of data science for the representation and inference of datasets with multiple levels of clustering~\citep{decelleAsymptoticAnalysisStochastic2011,schaubHierarchicalCommunityStructure2023}. %
In our case, SBMs also provide a controlled computational example to showcase the applicability of MCF and to calibrate its performance on datasets with different multiscale structure, hierarchical and non-hierarchical.

We consider SBMs defined on $N=270$ elements with  2,500 pairwise random links drawn from four types of planted multiscale structure:
(i) an Erd\H{o}s-Rényi (ER)  model~\citep{erdosRandomGraphs1959,bollobasRandomGraphs2011}, i.e., an SBM with a single block with no hierarchy and no planted partition such that any pair of elements is linked with the same probability; %
(ii) a single-scale SBM (sSBM)  with a planted partition into three equal-sized blocks of 90 elements with higher probability of in-block connections than across blocks~\citep{hollandStochasticBlockmodelsFirst1983, karrerStochasticBlockmodelsCommunity2011}; (iii) a hierarchical multiscale SBM (mSBM) with a hierarchical structure of partitions nested at three scales with 27, nine and three equal-sized blocks, respectively~\citep{%
schaubHierarchicalCommunityStructure2023}; and (iv) a non-hierarchical multiscale SBM (nh-mSBM) with planted partitions at three scales with 27, five and three equal-sized blocks, respectively, yet not nested across scales. 
 We then generate 200 random realisations from each SBM, which are represented by the set of $N\times N$ matrices $A_i$, $i=1,\dots,800$. 
See Appendix~\ref{appx:SBM_experiment} for details and all parameters.

To each realisation $A_i$, we then apply a multiscale community detection algorithm (\emph{Markov Stability} (MS)~\citep{delvenneStabilityGraphCommunities2010, lambiotteRandomWalksMarkov2014,
arnaudonAlgorithm1044PyGenStability2024} with Louvain optimisation ~\citep{blondelFastUnfoldingCommunities2008})  to obtain a multiscale, non-hierarchical sequence of partitions, $\theta_i$, where 
$\theta_i(t_1=-1.5)=\hat 0$ and the partitions get coarser as the scale $t$ is increased. 
Figure~\ref{fig:sbms_summary} shows examples of $A_i$ and the corresponding $\theta_i$ (Sankey visualisation)  for one realisation of each SBM. 
For each $\theta_i$,  %
we then construct the MCF $\mathcal{M}_i=(K_i^t)_{t\ge t_1}$, and we compute the PH for dimensions $k\le 2$. The average compute time for the nerve-based MCF is $<8$s for each realisation (see Table~\ref{tab:timing_results_ms} in Appendix~\ref{appx:sbm_compute_times}). 
The %
PDs of the 200 realisations for each SBM are shown in the third column of Figure~\ref{fig:sbms_summary}.
Note that the particular choice of multiscale community detection algorithm is unimportant; the analysis is fully consistent when computing the $\theta_i$ using Modularity with a resolution parameter~\citep{reichardtStatisticalMechanicsCommunity2006} and Leiden optimisation~\citep{traagLouvainLeidenGuaranteeing2019} (see   Appendix~\ref{appx:sbm_modularity}).

\begin{figure}[!htb]
    \centering
    \includegraphics[width=\textwidth]{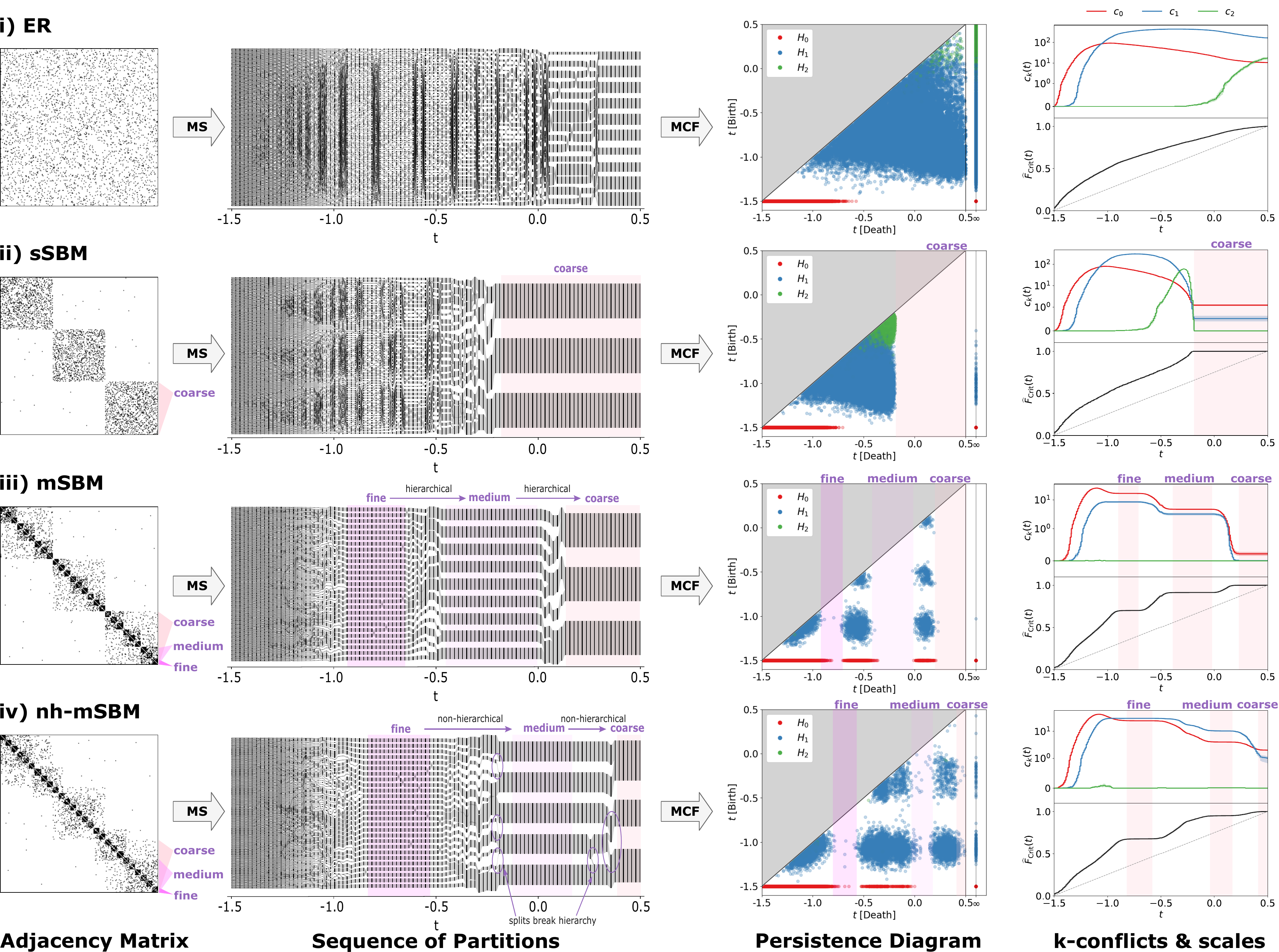}
    \vspace{-0.7cm}
    \caption{%
    {i)-iv) For each model (ER, sSBM, mSBM, nh-mSBM), we present one realisation of the adjacency matrix, the corresponding (non-hierarchical) sequence of partitions %
    visualised with a Sankey diagram, and the MCF PD. We annotate the fine, medium and coarse ground-truth scales if present and mark splits of clusters that break the hierarchy in the Sankey diagram of iv) with ellipses. We then compute: the ensemble PD of all 200 samples from each model; the average %
    persistent $k$-conflict $c_k(t)$ for $k=0,1,2$ (Definitions~\ref{def:persistent_0_conflict} and \ref{def:average_k_conflict}), %
    and the normalised cumulative count of critical values $\widehat F_\mathrm{Crit}(t)$ (Definition~\ref{def:critical_values_cdf}) with 95\% confidence intervals. %
    Our scale selection heuristic (sharp dips in $c_k(t)$ followed by plateaux in $\widehat F_\mathrm{Crit}(t)$, see Remark~\ref{rem:heuristic_gaps}) identifies conflict-resolving partitions that correspond well with ground-truth planted partitions at different scales (shaded in pink; determined \emph{a posteriori} by comparison with the planted partitions, Figure~\ref{S_fig:ensembles_nvi}). %
    No %
    plateaux are present for the ER model, confirming its lack of robust partitions. %
    }}
    \label{fig:sbms_summary}
    \vspace{0cm}
\end{figure}

The MCF PD of the sequence provides an %
interpretable representation of its structure, as seen by computing the pairwise 2-Wasserstein distances (Equation~\ref{eq:Wasserstein_distance}) between the 0- and 1-dimensional PDs of all 800 sequences.
Figure~\ref{fig:sbm_model_comparison} shows that the mean pairwise Wasserstein distance within each SBM model is significantly smaller than across SBM models ($p<0.0001$, 
Wilcoxon signed-rank test with Benjamini-Yekutieli correction).
Hence, the MCF PDs capture the differences in hierarchy and multiscale structure in the sequences of partitions from the four models (ER, sSBM, mSBM, nh-mSBM), whereas the Frobenius distance between matrices $A_i$ cannot distinguish across models
($p>0.5$).

Figure~\ref{fig:sbms_summary} also shows 
 the persistent $k$-conflict measures $c_k(t)$, $k\ge 0$, and the normalised cumulative count of the MCF critical values  $\widehat F_\mathrm{Crit}(t)$ averaged over the ensemble of realisations for each SBM  
 (see Table~\ref{tab:mathematial_symbols} for definitions). %
For the sequences of partitions obtained from the ER model, the ensemble PD shows no distinctive gaps in the death times for dimensions 1 and 2, 
indicating that there are no conflict-resolving partitions (Remark~\ref{rem:heuristic_gaps}), with  a large number of conflicts unresolved, as indicated by points at infinity, many of them two-dimensional and thus involving more ``conflict parties'' (Remark~\ref{rem:conflict_parties}). 
The  %
average 0-conflict (Equation~\ref{eq:average_0_conflict}) has a high value: $\bar c_0=0.7734$ (0.7726--0.7743), where we report 95\% confidence intervals over the 200 realisations. This shows that the sequences of partitions from the ER model have no natural quasi-hierarchical ordering.
For the sSBM model, the ensemble PD shows a distinct gap after $t=-0.2$, corresponding to a conflict-resolving partition which leads to a very  small number of unresolved conflicts at infinity, none of which are two-dimensional.  This natural scale is also robustly indicated by a plateau in $\widehat F_\mathrm{Crit}(t)$ after a dip in the persistent $k$-conflict $c_k(t)$ (see Remark~\ref{rem:heuristic_gaps}). We have checked \emph{a posteriori} that this scale corresponds to the planted partition in the sSBM (shaded in pink, see Figure~\ref{S_fig:ensembles_nvi}). The persistent 0-conflict rises initially due to the many non-nested clusters resulting from the random structure at the fine scales but decreases once we approach the ground-truth 3-way block structure. Note that $c_0(t)$ does not reach the value of a perfect hierarchy $c_0(t)=0$ 
because of the stochastic nature of the sSBM and the fact that the MCF contains information across all scales. Of course, the $c_0(t)$ would reach the zero value if the final scale corresponded to the partition into one cluster, $\hat 1$.
Overall, this leads to an average 0-conflict $\bar c_0=0.56$ (0.54--0.57), which indicates the presence of quasi-hierarchy across long intervals in the large scales in the sequence of partitions.
For the mSBM model, we find  three gaps in the distribution of birth-death tuples, also seen as plateaux in $\widehat F_\mathrm{Crit}(t)$ after dips in $c_k(t)$, which correspond to the intrinsic scales in the planted model with a  
low $\bar c_0=0.36$ (0.35--0.37) indicating strong quasi-hierarchy across scales. 
Finally, for the nh-mSBM model, we find increased values of the average %
0-conflict %
with $\bar c_0=0.675$ (0.672--0.679), reflecting the lack of nestedness, yet with gaps between birth-death tuples and the corresponding plateaux in $\widehat F_\mathrm{Crit}(t)$ after dips in $c_k(t)$ for the three intrinsic scales of the model.

In summary, our numerics illustrate how the MCF can provide a rich, distinctive summary characterisation for sequences of partitions generated from the four SBMs. The MCF measures reflect the planted structure, 
from the level of hierarchy %
to the multiscale nature of conflict-resolving partitions. %
The PD itself, in conjunction with the Wasserstein distance as a metric, can then act as a good representation to distinguish the different SBM models.  

\subsection{%
Application to Single-Cell Data}\label{sec:single_cell}

Unsupervised clustering algorithms play an important role in the detection and annotation of  cell types based on high-dimensional single-cell transcriptomic data~\citep{tasicSharedDistinctTranscriptomic2018}. %
In this setting, multiscale clusterings appear as a result of multiple levels of transcriptomic similarity between cell types and subtypes, either as part of cell populations and subpopulations, or as observed across dynamical or developmental processes, such as cell differentiation~\citep{haoIntegratedAnalysisMultimodal2021, kiselevSC3ConsensusClustering2017}. We illustrate in this section how MCF can be used to analyse and compare the multiscale structure of single-cell data using both synthetic data (with  ground truth) and real-world datasets. 

\subsubsection{Synthetic Single-Cell Data with Noisy Phylogenetic Structure
}\label{sec:synthetic_single_cell}

Multiscale structure in single-cell datasets results from, e.g., the presence of cell types and subtypes in populations or from phylogenetic differentiation processes. However, such structure can be blurred by %
stochasticity at the single-cell level, which can obscure the biological phylogeny in the observed single-cell data~\citep{chang_hemberg_nature2008}. To study this issue in a controlled manner, our second experiment computationally generates datasets of single cells from known cell phylogenies with varying levels of %
cellular noise. We then apply a popular  single-cell clustering algorithm to obtain a multiscale clustering of the generated data. As expected, as the level of cellular noise increases, the hierarchical structure of the phylogeny becomes less apparent in the multiscale clustering extracted from the data. We then employ MCF to analyse the multiscale clusterings and to predict the level of cellular stochasticity from the multiscale structure of the simulated single-cell dataset.

\begin{figure}[!htb]
    \centering
    \includegraphics[width=\textwidth]{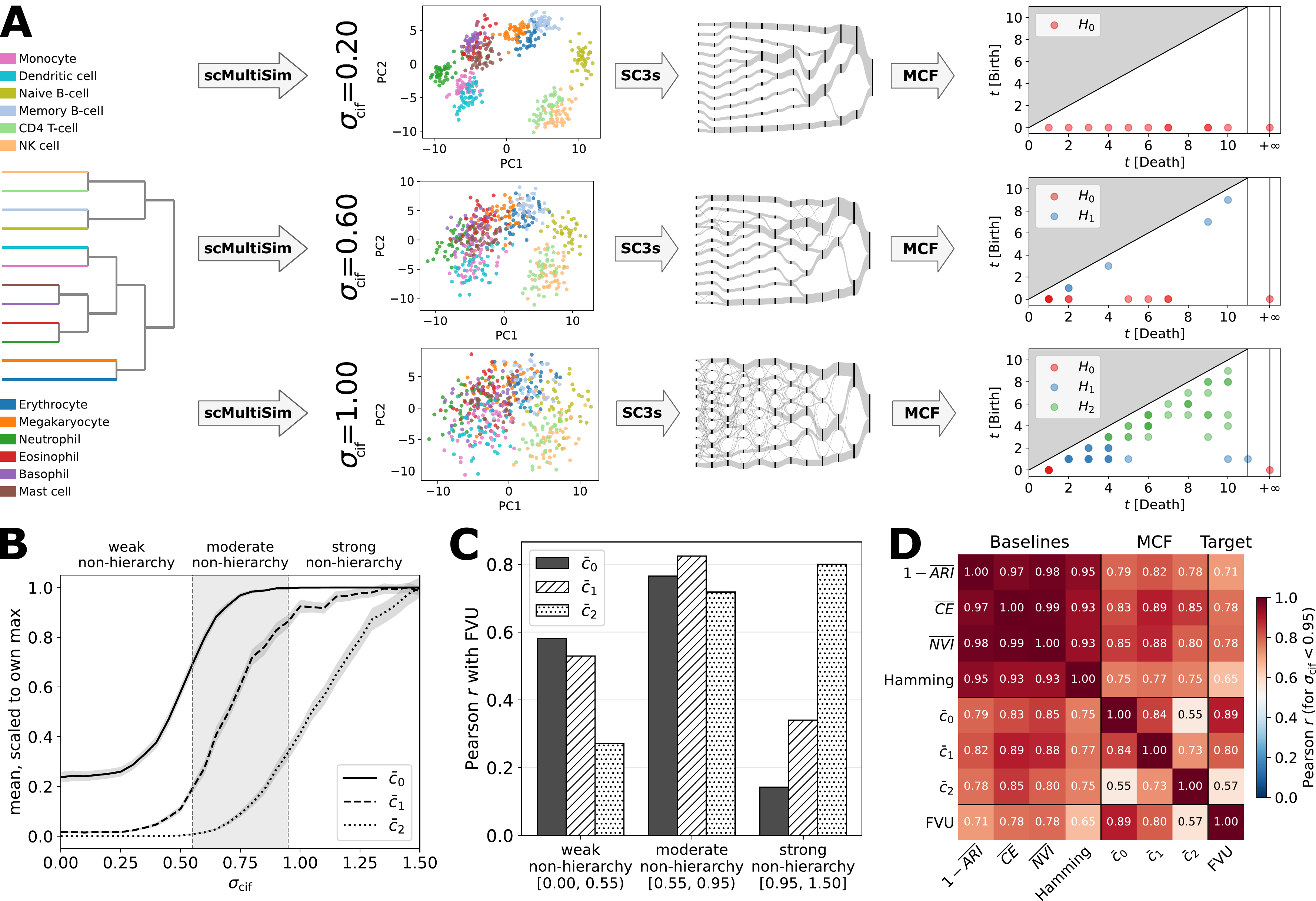}
    \vspace{-0.9cm}
    \caption{MCF applied to synthetic single-cell data. (\textbf{A}) Illustration of data simulation for the (simplified) haematopoietic phylogeny. (\textbf{B}) The MCF measures $\bar c_0$, $\bar c_1$, $\bar c_2$ distinguish three different bands of non-hierarchy as a function of  $\sigma_\mathrm{cif}$. (\textbf{C}) Pearson correlation of the MCF average measures against FVU for the three different bands. (\textbf{D}) Pearson correlations between baseline summary measures, MCF summary measures and FVU for $\sigma_\mathrm{cif}<0.95$.
    }
    \label{fig:sc_synthetic_summary}
    \vspace{-0.3cm}
\end{figure}

We use \texttt{scMultiSim}~\citep{liScMultiSimSimulationSinglecell2025} to generate synthetic single-cell datasets from ground-truth phylogenies of cell types. Each dataset is composed of $N=500$ cells, where each cell is described by a transcriptomic profile of 200 genes.
To cover a broad range of biological systems, we simulate single-cell datasets for six different cell differentiation trees, including haematopoiesis and cortical neuronal subtypes.
Each phylogeny is simulated for 31 different noise levels $\sigma_{\mathrm{cif}}\in[0,1.5]$, i.e., the standard deviation of the \emph{cell identity factors} (CIF)~\citep{liScMultiSimSimulationSinglecell2025}, where increasing $\sigma_{\mathrm{cif}}$ means that the ground-truth multiscale structure induced by the differentiation tree gets increasingly blurred. We use 27 random seeds for each simulation for a total of 5,022 datasets (27 seeds, 31 $\sigma_{\mathrm{cif}}$, 6 differentiation trees). 
Next, we obtain a multiscale clustering from each of the synthetic single-cell datasets using the widely used \texttt{SC3s} method~\citep{quahSC3sEfficientScaling2022, kiselevSC3ConsensusClustering2017} with the number of clusters decreasing from 12 to 1. This results in 5,022 sequences of partitions $\theta$ with $M=12$ change points.
For each $\theta$, we then compute the nerve-based MCF $\tilde{\mathcal{M}}_\theta$ and its PH (average total compute time is 0.054s).
Figure~\ref{fig:sc_synthetic_summary}\textbf{A} provides an illustration of the simulated data, the Sankey diagrams, and the PDs for different levels of $\sigma_\mathrm{cif}$. 
See  Appendix~\ref{appx:sc_synth_data} for 
further simulation details. 

An intrinsic measure of explainability of the cell labels in our dataset is the \emph{fraction of variance unexplained}, $0\le\mathrm{FVU}%
\le 1$, with $\mathrm{FVU}=0$ when the cell types perfectly explain the variance in the data
(see Equation~\ref{eq:fvu_ss} in Appendix~\ref{appx:sc_regression_task}). %
We have checked that, as expected, there is a strong Pearson correlation ($r=0.95$) between $\sigma_{\mathrm{cif}}$ and FVU. %
Given that the datasets are simulated from a phylogeny, a high FVU means that the ground-truth hierarchical cell-type structure is poorly reflected in the data. Hence, we use FVU as a ground-truth proxy 
for non-hierarchy in the data.  

Figure~\ref{fig:sc_synthetic_summary}\textbf{B} shows that the dependence of the MCF measures $\bar c_0$, $\bar c_1$, $\bar c_2$ on $\sigma_\mathrm{cif}$ captures the gradient of hierarchy: a region of weak non-hierarchy for $0\le\sigma_\mathrm{cif}<0.55$ (when $\bar c_2 \approx 0$), 
and a region of strong non-hierarchy for $0.95\le \sigma_\mathrm{cif}\le 1.5$ (when $\bar c_0$ reaches its maximum), with a transition region between both. 
We also find that $\bar c_0/\bar c_1/\bar c_2$ have the highest Pearson correlation with FVU within the first/second/third regions, respectively  (Figure~\ref{fig:sc_synthetic_summary}\textbf{C}).
Furthermore, Figure~\ref{fig:sc_synthetic_summary}\textbf{D} shows that the three MCF measures provide complementary, non-redundant %
information ($r=0.55$ for $\bar c_0$ and $\bar c_2$), and the cross-correlation %
with other standard baselines (ARI, CE, NVI, Hamming) is lower than the cross-correlation between the baseline measures. 
The distinct characterisation of non-hierarchy provided by $\bar c_0, \bar c_1, \bar c_2$ is consistent with the fact that $k$-conflicts for higher $k$ involve more complicated patterns of higher-order inconsistencies (Remark~\ref{rem:conflict_parties}).

Given the underlying phylogenetic origin of the data, we treat FVU as a regression target to be predicted from the observed data. This regression task can be seen as predicting the degree of organisation in the single-cell data.
Here, we compare the performance of our MCF Betti curves $\beta_k(t_m)$, $k=0,1,2$  
(a $3M$-dimensional vector) against various baseline representations of the data: (i) the $N \times 4$ matrix of first four cell PCA coordinates; (ii) the $N\times M$ matrix of the
 (non-unique) raw label encoding of $\theta$, where the $m$-th column contains the labels of the clusters in $\theta(t_m)$ assigned to the elements in $X$; (iii) the Hamming distance of the Sankey diagram $S(\theta)$ to the closest tree computed by \texttt{MRtree}; and (iv) $(M-1)$-dimensional vectors containing the ARI, CE and NVI measures for consecutive partitions $\theta(t_m)$ and $\theta(t_{m+1})$. 
 Note that the average baseline measures from (iii) and (iv) have high cross-correlation suggesting that they carry similar information 
 (Figure~\ref{fig:sc_synthetic_summary}\textbf{D}). This is consistent with the fact that these measures fail to capture higher-order inconsistencies in the multiscale clusterings, see Sections~\ref{sec:discussion_mrtree} and \ref{sec:discussion_information}.
 
 We produce a balanced split of the standardised data (64\% train, 16\% validation, 20\% test) and perform hyperparameter tuning of Ridge regression and multilayer perceptron (MLP) models on the train and validation data and evaluate the performance on the unseen test data. See Appendix~\ref{appx:sc_regression_task} for details.
Table~\ref{tab:fvu_test_r2} shows that the MCF features lead to the highest test $R^2$ score for both models. The overall best score $R^2=0.946$ is achieved by MLP with MCF features and the performance is significantly better than MLP with NVI features ($p<0.0001$, Wilcoxon signed-rank test on squared test residuals). 
See Table~\ref{tab:fvu_train_r2} in Appendix~\ref{appx:sc_regression_task} for the train scores and indications of overfitting for certain representations. 
The fact that Ridge regression on the MCF features achieves $R^2=0.932$ suggests a strong linear relationship between MCF features and FVU, whereas the gap between Ridge and MLP is larger for all other descriptors. Feature importance in Ridge 
shows negative weights for the $\beta_0$-curve, consistent with higher $\beta_0$ values implying lower 0-conflict (Proposition~\ref{prop:0-dim-Betti}), and positive weights for the $\beta_1$- and $\beta_2$-curves, consistent with higher-dimensional homology as a measure of conflict (Definition~\ref{def:k-conflicts}). This demonstrates the interpretability of MCF Betti curves as feature maps for explainable AI~\citep{adadiPeekingBlackBoxSurvey2018}.

\begin{table}[!htb]
\centering
\begin{tabular}{lccccccc}
\toprule
 & PCA & Raw label encoding & Hamming & ARI & CE & NVI & MCF \\
\midrule
Ridge & $-0.009$ & 0.351 & 0.699 & 0.822 & 0.811 & 0.863 & \textbf{0.932} \\
MLP & 0.653 & 0.751 & 0.756 & 0.884 & 0.914 & 0.925 & \textbf{0.946} \\
\bottomrule
\end{tabular}
\caption{Test $R^2$ scores for FVU regression. See Table~\ref{tab:fvu_train_r2}
(Appendix~\ref{appx:sc_regression_task}) for train scores.}
\label{tab:fvu_test_r2}
\end{table}

\subsubsection{Real-World Single-Cell Data}\label{sec:real_single_cell}

In our third experiment, we apply MCF to analyse the multiscale structure of two experimental single-cell datasets: i) `Baron': a dataset of human pancreatic islets, each described by 20,125 genes, comprising $N=8,569$ cells of 14 distinct cell types~
\citep{baronSingleCellTranscriptomicMap2016}; and ii) `Paul': a dataset of mouse myeloid progenitor cells, each described by 25,188 genes, comprising $N=8,834$ cells (after removing cells with fewer than 200 genes detected), which was  collected by \citet{paulTranscriptionalHeterogeneityLineage2015} from a process of active, continuous differentiation without ground-truth cell type annotation but with some emergent subpopulations. See Appendix~\ref{appx:real_single_cell} for details. %
We chose the two datasets because of their similar size and different structure: \citet{baronSingleCellTranscriptomicMap2016} report a hierarchical organisation of discrete cell types, whereas \citet{paulTranscriptionalHeterogeneityLineage2015} identify subpopulations that are not organised hierarchically. This qualitative difference is already captured visually in the  UMAP~\citep{mcinnesUMAPUniformManifold2020} projection of both datasets with clearly separated clusters for Baron and a lack of cluster separation for Paul (Figure~\ref{fig:sc_real_summary}\textbf{A}).

We apply the popular   \texttt{Seurat}~\citep{haoIntegratedAnalysisMultimodal2021} single-cell clustering method to both datasets and compute sequences of partitions with $M=70$ change points (indexed by $t_m=m$ for $m=1,\dots,M$), such that the number of clusters decreases from about 25 to 1 (see Figure~\ref{fig:sc_real_summary}\textbf{A} for Sankey diagram visualisations and Appendix~\ref{appx:real_single_cell} for details). 
To account for computational variability in the clustering method, we generate an ensemble of 100 sequences of partitions for each dataset using different initial conditions for the Louvain optimisation in \texttt{Seurat} and compute the PH for each of those 200 sequences using the nerve-based MCF (12s per realisation for Baron; 800s per realisation for Paul). The lack of hierarchy in Paul leads to more cluster intersections and a 21 times larger filtration. See Figure~\ref{fig:sc_real_summary}\textbf{A} for summary PDs of both ensembles.

\begin{figure}[!htb]
    \centering
    \includegraphics[width=\textwidth]{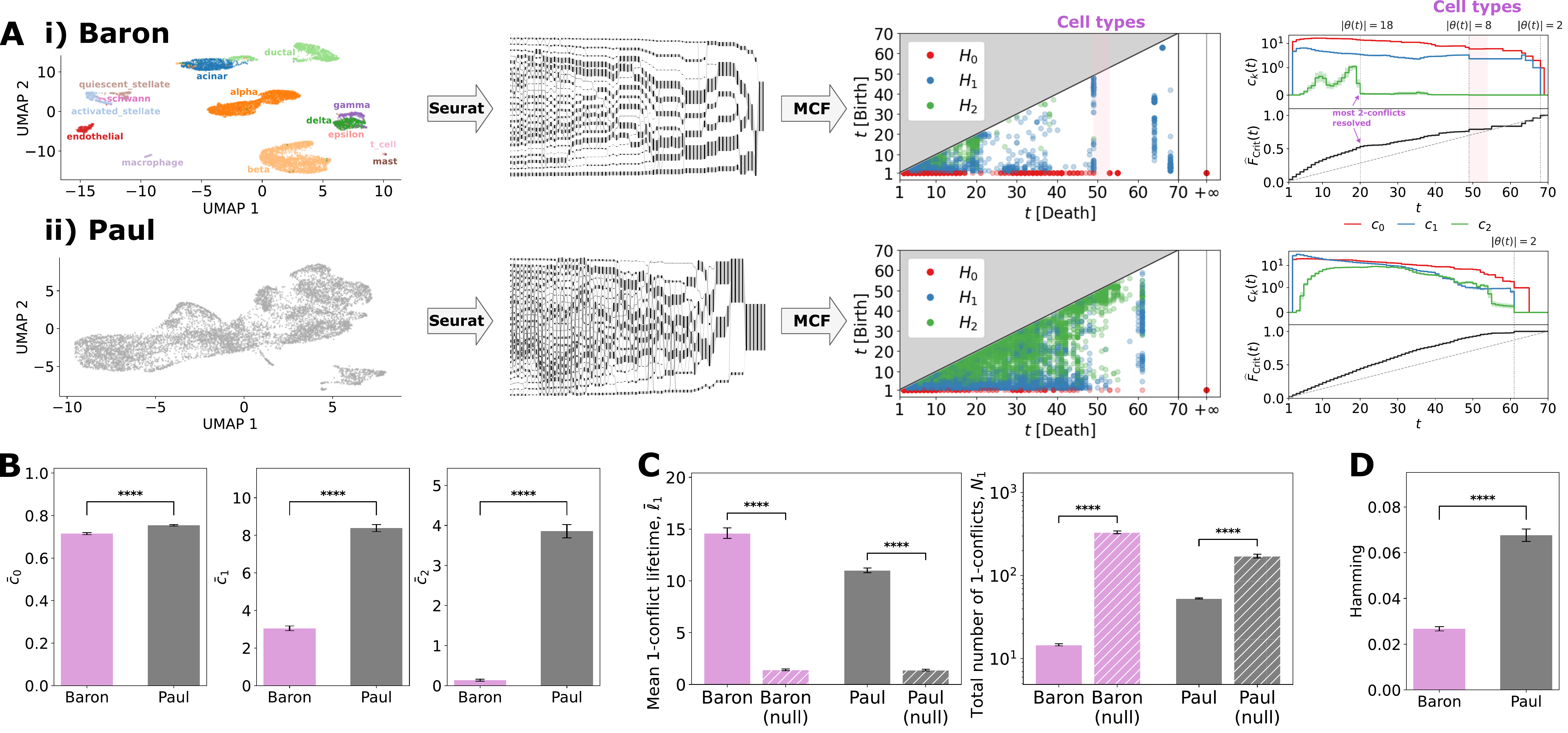}
    \vspace{-0.9cm}
    \caption{MCF applied to two real-world single-cell datasets (`Baron' and `Paul'). \textbf{A} From left to right: UMAP projection of the dataset; one \texttt{Seurat} multiscale clustering visualised with a Sankey diagram; the ensemble PD of all 100 multiscale clusterings for the dataset; the average persistent $k$-conflict curves $c_k(t)$ for $k=0,1,2$ (top) %
    and the normalised cumulative count of critical values $\widehat F_\mathrm{Crit}(t)$ (bottom), all with 95\% confidence intervals. For Baron, dips in $c_k(t)$ followed by plateaux in $\widehat F_\mathrm{Crit}(t)$  reveal conflict-resolving partitions including the ground-truth cell types (shaded in pink; determined \emph{a posteriori} from the annotated cell types, Figure~\ref{fig:sc_real_nvi}). %
    \textbf{B} The average 0-, 1- and 2-conflict is significantly higher for Paul. \textbf{C} Randomising the two datasets leads to significantly more 1-conflicts with significantly shorter lifetimes. \textbf{D} Hamming distances between Sankey diagrams and closest dendrogram are significantly higher for Paul. 
    } \label{fig:sc_real_summary}
    \vspace{0cm}
\end{figure}

We find that the mean pairwise 2-Wasserstein distances between MCF PDs are significantly smaller for PDs within an ensemble ($28.6$ for Baron vs $30.4$ for Paul at dimension $k=1$) than for PDs across the two ensembles ($68.3$ at dimension $k=1$) ($p<0.0001$, Wilcoxon signed-rank test with Benjamini-Yekutieli correction). Hence, MCF can distinguish the two datasets based on their multiscale structure. %
Our MCF summary measures capture the expected difference in hierarchy for both datasets with significantly lower average 0-, 1- and 2-conflict ($p<0.0001$,  Mann-Whitney U test with Benjamini-Yekutieli correction) for Baron compared to Paul (Figure~\ref{fig:sc_real_summary}\textbf{B}), with increasing difference as $k$ grows. 
This is consistent with the fact that \citet{paulTranscriptionalHeterogeneityLineage2015} collected the cells  during a continuous differentiation process, and is
also consistent with Section~\ref{sec:synthetic_single_cell}, i.e., the gene variance of the cells in Baron is better explained by discrete cell types than in Paul. This is also seen in the higher Hamming distances between Sankey diagrams and the closest dendrogram determined by \texttt{MRtree}~\citep{pengCellTypeHierarchy2021} (Figure~\ref{fig:sc_real_summary}\textbf{D}).

Using our scale selection heuristic (a plateau in $\widehat F_\mathrm{Crit}(t)$ after a dip in the persistent $k$-conflict $c_k(t)$ for $k=0,1,2$, see Remark~\ref{rem:heuristic_gaps}), we identify two conflict-resolving partitions in Baron  (Figure~\ref{fig:sc_real_summary}\textbf{A}). The first ($t=20$) is caused by the resolution of most 2-conflicts.  
Indeed, an inspection of birth and death simplices of the 2-dimensional homology generators computed by \texttt{GUDHI} reveals that the 2-conflicts emerge through a three-way split of alpha cells and get resolved with a two-way split around $t=20$. This is consistent with \citet{baronSingleCellTranscriptomicMap2016} who report that alpha cells have no detectable substructure. The second conflict-resolving partition ($t=49$) coincides with the partition that is closest to the ground-truth cell types (checked \emph{a posteriori} with NVI and shaded in pink). In contrast, conflicts in Paul are continuously created and resolved without any partition standing out (apart from the bipartition at $t=61$), consistent with this dataset being captured during continuous cell differentiation rather than at terminal states~\citep{paulTranscriptionalHeterogeneityLineage2015}.

It is also instructive to consider the MCF analysis of  randomised nulls of the two datasets. To do so, we randomise each dataset 20 times by drawing gene expressions for each cell from a multinomial null model~\citep{townesFeatureSelectionDimension2019} and process the data as before with \texttt{Seurat} and MCF. The MCF of the randomised data is already fully connected from the second change point leading to higher average 0-conflict, as the clusters between consecutive partitions tend to completely overlap in random data. Our analysis in Figure~\ref{fig:sc_real_summary}\textbf{C} shows that the randomised data leads to significantly more 1-conflicts ($p<0.0001$, Mann-Whitney U test) with significantly shorter lifetimes ($p<0.0001$, Mann-Whitney U test), because the many random clusters in the structureless data lead to a fast resolution of conflicts, see Remark~\ref{rem:k_conflict_count_vs_lifetime}. Hence, comparison with the null model reveals that the 0- and 1-conflicts in the real data are intrinsic, non-random features of the single-cell multiscale structure (more hierarchical in Baron and less so in Paul).

\subsection{Application to Real-World Temporal Social Interaction Data}\label{sec:primary_school}

In our last experiment, we analyse the multiscale structure in a temporal sequence of partitions computed from 77,600 face-to-face interactions between $N=242$ individuals (232 students from 10 school classes and 10 teachers) at a French primary school over two days (1--2 October 2009) recorded by \citet{stehleHighResolutionMeasurementsFacetoFace2011}. Each school day runs from 8:30am to 4:30pm with a 25-minute morning recess starting around 10:30am, a 2-hour lunch break around 12--2pm, and a 25-minute afternoon recess starting around 3:30pm. The schedule times are approximate because classes enter recess and lunch in a staggered way. See Appendix~\ref{appx:primary_school} for details.

To process the data, we construct a temporal graph by generating the contact networks over each 4-minute snapshot, leading to $130$ snapshots for each of the two school days. We then compute one partition per snapshot by optimising 
Modularity with the Leiden algorithm~\citep{traagLouvainLeidenGuaranteeing2019}, and we concatenate them into two sequences of partitions $\theta_i$, $i=1,2$, one per day (see Figure~\ref{fig:primary_school_summary}\textbf{A} for a Sankey diagram visualisation). The sequences are indexed over physical time $t$ with change points $t_m$ %
spaced four minutes apart.
We follow the original work of \citet{stehleHighResolutionMeasurementsFacetoFace2011}, who analysed the data over a 20-minute rolling window, and
we compute the nerve-based MCF for every 20-minute subsequence $\theta_i([t_{m-2},t_{m+2}])$ %
(total compute time over all windows is 1.20s). 
Consecutive windows share four of their five partitions and are labelled by their central change point $t_m$. Applying MCF to a rolling window over the sequence of partitions  circumvents the saturation arising from $\theta_i$ not being globally coarse-graining 
(Remark~\ref{rem:ordering}).

\begin{figure}[!htb]
    \centering
    \includegraphics[width=\textwidth]{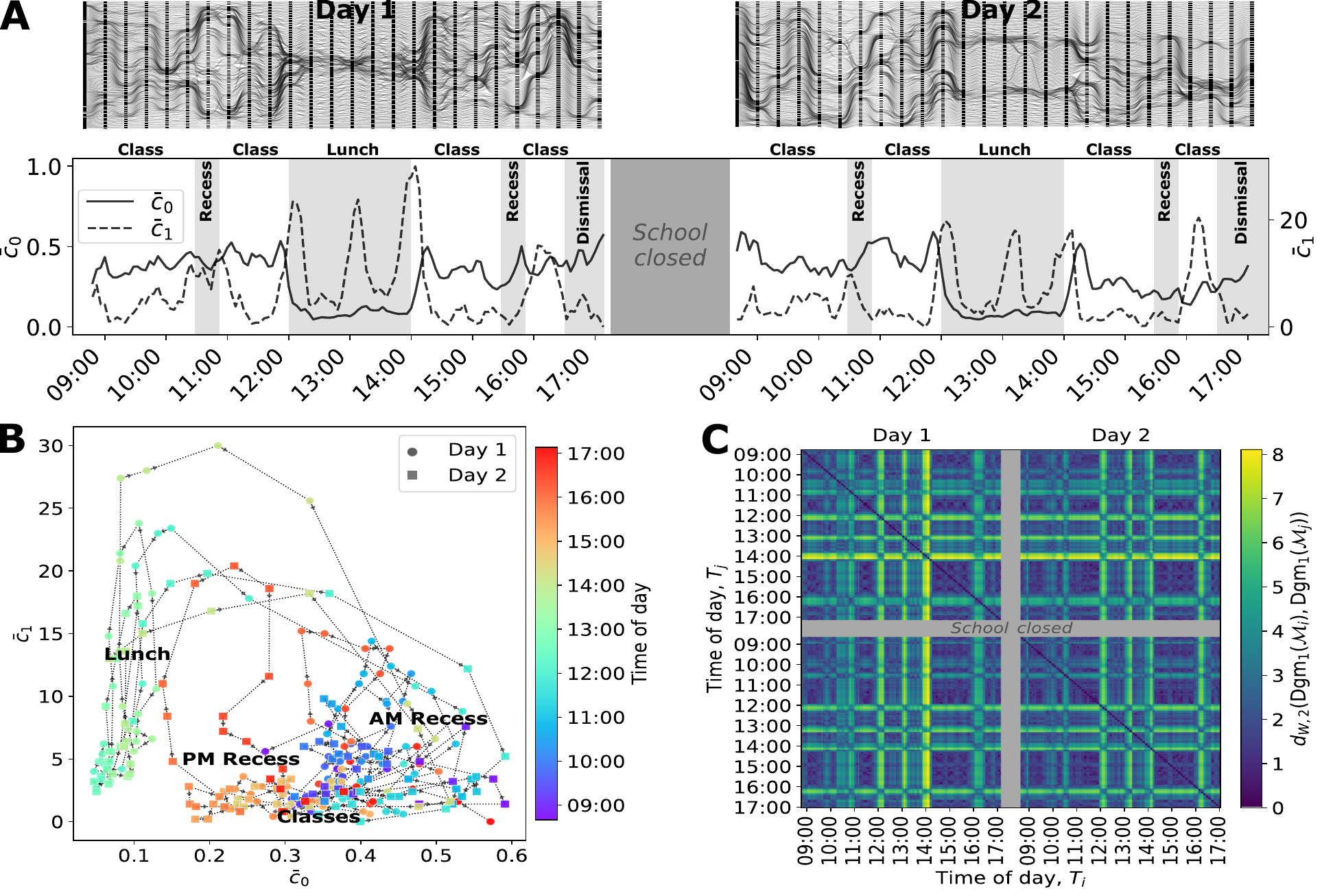}
    \vspace{-1cm}
    \caption{MCF applied to temporal clustering of primary school social interaction data over two days. (\textbf{A}) For each day, the temporal clustering is visualised with a Sankey diagram. The MCF average 0-conflict ($\bar c_0$) and average 1-conflict ($\bar c_1$) computed over a 20-minute rolling window show that recess and lunch are characterised by higher-order non-hierarchies (high $\bar c_1$) due to social mixing during breaks but that lunch has weaker lower-order non-hierarchies (low $\bar c_0$) because many individuals leave school for lunch. (\textbf{B}) Temporal trajectories in the $(\bar c_0, \bar c_1)$ space clearly distinguish different schedule events and (\textbf{C}) the pairwise 2-Wasserstein distances between the 1-dimensional PDs show that patterns of higher-order non-hierarchies are consistent across schedule events and over both days.
    }
    \label{fig:primary_school_summary}
    \vspace{0cm}
\end{figure}

The analysis of the MCF shows that the morning recess, lunch break and end of school are characterised by higher-order non-hierarchies (high $\bar c_1$) due to social mixing during these periods. We also observe a peak of $\bar c_1$ shortly after 4pm (nominally during class times), yet \citet{stehleHighResolutionMeasurementsFacetoFace2011} report that some classes go to the playground around 4pm due to staggered recess times. 
We also observe signatures of less regimented spatio-temporal mixing of students, as described by \citet{stehleHighResolutionMeasurementsFacetoFace2011}. For instance, the local maximum in $\bar c_1$ around 1pm corresponds to several classes moving from the cafeteria to the playground and back with ensuing mingling. Similarly, not all classes go to the playground at the same time during the morning %
recess; hence, the higher $\bar c_1$ in the second half of the recess appears due to different classes briefly meeting when going in and out of the playground.

Note that, while $\bar c_1$ depends only on the interacting individuals  (see Proposition~\ref{prop:restrict_true_overlaps}), $\bar c_0$ is affected by the number of singleton clusters, e.g., individuals that are either away from school or non-interacting. Hence, $\bar c_0$ also captures this complementary perspective, e.g., the lunch break has weaker lower-order non-hierarchy (low $\bar c_0$) because many individuals leave school for lunch and are thus singletons, which do not lead to violations of the refinement order. 
Combining both measures to generate temporal trajectories in the $(\bar c_0, \bar c_1)$ plane, we show in   Figure~\ref{fig:primary_school_summary}\textbf{B} that the $k=0,1$ homology features provide an informative mapping of the characteristics of AM/PM class times, AM/PM recesses and lunch breaks.

Summed over both days, 59\% of the 1-conflict lifetime (see Remark~\ref{rem:k_conflict_count_vs_lifetime}) arises from interactions between individuals from different school classes (41\% during class, 78\% during lunch), and the conflict-generating cluster pairs are significantly less class-concordant than the average overlapping pair in the same window ($p<0.001$, Wilcoxon signed-rank test on non-overlapping windows), see Appendix~\ref{appx:primary_school} for details.  

In addition, the Wasserstein distances between 1-dimensional MCF PDs across time reveal that higher-order interaction patterns recur from one day to the next within matched schedule periods  (Figure~\ref{fig:primary_school_summary}\textbf{C}). In particular, PDs from windows in the same period in the schedule have significantly smaller pairwise distances than PDs from
windows in different periods ($p<0.05$, Wilcoxon signed-rank test on non-overlapping windows), and same-period Wasserstein distances are essentially identical within and across days ($2.94$ vs $2.95$). %
Lunch windows are the least mutually similar (mean cross-day pairwise distance $3.63$ for lunch windows as compared to $2.85$ for class windows), reflecting the heterogeneity of mixing during lunchtime. This suggests that the patterns of non-hierarchy identified by MCF are intrinsic to the structure of social interaction.

These results enrich the descriptive analysis by \citet{faillaDescribingGroupEvolution2024}, who characterised the changes in consecutive layers of the Sankey diagram of hourly snapshots.
Even with a more systematic analysis of the Sankey diagram,
e.g., using formigrams~\citep{kimAnalysisDynamicGraphs2020}, 
MCF provides us with additional measures that capture the higher-order structure that emerges through patterns of splits and merges over more than two time steps.

\subsection{Computational Efficiency of MCF}\label{sec:computational_efficiency}

 To evaluate the computational efficiency of the MCF and nerve-based MCF, we generate synthetic non-hierarchical sequences of partitions with varying numbers of elements $N$ and change points $M$, 
 such that the number of clusters decreases linearly %
across the $M$ change points
from $N$ ($\theta(t_1)=\hat 0$) to 1 ($\theta(t_M)=\hat 1$). 
In particular, we construct $\theta(t_{m+1})$ from $\theta(t_{m})$ for $1\le m\le M-2$ by repeatedly merging the two smallest clusters until reaching the target coarseness
and then swapping 5\% of cluster assignments randomly to break hierarchy. This means that the dimension of the standard MCF is equal to $N-1$ (with the largest cluster being the full set), whereas the dimension of the nerve-based MCF is equal to $M-1$.

\begin{figure}[!htb]
    \centering
    \includegraphics[width=\textwidth]{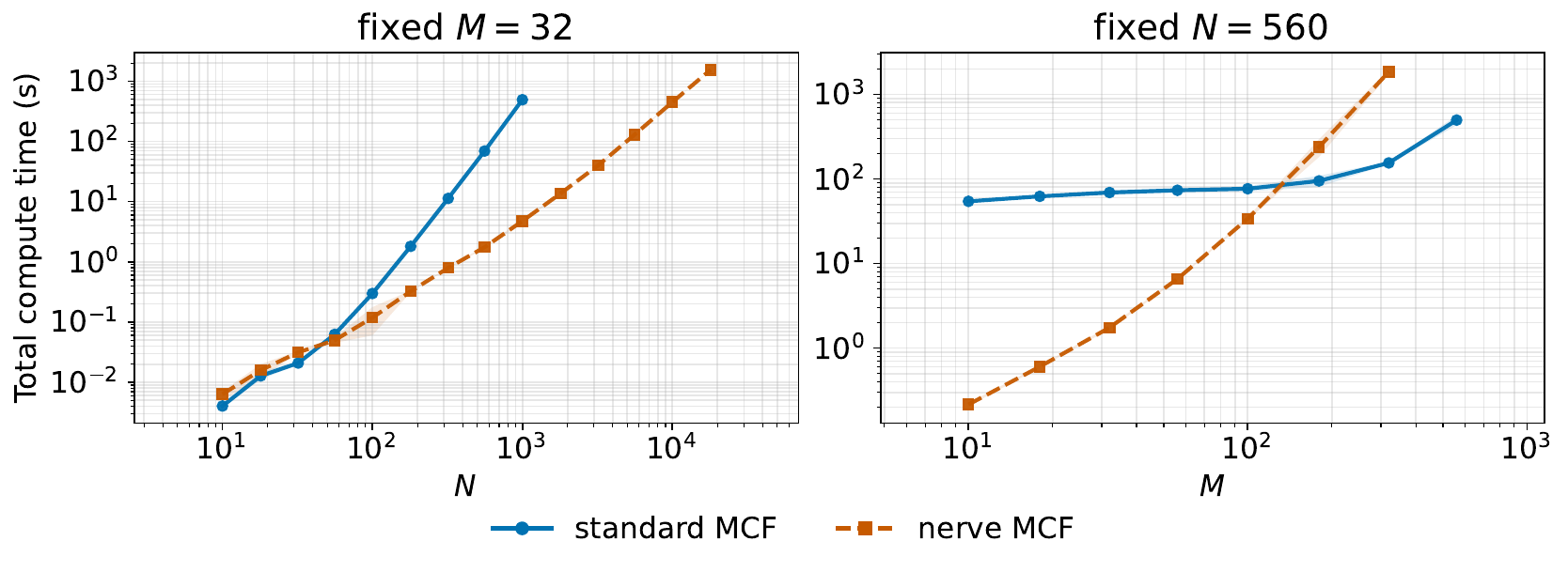}
    \vspace{-1.1cm}
    \caption{Total compute time of standard and nerve-based MCF (filtration construction and PH computation combined) for fixed $M=32$ with increasing $N$ and for fixed $N=560$ with increasing $M$.} 
    \label{fig:mcf_efficiency}
    \vspace{0cm}
\end{figure}

To analyse the computational trade-off between the standard and nerve-based MCF, we measure compute times of the filtration construction with simplices up to dimension 2, and the PH computation with $k\le 1$. This is done for a grid of $(N,M)$ values with  $10\le N\le 32,000$ and $10 \le M \le 10,000$. Compute times are averaged 
over five different realisations for each $(N,M)$ combination and a maximum time budget of 30 minutes was set for each of the computational steps. Figure~\ref{fig:mcf_efficiency} visualises the total compute time for fixed $M=32$ with increasing $N$ and for fixed $N=560$ with increasing $M$. It shows that the standard MCF is faster for the low-$N$/high-$M$ regime whereas the nerve-based MCF is faster for the high-$N$/low-$M$ regime, consistent with Remark~\ref{rem:computational_complexity}. More specifically, %
we estimate empirical power laws for the computational complexity of the filtration construction and PH computation separately (Table~\ref{tab:mcf_empirical_scale}) and find that the complexity of both methods is dominated by the PH computation, which scales $\propto N^{4.0}M^{0.7}$ for the standard MCF and $\propto N^{1.8}M^{3.5}$ for the nerve-based MCF.

\section{Discussion}\label{sec:discussion}

\subsection{Comparison to Dendrograms and Ultrametrics}\label{sec:discussion_ultrametrics}

We showed in Section~\ref{sec:CAG_hierarchical} that if the sequence of partitions $\theta$ is hierarchical, its MCF is fully described by dimension $k=0$ with trivial PH for dimension $k\ge 1$, and the MCF then corresponds to a VR filtration on the ultrametric space defined by the pairwise first-merge times $D_\theta$. 
In contrast to ultrametrics, the MCF applies generally to non-hierarchical sequences of partitions that emerge naturally in multiscale data analysis.

\subsection{Comparison to Minimal Distance to a Tree}\label{sec:discussion_mrtree}

The \texttt{MRtree}~\citep{pengCellTypeHierarchy2021} algorithm uses a greedy optimisation to compute a tree close to the Sankey diagram $S(\theta)$ of a sequence of partitions $\theta$. 
The Hamming distance between $S(\theta)$ and the computed \texttt{MRtree} tree can then serve as a measure of hierarchy of the Sankey diagram.
Within MCF, the average $k$-conflict measures also quantify non-hierarchy for $\theta$ (instead of $S(\theta)$), and, as expected, we find high correlation between the \texttt{MRtree} Hamming distance and our MCF measures 
(Section~\ref{sec:single_cell}). The MCF features, however, have a higher performance on the single-cell regression task in Section~\ref{sec:synthetic_single_cell} 
because they provide a richer description of different aspects of hierarchy, including higher-order inconsistencies that are missed by the Sankey diagram.

\subsection{Comparison to Information-based Measures}\label{sec:discussion_information}

The CE, $H(\theta(t_{m+1})|\theta(t_m))$, measures the information gain (or loss) from partition $\theta(t_m)$ to $\theta(t_{m+1})$, 
and is given by the negative expected log of the conditional probability matrix $P_{t_{m+1}|t_m}$
~\citep{meilaComparingClusteringsVariation2003}. 
Recall that the $m$-th layer of the Sankey diagram, 
$S_m$, 
is a bipartite graph between two sets of vertices corresponding to the clusters in  each layer
$V_{m+1}:=\{C_1,\dots,C_n\}$ and $V_{m}:=\{C_1',\dots,C_{n'}'\}$ 
with weights $W[i,j]=\lvert C_i\cap C_j' \rvert$ (see Appendix~\ref{appx:Sankey}). 
Then $P_{t_{m+1}|t_m}[i,j]=W[i,j]/\lvert C_j' \rvert$ is the degree-normalised weighted $n \times n'$ biadjacency matrix of $S_m$, and, hence, $H(\theta(t_{m+1})|\theta(t_m))$ is computed from its weights. %

As CE is only defined for pairs of partitions, the natural comparison is the zigzag MCF defined in Section~\ref{sec:zigzag_mcf}. %
 For $Z^{m,m+1}$, %
 we have $\dim H_k(Z^{m,m+1})=\dim H_k(S_m)$, where $\dim H_0(S_m)$ is the number of connected components of $S_m$ and $\dim H_1(S_m)$ is its cyclomatic number (see Remark~\ref{rem:zigzag_barcode_interpretation}). 
 Therefore, and in contrast to CE, the homology of $Z^{m,m+1}$ in the zigzag MCF is computed only from spectral properties of the $S_m$ graph (see Remark~\ref{rem:zigzag_barcode_interpretation}), which do not take into account the weights of the Sankey diagram.  

However, it is important to remark that CE remains pairwise, and cannot pick up the higher-order inconsistencies emerging across more than two partitions. Such higher-order effects are  captured with the $k$-conflict measures of the full MCF, 
as in Example~\ref{ex:zigzag_mcf}, where both sequences $\theta$ and $\eta$ 
have the same Sankey diagram, and so the same CEs, %
yet they are distinguishable via the full MCF. 
Note that these same shortcomings apply to the VI and NVI measures, both computed from the CE. Hence, MCF provides an advantage over pairwise information-based measures for downstream learning tasks where higher-order inconsistencies capture important features of multiscale data (see %
Section~\ref{sec:synthetic_single_cell}).

\subsection{Comparison to Formigrams}\label{sec:discussion_formigrams}

\citet{kimAnalysisDynamicGraphs2020} define a formigram as a function $f$ that maps a time $t\in\mathbb{R}$ to a subpartition of $X$, where a subpartition is a partition of a subset $X'\subseteq X$, such that the set of points of discontinuity (our change points) is locally finite. 
In addition, $f$ also needs to fulfil a \emph{comparability condition}, which ensures that consecutive partitions always form a zigzag with respect to the refinement order, i.e., $\forall t\in\mathbb{R}$ there is an $\epsilon>0$ small enough such that $f(t-\epsilon)\le f(t) \ge f(t+\epsilon)$.  
This technical condition allows %
the construction of the Reeb graph of the formigram and the computation of  its 0-dimensional zigzag %
PD, which tracks the connected components from one change point to the next. Importantly, for dimensions $k\ge 1$, the zigzag %
PD of the formigram is empty because of the comparability condition. Note that, in contrast, our definition of $\theta$ does not require the comparability condition, %
and the MCF can  quantify non-hierarchy and higher-order inconsistencies for $k\ge 1$ that are not captured by zigzag approaches, %
e.g., conflicts that span more than two partitions %
(see Section~\ref{sec:zigzag_mcf}). We also remark that the MCF 
construction would extend readily to the particular case of subpartitions considered in formigrams (see future work in Section~\ref{sec:Conclusion}).

The formigram framework was developed for analysing connected components in temporal graphs and has been applied to study animal flocking patterns. 
If, after additional pre-processing to comply with the comparability condition, formigrams were used to analyse the primary school dataset in Section~\ref{sec:primary_school}, %
the formigram PD (which is only 0-dimensional) would then capture changes in connectivity patterns between consecutive partitions (i.e., over each 8-minute window). In contrast,   
MCF is applied to a 20-minute rolling window of more than two partitions %
so that even the 0-dimensional PH of MCF captures the connectivity over a longer time interval and smooths out short-term noise. In addition, the MCF also encodes higher-order ($k \geq 1$) non-hierarchies emerging across time, which are important for characterising the different periods of the school schedule.

\subsection{Comparison to HodgeRank Analysis of Ranking Inconsistencies} \label{sec:discussion_hodgerank}

Recent advances in discrete Hodge theory have led to \emph{HodgeRank} models that can be used to analyse inconsistencies in a set of rankings~\citep{jiangStatisticalRankingCombinatorial2011,weiHodgeTheorybasedBiomolecular2022}. 
However, HodgeRank is not directly applicable to sequences of partitions as there is no natural ranking of points or clusters within each of the partitions.
An interesting direction left for future work would be to derive a set of %
rankings, parametrised by scale $t$, which encode the structure of $\theta$ so that HodgeRank models become applicable. This would allow us to exploit the discrete Hodge decomposition, including \emph{curl flows}, to capture global inconsistencies that are not caused by local ones, which could then be compared to our notion of $k$-conflicts.

\section{Conclusion and Future Work}\label{sec:Conclusion}

In this work, we developed a TDA-based framework for the analysis and comparison of (non-hierarchical) sequences of partitions $\theta$ arising in multiscale clustering. We defined the MCF $\mathcal{M}_\theta$, a filtration of abstract simplicial complexes that encodes arbitrary patterns of cluster assignments across scales, and proved the stability of its PD. We  showed that the MCF captures different aspects of non-hierarchy: violations of the refinement order (0-conflicts) at dimension 0 and higher-order cluster inconsistencies ($k$-conflicts) at dimension $k\ge 1$.

We illustrated numerically that the MCF PD provides a representation for sequences of partitions that can be used in downstream learning tasks. We illustrated these applications on several domains, including stochastic block models, single-cell data and temporal social interaction data. Importantly, MCF is not dependent on any specific clustering algorithm: it is generally applicable to sequences of partitions generated from a wide range of algorithms, including Markov Stability, Modularity, \texttt{SC3s} and \texttt{Seurat}. We also introduced an equivalent nerve-based MCF construction, and demonstrated it provides additional  scalability to sequences of partitions defined on large-$N$ datasets. %
Furthermore, the grounding of the MCF features in algebraic topology makes the representation interpretable, which is crucial for real-world data and explainable AI.

Several open research directions remain. The construction of the MCF (Equation~\ref{eq:MCF_construction}) is flexible and allows for more general inputs, such as sequences of subpartitions or sequences of soft partitions with overlapping clusters, and future work will aim to characterise the PH of the MCF in those cases. Future research could also investigate how to improve further the computational efficiency of the MCF PH computation by combining MCF and nerve-based MCF computations, so as to exploit their specific computational advantages. Finally, we leave for future work to compute minimal generators of the MCF PH classes~\citep{escolarOptimalCyclesPersistent2016} to locate not only when, but also where, conflicts emerge in the dataset.

\newpage

\acks{JS acknowledges support from EPSRC (PhD studentship through the Department of Mathematics at Imperial College London). MB acknowledges support from the EPSRC grant EP/N014529/1 supporting the EPSRC Centre for Mathematics of Precision Healthcare. We thank Heather Harrington for helpful discussions and for the opportunity to present work in progress at the Oxford Centre for Topological Data Analysis in January 2023. This work benefited from JS's participation in Dagstuhl Seminar 23192 ``Topological Data Analysis and Applications'' in May 2023. We thank Iris Yoon, Lewis Marsh and Amritendu Dhar Shekhar for valuable discussions about theoretical aspects of this work; Alexandre Bovet and Yasaman Asgari for useful discussions about the primary school dataset; and Zekai Li for helpful discussions about  single-cell data. We thank  Philipp Thomas for suggestions in earlier stages of this project.  Finally, we thank our anonymous reviewers for their constructive feedback that strengthened our paper significantly. We acknowledge the use of Claude Opus 4.6--5.5 and Claude Fable only for help with code implementations  %
and to check the validity of some of the proofs. Code suggestions from these tools were never accepted without scrutiny and final code was reviewed.}

\appendix

\section{Table of Mathematical Symbols}\label{appx:table_of_symbols}

\begin{table}[h]
\centering
\small
\begin{tabular}{lll}
\toprule
Symbol & Description & Reference \\
\midrule
$X$, $N$ & Finite set of points, $N=\lvert X\rvert$ & Section~\ref{sec:problem_def} \\
$\pi$, $C$, $\lvert C\rvert$ & Partition, cluster, cluster size & Section~\ref{sec:problem_def} \\
$\Pi_X$, $\le$ & Set of partitions of $X$, refinement order & Section~\ref{sec:problem_def} \\
$\hat 0$, $\hat 1$ & Partition into singletons, partition $\{X\}$ & Section~\ref{sec:problem_def} \\
$\theta$, $t$ & Sequence of partitions, scale & Definition~\ref{def:sequence_of_partitions} \\
$|\theta(t)|$ & Number of clusters in partition $\theta(t)$ & Section~\ref{sec:problem_def}\\
$t_1<\dots<t_M$ & Change points of $\theta$ & Definition~\ref{def:sequence_of_partitions} \\
$M$ & Number of change points of $\theta$ & Definition~\ref{def:sequence_of_partitions}\\
$S(\theta)$ & Sankey diagram of $\theta$ & Appendix~\ref{appx:Sankey} \\
$\Delta C$ & Solid simplex on cluster $C$ & Definition~\ref{def:MCF_Construction} \\
$\mathcal M_\theta=(K^t_\theta)_{t \ge t_1}$ & Multiscale Clustering Filtration (MCF) & Definition~\ref{def:MCF_Construction} \\
$D_{\theta,k}$, $D_\theta$ & Joint and pairwise first-merge times & Definition~\ref{def:joint-first-merge-times} \\
$\mathrm{Dgm}_k(\mathcal M)$ & $k$-dimensional persistence diagram & Appendix~\ref{background:PDs} \\
$d_{W,q}$, $d_{W,\infty}$ & Wasserstein and bottleneck distance & Appendix~\ref{background:PD_distance} \\
$c_0(t)$  & Persistent 0-conflict & Definition~\ref{def:persistent_0_conflict} \\
$c_k(t)=\beta_k(t)$ & Persistent $k$-conflict ($k$-th Betti number), $k\ge 1$ & Definition~\ref{def:average_k_conflict} \\
$\bar c_0$, $\bar c_k$ & Average 0- and $k$-conflict, $k\ge 1$ & Definitions~\ref{def:average_0_conflict} and \ref{def:average_k_conflict}\\
$\widehat F_\mathrm{Crit}(t)$ & Normalised cumulative count of critical values & Definition~\ref{def:critical_values_cdf} \\
$\tilde{\mathcal{M}}_\theta=(\tilde K^t_\theta)_{t \ge t_1}$  & Nerve-based MCF & Definition~\ref{def:nerve-based MCF}\\
\bottomrule
\end{tabular}
\caption{Table of main mathematical symbols.}\label{tab:mathematial_symbols}
\end{table}

\section{Proofs of Theoretical Results}\label{appx:proofs}

\subsection{Proofs for Section~\ref{sec:mcf_stability}}\label{appx:proofs_stability}

We first provide a proof for the MCF stability result in Proposition~\ref{prop:mcf_stability}.\\

\begin{proof}[Proof of Proposition~\ref{prop:mcf_stability}]
    One can equivalently define $\mathcal{M}_\theta$ from the sublevel sets of the filtration function $f:K \rightarrow \mathbbm{R}$ given by  
    $f(\sigma) = \min\{t\ge t_1\; |\; \exists\; C \in \theta(t):\; \sigma\subseteq C\}$, which is the smallest scale $t\ge t_1$ where all the elements $\sigma\subseteq X$ are contained in the same cluster in $\theta(t)$. Let us also denote by $\tilde{f}:K \rightarrow \mathbbm{R}$ the filtration function of $\mathcal{M}_\eta$, which is defined on the same domain $K$ by assumption. The Cellular Wasserstein Stability Theorem by \citet[Theorem 4.6]{skrabaWassersteinStabilityPersistence2025} then yields that 
    \begin{equation}\label{eq:skraba_stability}
    d_{W,q}\left(\Dgm_k(\mathcal{M}_\theta),\Dgm_k(\mathcal{M}_\eta)\right)^q \; \le\; \sum_{\dim\sigma\in\{k,k+1\}} |f(\sigma) - \tilde{f}(\sigma) |^q =: S_k,
    \end{equation}
    where $q\ge 1$ is required.
    
    Define the function $g:X^{k+2}\rightarrow 2^X\setminus\{\emptyset\}$ that maps $\boldsymbol{x}:=(x_0,\dots,x_{k+1})$ to the set $\{x_0,\dots,x_{k+1}\}$. 
    It follows directly from Definition~\ref{def:joint-first-merge-times} that $D_{\theta,k+1}(\boldsymbol{x})=f(g(\boldsymbol{x}))$ for $g(\boldsymbol{x})\in K$ and likewise $D_{\eta,k+1}(\boldsymbol{x})=\tilde f(g(\boldsymbol{x}))$. Hence, the summand $|f(\sigma) - \tilde{f}(\sigma) |^q$ is equal to $|D_{\theta,k+1}(\boldsymbol{x})-D_{\eta,k+1}(\boldsymbol{x})|^q$ for $\sigma=g(\boldsymbol{x})\in K$ and constant on each fibre $g^{-1}(\sigma)$. Together with $D_{\theta,k+1}(\boldsymbol{x})=D_{\eta,k+1}(\boldsymbol{x})=\infty$ for $g(\boldsymbol{x})\notin K$ and our convention $\infty - \infty := 0$, this leads to the identity
    \begin{align}\label{eq:identity_D_theta_k+1}
        \|D_{\theta,k+1}-D_{\eta,k+1}\|_q^q = \sum_{\sigma\in K}|g^{-1}(\sigma)| \cdot |f(\sigma) - \tilde{f}(\sigma) |^q,
    \end{align}
    where $|g^{-1}(\sigma)|$ is the cardinality of the fibre $g^{-1}(\sigma)$.
    
    Note that $m_j:=|g^{-1}(\sigma)|$ for $\dim\sigma=j$ is equal to the number of surjections from a set with $k+2$ elements to a set with $j+1$ elements. In particular, we have $m_{k+1}=(k+2)!$ and $m_{k}=(k+1)!\binom{k+2}{2}$ by direct counting. Hence, we get the lower bound $|g^{-1}(\sigma)|\ge (k+1)!$ for $\dim\sigma\in\{k,k+1\}$. %
    Together with Equation~\eqref{eq:identity_D_theta_k+1}, this yields
    \begin{align*}
        \frac{1}{(k+1)!} \|D_{\theta,k+1}-D_{\eta,k+1}\|_q^q &\;\ge\; 
        \frac{1}{(k+1)!} \sum_{\substack{\sigma\in K \\ \dim\sigma\in\{k,k+1\}}}|g^{-1}(\sigma)| \cdot |f(\sigma) - \tilde{f}(\sigma) |^q \\ 
        & \;\ge\; \frac{1}{(k+1)!} \sum_{\substack{\sigma\in K \\ \dim\sigma\in\{k,k+1\}}} (k+1)! \,\cdot\, |f(\sigma) - \tilde{f}(\sigma) |^q = S_k,
    \end{align*}
    and the statement follows by combining with Equation~\eqref{eq:skraba_stability}.
\end{proof}

\subsection{Proofs for Section~\ref{sec:higher_dim_ph}}\label{appx:proofs_higher_dim_ph}

We next provide a proof for Proposition~\ref{prop:restrict_true_overlaps} about restricting the MCF to the set of non-nested intersections.\\

\begin{proof}[Proof of Proposition~\ref{prop:restrict_true_overlaps}]
Let $x_0\in X\setminus\bar{X}$, then the largest cluster that contains $x_0$ up to scale $t\ge t_1$ is given by 
    $C_{x_0}(t):=\bigcup_{s \le t}\bigcup_{\substack{C\in\theta(s),\,%
    x_0 \in C}} C$,
which is a cluster in $\theta([t_1,t])$ because any two clusters containing $x_0\notin \bar X$ are nested.
In particular, $x_0\in\sigma$ for $\sigma\in K^t$ implies $\sigma\in \Delta C_{x_0}(t)\subseteq K^t$. 
Fix a vertex $y\in X\setminus\{x_0\}$ such that $y\in C_{x_0}(t)$ whenever $C_{x_0}(t)\neq \{x_0\}$ and $y\in C_{x_0}(t+p)$ whenever $C_{x_0}(t+p)\neq \{x_0\}$. This is possible because $C_{x_0}(t)\subseteq C_{x_0}(t+p)$. If $C_{x_0}(t+p)= \{x_0\}$, we choose $y$ arbitrary from $X\setminus\{x_0\}$. Let $\phi:X\rightarrow X$ be the vertex map with $\phi(x_0):=y$ and $\phi(x):=x$ for $x\neq x_0$. We define simplicial maps $f:K^t\rightarrow K|_{X\setminus\{x_0\}}^t,\sigma\mapsto\phi(\sigma)$ and $h:K^{t+p}\rightarrow K|_{X\setminus\{x_0\}}^{t+p},\sigma\mapsto\phi(\sigma)$, which commute with the canonical inclusions $K^t \hookrightarrow K^{t+p}$ and $K|_{X\setminus\{x_0\}}^t \hookrightarrow K|_{X\setminus\{x_0\}}^{t+p}$. In the following, we show that both $f$ and $h$ induce isomorphisms $f*$ and $h^*$ between the higher-dimensional homology groups $H_k(\cdot)$, $k \ge 1$ leading to the commutative diagram:
    \[\begin{tikzcd}
            H_k(K^t) \arrow{r} \arrow[swap]{d}{f^*} & H_k(K^{t+p}) \arrow{d}{h^*} \\
            H_k(K|_{X\setminus\{x_0\}}^t) \arrow{r} & H_k(K|_{X\setminus\{x_0\}}^{t+p})
        \end{tikzcd}
    \]
This yields  $H_k^p(K^t) \cong  H_k^p(K|_{X\setminus\{x_0\}}^t)$ %
and by induction we get $H^p_k(K^t) \cong H^p_k(K|_{\bar{X}}^t)$. 

\textbf{Case 1:} If $C_{x_0}(t)\neq \{x_0\}$, then both $f$ and $h$ are the contraction of the edge $[x_0,y]$ into vertex $[y]$. Recall that for a simplex $\tau$, $\Lk \tau$ refers to its link as defined in Appendix~\ref{background:complex}. It is easy to see that $f$ fulfils the so-called \emph{link condition}~\citep{attaliEfficientDataStructure2011} for the pair $(x_0,y)$ because in $K^t$ 
$\Lk[x_0]=\Delta(C_{x_0}(t)\setminus \{x_0\})$ and $\quad \Lk [x_0] \cap \Lk [y] = \Delta(C_{x_0}(t)\setminus\{x_0,y\}) = \Lk [x_0,y]$, since $\tau \cup \{x_0\}\in K^t$ implies $\tau \cup \{x_0,y\}\in K^t$. Similarly, $h$ also fulfils the link condition. The \emph{link condition theorem}~\cite[Theorem 2]{attaliEfficientDataStructure2011} then implies that $f^*$ and $h^*$ are isomorphisms.

\textbf{Case 2:} If $C_{x_0}(t)=C_{x_0}(t+p)=\{x_0\}$, then  $K^t=K|_{X\setminus\{x_0\}}^t \cup \{[x_0]\}$ and $K^{t+p}=K|_{X\setminus\{x_0\}}^{t+p} \cup \{[x_0]\}$. Since $f$ and $h$ restrict to the identity on $K|_{X\setminus\{x_0\}}^t$ and $K|_{X\setminus\{x_0\}}^{t+p}$, respectively, and map the isolated vertex $[x_0]$ to $[y]$, $f^*$ and $h^*$ are isomorphisms for the higher-dimensional homology groups.

\textbf{Case 3:} If $C_{x_0}(t)=\{x_0\}$ but $C_{x_0}(t+p)\neq\{x_0\}$ so that $y\in C_{x_0}(t+p)$, then $f^*$ is an isomorphism as argued in Case 2 and $h^*$ is an isomorphism as argued in Case 1.
\end{proof}

We next prove Proposition~\ref{prop:kconflicts-intersections} in Section~\ref{sec:higher_dim_ph}, which relates $k$-conflicts in the MCF to cluster intersection patterns more readily understood in the nerve-based MCF.\\
\begin{proof}[Proof of Proposition~\ref{prop:kconflicts-intersections}]
    Let $\sigma_1,\dots,\sigma_n\in K^t$ be the (distinct) maximal simplices of $K^t$, so that $\sigma_i\subseteq\sigma$ for $\sigma\in K^t$ implies $\sigma_i=\sigma$. In particular, if $\sigma_i\cap\sigma_j\neq \emptyset$ for $i\neq j$, then maximality implies $\sigma_i\setminus \sigma_j\neq \emptyset$ and $\sigma_j\setminus \sigma_i\neq \emptyset$, i.e., $\sigma_i$ and $\sigma_j$ are non-nested. %
    Let $C_i$ be a cluster in $\theta([t_1,t])$ such that $\sigma_i\subseteq C_i$. As $C_i\in K^t$ is itself a simplex, maximality yields $C_i=\sigma_i$, and the clusters $C_1,\dots,C_n$ are pairwise either disjoint or non-nested. With $I=\{1,\dots,n\}$ we get $K^t=\bigcup_{i\in I}\Delta C_i$, since every simplex of a finite complex is a face of a maximal one. 
    
    As the intersection $\bigcap_{i\in J}\Delta C_i$ of solid simplices for $J\subseteq I$ is either empty or given by the solid simplex $\Delta\big(\bigcap_{i\in J} C_i\big)$, which is contractible, the standard nerve theorem~\citep[Proposition~4.2]{bauerUnifiedViewFunctorial2023} yields a homotopy equivalence between $K^t$ and the nerve complex $\Nrv(\{\Delta C_i\}_{i\in I})$. The assumption $\beta_k(t)>0$ then implies there exists a non-bounding cycle $z=\tau_1+\dots+\tau_r\in Z_k[\Nrv(\{\Delta C_i\}_{i\in I})]$ with $[z]\neq 0$. In particular, there are $i_0,\dots,i_{k}\in I$ with $\tau_1=[i_0,\dots,i_{k}]$ and this means $\bigcap_{j=0}^{k}C_{i_j}\neq \emptyset$, following the definition of the nerve complex (see Appendix~\ref{background:complex}). The statement follows as any pair of clusters in $\{C_{i_0},\dots,C_{i_{k}}\}$ has non-empty intersection and is thus mutually non-nested.
\end{proof}

\subsection{Proofs for Section~\ref{Sec:NerveComplex}}\label{appx:proofs_nerve}

We use a functorial nerve theorem for simplicial covers by \citet{bauerUnifiedViewFunctorial2023} to prove Proposition~\ref{prop:equiv_MCF_nerve-based MCF} about the equivalence of the PH of the MCF and the nerve-based MCF.\\

\begin{proof}[Proof of Proposition~\ref{prop:equiv_MCF_nerve-based MCF}]
    It is sufficient to consider the case $t=t_m$, $m=1,\dots,M$, see Proposition~\ref{prop:MCF_is_filtration}.
    Define $m'\in\{m,\dots,M\}$ such that
    $t+p\in[t_{m'},t_{m'+1})$ if $m'<M$ or $t+p\ge t_M$ if $m'=M$ and consider the multi-index sets $A(m)$ and $A(m')$ from Definition~\ref{def:nerve-based MCF}. Then the corresponding families of subcomplexes $\mathcal{A}(m):=(\Delta C_\alpha)_{\alpha\in A(m)}$ and $\mathcal{A}(m'):=(\Delta C_\alpha)_{\alpha\in A(m')}$ are covers of $K^{t_m}$ and $K^{t_{m'}}$ by Definition~\ref{def:MCF_Construction}, respectively. %
    Both covers $\mathcal A(m)$ and $\mathcal A(m')$ are indeed \emph{good} covers in the sense of \citet{bauerUnifiedViewFunctorial2023} because the non-empty intersection $\bigcap_{\alpha\in S}\Delta C_\alpha=\Delta\big(\bigcap_{\alpha\in S}C_\alpha\big)$ of solid simplices for $S\subseteq A(m')$ is a solid simplex and thus contractible. Also note $\Nrv(\mathcal{A}(m))=\tilde K^{t_m}$ and $\Nrv(\mathcal{A}(m'))=\tilde K^{t_{m'}}$.
    
    The inclusion $|K^{t_m}|\hookrightarrow|K^{t_{m'}}|$ between the geometric realisations of $K^{t_m}$ and $K^{t_{m'}}$ together with the inclusion of index sets $\varphi:A(m)\hookrightarrow A(m')$ is a \emph{morphism of covered spaces} in the sense of \citet[Definition~1.2]{bauerUnifiedViewFunctorial2023}, because $\Delta C_\alpha\subseteq\Delta C_{\varphi(\alpha)}$ for all $\alpha\in A(m)$, and $\varphi$ induces the canonical inclusion $\tilde K^{t_m}\hookrightarrow\tilde K^{t_{m'}}$ on the nerve complex. This allows us to apply the \emph{functorial nerve theorem for simplicial covers}~\citep[Theorem~4.8]{bauerUnifiedViewFunctorial2023}, which associates to each cover an intermediate complex (a so-called \emph{blowup complex}), here denoted $\mathrm{Blowup}(\mathcal{A})$ and $\mathrm{Blowup}(\mathcal{A}')$, together with homotopy equivalences $\mathrm{Blowup}(\mathcal{A})\to|K^{t_m}|$, $\mathrm{Blowup}(\mathcal{A})\to|\tilde K^{t_m}|$ and $\mathrm{Blowup}(\mathcal{A}')\to|K^{t_{m'}}|$, $\mathrm{Blowup}(\mathcal{A}')\to|\tilde K^{t_{m'}}|$ that commute with the maps induced by the above morphism of covered spaces \citep[Diagram~1.5]{bauerUnifiedViewFunctorial2023}. Passing to homology thus gives the commutative diagram
    \[
    \begin{tikzcd}[column sep=small]
    H_k(K^{t_m})\arrow[d] & H_k(\mathrm{Blowup}(\mathcal{A}))\arrow[l,"\cong"']\arrow[r,"\cong"]\arrow[d]
      & H_k(\tilde K^{t_m})\arrow[d]\\
    H_k(K^{t_{m'}}) & H_k(\mathrm{Blowup}(\mathcal{A}'))\arrow[l,"\cong"]\arrow[r,"\cong"'] & H_k(\tilde K^{t_{m'}}),
    \end{tikzcd}
    \]
    where vertical maps are induced by canonical inclusions and horizontal maps are isomorphisms. This implies $H^p_k(\tilde K^{t_m})\cong H^p_k(K^{t_m})$. %
    \end{proof}

\subsection{Proofs for Section~\ref{sec:CAG_hierarchical}}\label{appx:proofs_hierarchical}

Next, we provide a full proof of Corollary~\ref{cor:MCF_SLHC} about the MCF stability in the special case of singlelinkage hierarchical clustering.\\

\begin{proof}[Proof of Corollary~\ref{cor:MCF_SLHC}] 
    Single-linkage hierarchical clustering produces hierarchical sequences of partitions $\theta$ and $\eta$ so that the pairwise first-merge times $D_\theta$ and $D_{\eta}$ %
    as defined in Equation~\eqref{eq:joint-first-merge-times} are ultrametrics. The second inequality thus follows from the stability result for single-linkage hierarchical clustering by \citet[Proposition 26]{carlssonCharacterizationStabilityConvergence2010}, because $D_\theta$ and $D_{\eta}$ are equivalent to the ultrametrics associated to the dendrograms of the sequences of partitions obtained from single-linkage hierarchical clustering.

    Furthermore, one can interpret the clique complex filtrations $\mathcal{L}_\theta$ and $\mathcal{L}_\eta$ derived from $D_\theta$ and $D_{\eta}$ as Vietoris-Rips filtrations defined on the finite ultrametric spaces $(X,D_\theta)$ and $(Y,D_{\eta})$, respectively. Following a standard stability result for the Vietoris-Rips filtration by \citet[Theorem 3.1]{chazalGromovHausdorffStableSignatures2009}, we get
    \begin{equation*}    \frac{1}{2} d_{W,\infty}\left(\Dgm_0(\mathcal{L}_\theta),\Dgm_0(\mathcal{L}_\eta)\right) \; \le\; d_\GH ((X,D_\theta),(Y,D_\eta)).
    \end{equation*}
    The first inequality then follows because Corollary~\ref{cor:hierarchical_cag_equivalent} implies that the persistence diagrams for the MCF and the clique complex filtration of the CAG are the same.
\end{proof}

\subsection{Proofs for Section~\ref{sec:zigzag_mcf}}\label{appx:proofs_zigzag}

We prove Proposition~\ref{prop:mcf_zigzag_homology} about the relation between the homology groups of the zigzag MCF $\mathcal{Z}_\theta$ and the Sankey diagram $S(\theta)$.\\

\begin{proof}[Proof of Proposition~\ref{prop:mcf_zigzag_homology}]
    We use the following three facts: (i) $Z^m=\biguplus_{C\in V_m}\Delta C$ is a disjoint union of solid simplices and its connected components correspond to $V_m$; (ii) the intersection $Z^m\cap Z^{m+1}=\biguplus_{(C,C')\in E_m}\Delta(C\cap C')$ is non-empty and its components correspond to $E_m$; %
    and (iii) $\Delta C$ and $\Delta C'$ have non-empty intersection iff $(C,C')\in E_m$, so two elements lie in the same component of $Z^{m,m+1}$ iff their corresponding clusters are joined by a path in $S_m$, in particular $\dim H_0(Z^{m,m+1})=\dim H_0(S_m)$. 
    For our proof we treat the dimensions $k$ separately.

    $k\ge 2$: By fact (i), we have $H_k(Z^m)=H_k(Z^{m+1})=0$ for $k\ge 2$. By fact (ii), $Z^m\cap Z^{m+1}$ is also a disjoint union of solid simplices and so $H_{k-1}(Z^m\cap Z^{m+1})=0$ for $k\ge 2$. So the Mayer-Vietoris sequence~\citep{hatcherAlgebraicTopology2002}
    \[
   0=H_k(Z^m)\oplus H_k(Z^{m+1})\to H_{k}(Z^{m,m+1})\to H_{k-1}(Z^m\cap Z^{m+1})=0
    \]
    gives $H_{k}(Z^{m,m+1})=0$ for $k\ge 2$.
    
    $k = 1$: By fact (i) we have $H_1(Z^m)=H_1(Z^{m+1})=0$; hence, the Mayer-Vietoris sequence given by
    \[
    0\to H_1(Z^{m,m+1})\to H_0(Z^m\cap Z^{m+1})\to
    H_0(Z^m)\oplus H_0(Z^{m+1})\to H_0(Z^{m,m+1})\to0
    \]
    is exact. With (i) and (ii) the alternating sum of dimensions gives $\dim H_1(Z^{m,m+1})=\lvert E_m\rvert-\lvert V_m\rvert-\lvert V_{m+1}\rvert+\dim H_0(Z^{m,m+1})$. Moreover, by (iii) we have $\dim H_0(Z^{m,m+1})=\dim H_0(S_m)$ so that $\dim H_1(Z^{m,m+1})=\dim H_1(S_m)$, where the last equation uses the formula for the cyclomatic number of a graph. 
    
    $k=0$: By fact (i) we have $\dim H_0(Z^m)=\lvert V_m\rvert$, and by (iii) $H_0(Z^{m,m+1})$ has the connected components of $S_m$ as a basis, i.e., $H_0(Z^{m,m+1})\cong H_0(S_m)$.
\end{proof}

Finally, we provide the proof of Proposition~\ref{prop:zigzag_sankey} about the isomorphism between zigzag PHs.\\

\begin{proof}[Proof of Proposition~\ref{prop:zigzag_sankey}]
Let $\varphi_m$, $m=1,\dots,M$ be the bijections from Definition~\ref{def:sankey_isomorphism} that constitute the Sankey diagram isomorphism between $S(\theta)$ and $S(\eta)$. Furthermore, denote by $S_m(\theta)$ and $S_m(\eta)$ the $m$-th layers in the Sankey diagrams $S(\theta)$ and $S(\eta)$, respectively.
For our proof we treat again the dimensions $k$ separately.
    
    $k\ge 2$: By Proposition~\ref{prop:mcf_zigzag_homology} $H_k(\mathcal Z_\theta)$ and $H_k(\mathcal Z_\eta)$ are both trivial.

    $k=1$: By Proposition~\ref{prop:mcf_zigzag_homology} $H_1(Z^m_\theta)=H_1(Z^m_\eta)=0$ for all $m=1,\dots,M$ and so all maps in the zigzag persistence modules vanish. Hence, $H_1(\mathcal Z_\theta)$ and $H_1(\mathcal Z_\eta)$ are isomorphic when $\dim H_1(Z^{m,m+1}_\theta)=\dim H_1(Z^{m,m+1}_\eta)$ for all $m=1,\dots,M-1$. 
    This is true because the Sankey diagram isomorphism implies that $S_m(\theta)$ and $S_m(\eta)$ have the same cyclomatic numbers, so by Proposition~\ref{prop:mcf_zigzag_homology} we have $\dim H_1(Z^{m,m+1}_\theta)=\dim H_1(S_m(\theta))=\dim H_1(S_m(\eta))=\dim H_1(Z^{m,m+1}_\eta)$.

    $k=0$: By Proposition~\ref{prop:mcf_zigzag_homology}, $H_0(Z^m_\theta)$ has the clusters in $V_m$ as a basis and $H_0(Z_\theta^{m,m+1})$ has the connected components of $S_m(\theta)$ as a basis. The bijections $\varphi_m$ map paths in $S_m(\theta)$ to paths in $S_m(\eta)$ and vice versa, and therefore induce bijections $\psi_m$ between the components of $S_m(\theta)$ and those of $S_m(\eta)$ with $\psi_m([C])=[\varphi_\ell(C)]$ for $C\in V_\ell$, $\ell\in\{m,m+1\}$. Recall that maps induced by inclusions send a component to the component containing it; hence, the isomorphisms induced by $\varphi_m$ at the positions $Z^m$ and by $\psi_m$ at the positions $Z^{m,m+1}$ commute with all maps in Diagram~\eqref{eq:zigzag_mcf} and we get the commutative diagram:
    \[
    \begin{tikzcd}[column sep=large]
    H_0(Z^m_\theta)\arrow[r]\arrow[d,"\varphi_m"',"\cong"] &
    H_0(Z^{m,m+1}_\theta)\arrow[d,"\psi_m"',"\cong"] &
    H_0(Z^{m+1}_\theta)\arrow[l]\arrow[d,"\varphi_{m+1}"',"\cong"]\\
    H_0(Z^m_\eta)\arrow[r] & H_0(Z^{m,m+1}_\eta) & H_0(Z^{m+1}_\eta)\arrow[l]
    \end{tikzcd}
    \]
    This yields an isomorphism $H_0(\mathcal Z_\theta)\cong H_0(\mathcal Z_\eta)$ of zigzag persistence modules.
\end{proof}

\section{Details on Experiments}

Code to reproduce all our numerical experiments is available at: \url{https://github.com/barahona-research-group/MCF}

\subsection{Application to Stochastic Block Models}\label{appx:SBM_experiment}

\subsubsection{Definition of Stochastic Block Models}

Stochastic block models (SBMs) are random graph models to generate undirected, unweighted graphs with planted partition structure into different clusters (also called blocks), where the probability $P_{ij}$ of an edge between two vertices only depends on the cluster assignment~\citep{hollandStochasticBlockmodelsFirst1983, karrerStochasticBlockmodelsCommunity2011}. 

\paragraph{Single-scale SBM.} When only a single partition is used to define the SBM we call it a `single-scale' SBM. To formalise this, recall that a partition $\pi\in\Pi_V$ for a set of vertices $V=\{1,\dots,N\}$ into $c\in\mathbbm{N}$ clusters can be represented by a $N\times c$ \emph{cluster indicator matrix} $F$ where $F_{ir}=1$ if vertex $i\in V$ is part of cluster $r\in\{1,\dots,c\}$ and $F_{ir}=0$ otherwise. We further define a $c\times c$ \emph{affinity matrix} $\Omega$ such that $0\le \Omega_{rs}\le 1$ is the probability that a vertex in cluster $r$ is connected to a vertex in cluster $s$. We can then define the $N\times N$ \emph{probability matrix} $P$ as:
\begin{equation}
    P = F \Omega F^T,
\end{equation}
which allows us to generate adjacency matrices $A$ corresponding to the SBM through a Bernoulli distribution $\mathbbm{P}[A_{ij}=1]=P_{ij}$ and $\mathbbm{P}[A_{ij}=0]=1-P_{ij}$ for each edge $(i,j)\in V\times V$. %

\paragraph{Multiscale SBM.} Using single-scale SBMs as building blocks, we can construct a `multiscale SBM' with planted partitions at $L$ different scales as follows. For each scale $\ell\in\{1,\dots,L\}$ define an $N\times c(\ell)$ cluster indicator matrix $F^{(\ell)}$ %
and a $c(\ell)\times c(\ell)$ affinity matrix $\Omega^{(\ell)}$. We then define the $N\times N$ probability matrix $P$ of the multiscale SBM as the convex combination:
\begin{equation}\label{S_eq:SBM-convex-sum}
    P := \sum_{\ell=1}^L w_\ell P^{(\ell)} =  \sum_{\ell=1}^L w_\ell F^{(\ell)}\Omega^{(\ell)}(F^{(\ell)})^T,
\end{equation}
where each level $\ell$ has an associated weight $w_\ell>0$ with $\sum_{\ell=1}^{L}w_\ell=1$. Before sampling, we additionally permute the vertex labels so that they contain no information about the block structure. We can again use the probability matrix $P$ to sample graphs from the multiscale SBM and the expected adjacency matrix $\mathbbm{E}[A]$ is given by $\mathbbm{E}[A]=P$. In contrast to constructions of multiscale SBMs using binary trees~\citep{%
schaubHierarchicalCommunityStructure2023}, the construction via convex sums of probability matrices of single-scale SBMs presented here also extends to the case of non-hierarchical models with planted partition structure that is not nested across scales.

\subsubsection{Details on Model Parameters}

In our experiments, we restrict ourselves to models with equal-sized blocks, i.e., the clusters in a partition $F^{(\ell)}$ have the same size $s(\ell):=N/c(\ell)$. Moreover, we only consider affinity matrices $\Omega^{(\ell)}$ with diagonal elements $\Omega^{(\ell)}_{rr}=\alpha^{(\ell)}$ and off-diagonal elements $\Omega^{(\ell)}_{rs}=p_\mathrm{out}$, for $r\neq s$, across all scales. We define four models for $N=270$ elements each with fixed $p_\mathrm{out}=0.001$ at each scale and set the probabilities $\alpha^\ell$ for different scales $\ell$ in a way such that the expected number of edges is always 2,500. 
\begin{itemize}
    \item[(i)] Erd\H{o}s-Rényi (ER): To generate a fully random ER graph, we can use an sSBM with a single block such that the partition indicator matrix is given by the vector of ones, i.e., $F^{(1)}=\boldsymbol{1}$, and we set $\alpha^{(1)}=0.06884$.
    \item[(ii)] Single-scale SBM (sSBM): For the sSBM we define a partition indicator matrix $F^{(1)}$ with three equal-sized blocks and set $\alpha^{(1)}=0.20604$.
    \item[(iii)] Multiscale SBM (mSBM): We define an mSBM model with three hierarchical scales. At the coarse scale, the partition indicator matrix $F^{(1)}$ represents three equal-sized blocks with weight $w_1=\frac{1}{14}$. At the medium scale, $F^{(2)}$ represents nine equal-sized blocks with weight $w_2=\frac{3}{14}$. At the fine scale, $F^{(3)}$ represents 27 equal-sized blocks with weight $w_3=\frac{10}{14}$. Across all three scales we set $\alpha^{(1)}=\alpha^{(2)}=\alpha^{(3)}=0.96159$.
    \item[(iv)] Non-hierarchical multiscale SBM (nh-mSBM): We define an nh-mSBM model with three non-hierarchical scales corresponding to non-nested planted partitions. At the coarse scale, the partition indicator matrix $F^{(1)}$ represents three equal-sized blocks with weight $w_1=\frac{4}{65}$. At the medium scale, $F^{(2)}$ represents five equal-sized blocks with weight $w_2=\frac{9}{65}$. At the fine scale, $F^{(3)}$ represents 27 equal-sized blocks with weight $w_3=\frac{52}{65}$. Across all three scales we set $\alpha^{(1)}=\alpha^{(2)}=\alpha^{(3)}=0.91278$.
\end{itemize}

Figure~\ref{fig:sbms_summary} (first column) shows one adjacency matrix sampled from each of the models defined above (with vertices permuted so that the multiscale structure becomes visible). To account for stochasticity, we generate an ensemble of 200 graphs from each of the different models leading to 800 adjacency matrices $A_i$, $i=1,\dots,800$.

\subsubsection{Details on Multiscale Clustering with Markov Stability}\label{appendix:sbm_ms_clustering}

We then apply multiscale clustering with Markov Stability (MS)~\citep{lambiotteLaplacianDynamicsMultiscale2009,delvenneStabilityGraphCommunities2010, lambiotteRandomWalksMarkov2014,schaubMarkovDynamicsZooming2012} using the \texttt{PyGenStability} Python package~\citep{arnaudonAlgorithm1044PyGenStability2024} and obtain non-hierarchical sequences of partitions $\theta_i:[t_1,\infty)\rightarrow \Pi_X,t\mapsto\theta_i(t)$ for each adjacency matrix $A_i$, $i=1,\dots,800$. The sequences of partitions $\theta_i$ are indexed over the $\log_{10}$ of the continuous \emph{Markov time} (denoted $t$), evaluated at 200 equidistant scales between $t_1=-1.5$ to $t_{200}=0.5$. This leads to four ensembles of multiscale clusterings $\theta_i$ where indices $i=1,\dots,200$ correspond to ER, $i=201,\dots,400$ to sSBM, $i=401,\dots,600$ to mSBM and $i=601,\dots,800$ to nh-mSBM. From each ensemble, we visualise one sequence of partitions as a Sankey diagram in Figure~\ref{fig:sbms_summary} (second column). We observe that the sequences of partitions get coarser with increasing scale and that they inherit a quasi-hierarchical nature for the sSBM, mSBM and nh-mSBM models reflecting the underlying (multiscale) planted partition structure.

\begin{figure}[htb!]
    \centering
    \includegraphics[width=0.8\textwidth]{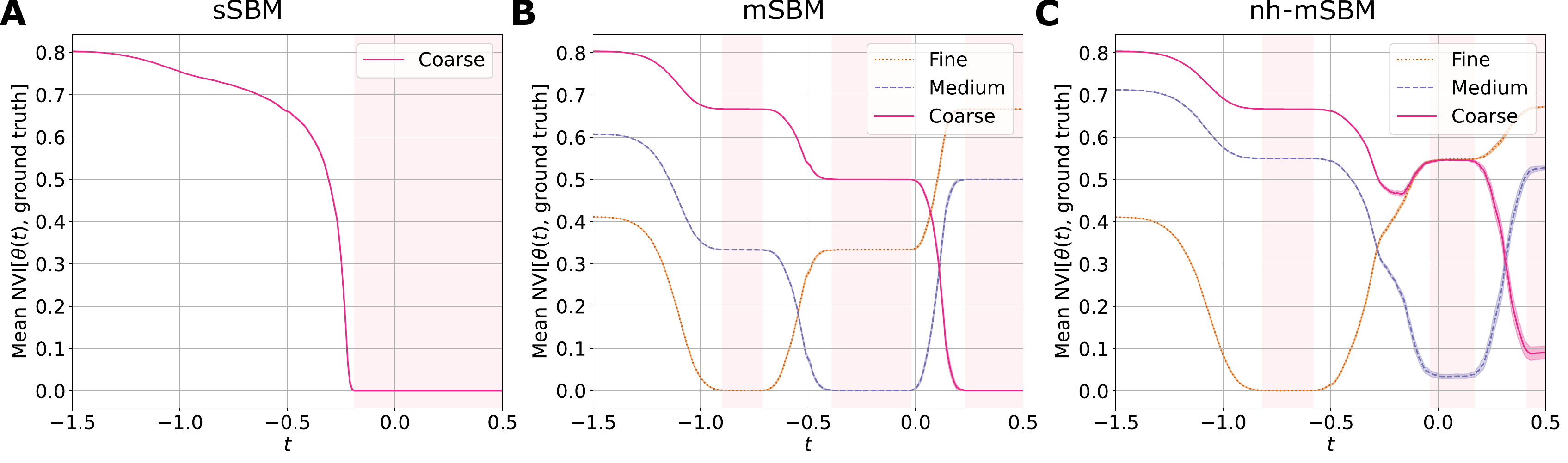}
    \vspace{-0.6cm}
    \caption{{Comparing the sequences of partitions and ground-truth partitions.} {We compute the average NVI with 95\% confidence intervals between all partitions in the sequences $\theta_i$ and ground-truth planted partitions of the different models. Low values of NVI indicate that we recover the single ground-truth partition of sSBM and the three ground-truth partitions of mSBM and nh-mSBM (shaded in pink).}}
    \label{S_fig:ensembles_nvi}
    \vspace{-0.2cm}
\end{figure}

For the three models with planted partitions (sSBM, mSBM and nh-mSBM), we compute the NVI between $\theta_i(t)$, $t\ge t_1$, and the ground-truth partitions at different scales averaged over all the sequences in the ensemble and we also compute 95\% confidence intervals. %
We find that the sequences of partitions retrieved from sSBM, mSBM and nh-mSBM recover the ground-truth partitions at different scales (highlighted in pink), see Figure~\ref{S_fig:ensembles_nvi}.

\begin{figure}[!htb]
    \centering
    \includegraphics[width=\textwidth]{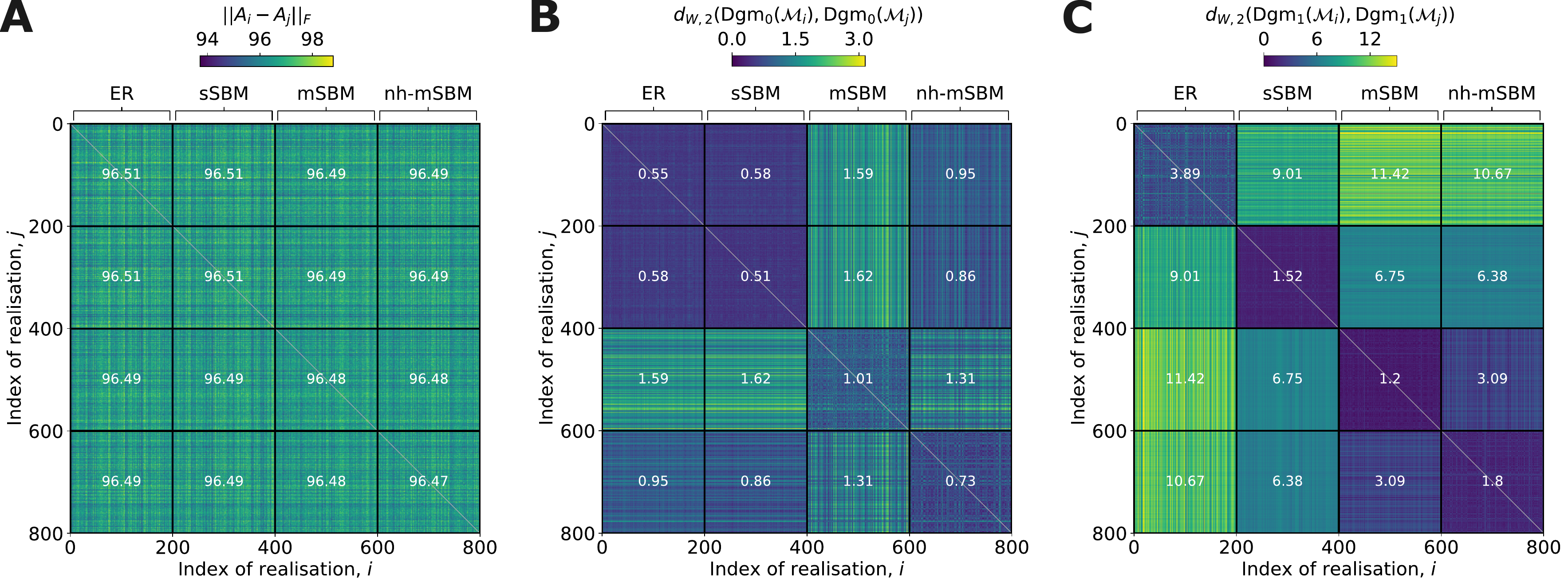}
    \vspace{-0.2cm}
    \caption{%
    {Pairwise distances between all $i=1,\dots,800$ model realisations of the four ensembles (ER, sSBM, mSBM and nh-mSBM):  
    (\textbf{A}) Frobenius distance of adjacency matrices $\|A_i-A_j\|_F$; (\textbf{B}) 2-Wasserstein distance of the 0-dimensional PDs 
    $d_{W,2}(\Dgm_0(\mathcal{M}_i, \Dgm_0(\mathcal{M}_j)$;  (\textbf{C}) 2-Wasserstein distance of the 1-dimensional PDs $d_{W,2}(\Dgm_1(\mathcal{M}_i, \Dgm_1(\mathcal{M}_j)$. 
    Whereas the Frobenius distance (\textbf{A}) is not able to distinguish the models, the 2-Wasserstein distance of 0- and 1-dimensional PDs (\textbf{B} and \textbf{C}) distinguishes the models based on their different levels of hierarchy and multiscale structure. %
    }}
    \label{fig:sbm_model_comparison}
    \vspace{-0.3cm}
\end{figure}

\subsubsection{Details on MCF Computation in SBM Experiment}\label{appx:sbm_compute_times}

First we present the compute times of MCF on the MS data. We measured the time for the construction of the filtration and computation of PH separately for both the standard and nerve-based MCF. To account for stochasticity we average the timing results over 10 realisations for each of the four models (ER, sSBM, mSBM and nh-mSBM).

\begin{table}[!htb]
\centering
\footnotesize
\begin{tabular}{lrrrr}
\toprule
 & ER & sSBM & mSBM & nh-mSBM \\
\midrule
Standard [Filtration] & 0.47 (0.35--0.58) & 9.31 (9.05--9.56) & 8.80 (8.71--8.89) & 11.76 (10.88--12.65) \\
Standard [PH] & 0.36 (0.26--0.45) & 6.66 (6.53--6.78) & 5.92 (5.80--6.03) & 6.72 (6.52--6.92) \\
\midrule
Nerve [Filtration] & 5.50 (4.45--6.55) & 5.35 (4.48--6.22) & 0.37 (0.34--0.40) & 0.80 (0.70--0.91) \\
Nerve [PH] & 1.33 (1.00--1.67) & 1.67 (1.34--2.01) & 0.050 (0.047--0.054) & 0.16 (0.13--0.20) \\
\bottomrule
\end{tabular}
\caption{Time (in seconds) for standard and nerve-based MCF computations (split into filtration construction and PH computation times) for the SBM datasets. Times are averaged over 10 realisations of each model with 95\% confidence intervals.}
\label{tab:timing_results_ms}
\end{table}

Next we present the pairwise Frobenius distances between adjacency matrices (Figure~\ref{fig:sbm_model_comparison}\textbf{A}) and pairwise 2-Wasserstein distances between 0- and 1-dimensional MCF PDs (Figure~\ref{fig:sbm_model_comparison}\textbf{B}--\textbf{C}) for all 800 model realisations.

\subsubsection{Results for Modularity Clustering}\label{appx:sbm_modularity}

Here we present the results of MCF applied to multiscale clusterings computed with Modularity with a resolution $\gamma$~\citep{reichardtStatisticalMechanicsCommunity2006} for the four different SBMs. For better comparability we used the same range of log-scales $t\in[-1.5,0.5]$ for both MS and Modularity, where $t:=\log_{10}(1/\gamma)$ in the case of Modularity.

\begin{figure}[htb!]
    \centering
    \includegraphics[width=\textwidth]{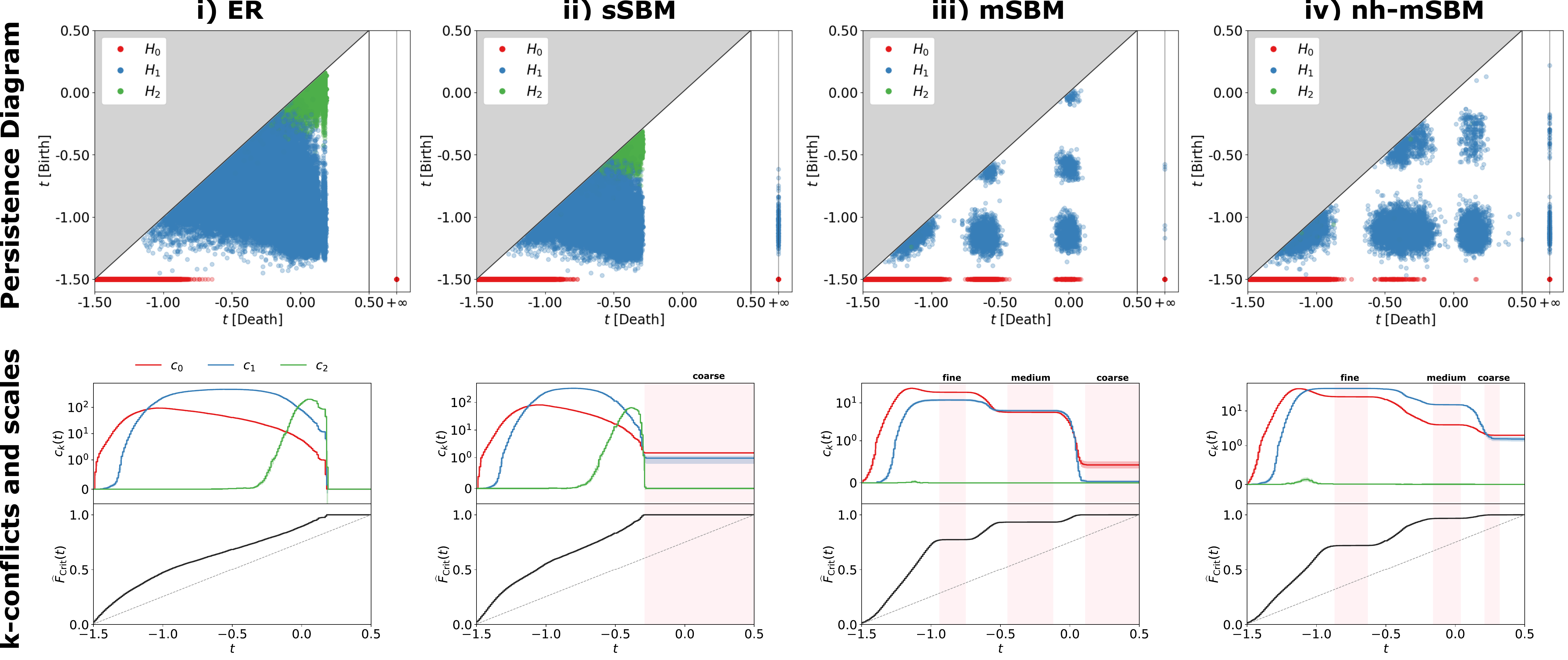}
    \vspace{-0.2cm}
    \caption{Summary figure of MCF applied to Modularity multiscale clusterings of our four models with panels matching those in Figure~\ref{fig:sbms_summary}. Results are consistent with the difference that the ER sequence of partitions computed with Modularity reaches $\hat 1$ already around $t=0.2$, so all conflicts get resolved early.}
    \label{fig:sbm_pds_conflicts_modularity}
    \vspace{-0.3cm}
\end{figure}

Figure~\ref{fig:sbm_pds_conflicts_modularity} shows that the MCF results for Modularity are consistent with those for MS (Figure~\ref{fig:sbms_summary}). This is expected because \cite{lambiotteLaplacianDynamicsMultiscale2009} showed that Modularity is mathematically equivalent to a linearised version of MS.

\subsection{Application to Synthetic Single-Cell Data}

\subsubsection{Details on Data Simulation}\label{appx:sc_synth_data}

We consider six differentiation trees inspired by known single-cell systems but simplified so that each is a binary tree with 12 leaves and height 6. We use the Newick notation to specify the trees below. A visualisation of each tree is provided in Figure~\ref{fig:sc_synthetic_all_trees}. 

\begin{figure}[htb!]
    \centering
    \includegraphics[width=\textwidth]{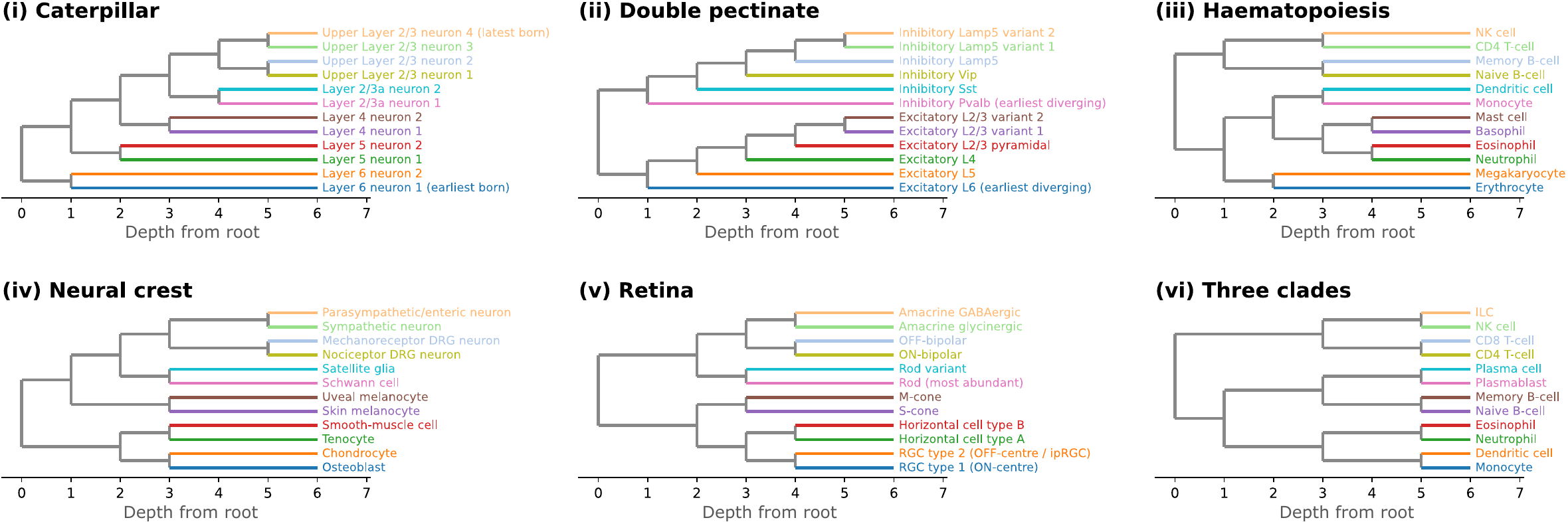}
    \vspace{-0.2cm}
    \caption{Visualisation of the six biologically inspired phylogenetic trees used in our \texttt{scMultiSim} single-cell simulations.}
    \label{fig:sc_synthetic_all_trees}
    \vspace{-0.3cm}
\end{figure}

\begin{enumerate}[(i)]
    \item Caterpillar: A tree in which one pair splits off at each of six successive heights, mirroring the order in which cortical layers are generated, from the deepest to the most superficial
    ~\citep{molyneauxNeuronalSubtypeSpecification2007}. \\
    \texttt{\scriptsize ((A:5,B:5):1,((C:4,D:4):1,((E:3,F:3):1,((G:2,H:2):1,((I:1,J:1):1,(K:1,L:1):1):1):1):1):1);}
    \item Double pectinate: Two `caterpillars' representing two classes of cortical neurons %
    specialising in parallel~\citep{zeiselCellTypesMouse2015}.\\
    \texttt{\scriptsize ((A:5,(B:4,(C:3,(D:2,(E:1,F:1):1):1):1):1):1,(G:5,(H:4,(I:3,(J:2,(K:1,L:1):1):1):1):1):1);}
    \item Haematopoiesis: %
    An unbalanced tree in which the red blood cell and platelet lineages separate from the other blood cell types before those diversify following 
    the standard hierarchy of human blood formation
    ~\citep{doulatovHematopoiesisHumanPerspective2012}.\\
    \texttt{\scriptsize(((A:4,B:4):1,(((C:2,D:2):1,(E:2,F:2):1):1,(G:3,H:3):1):1):1,((I:3,J:3):2,(K:3,L:3):2):1);}
    \item Neural crest: A three-fate tree into skeletal and connective tissue, pigment cells, and nerve cells, %
    with the skeletal fate separating first and the nerve fate last~\citep{soldatovSpatiotemporalStructureCell2019}.\\
    \texttt{\scriptsize (((A:3,B:3):1,(C:3,D:3):1):2,((E:3,F:3):2,((G:3,H:3):1,((I:1,J:1):2,(K:1,L:1):2):1):1):1);}
    \item Retina: A two-fate tree encoding the developing retina's switch from early cell types 
    (before birth) to a late set~\citep{clarkSingleCellRNASeqAnalysis2019}.\\
    \texttt{\scriptsize((((A:2,B:2):1,(C:2,D:2):1):1,(E:3,F:3):1):2,((G:3,H:3):1,((I:2,J:2):1,(K:2,L:2):1):1):2);}
    \item Three clades: A tree with the three main branches of the immune system: myeloid, B and T cells~\citep{dominguezcondeCrosstissueImmuneCell2022}.\\
    \texttt{\scriptsize ((((A:1,B:1):2,(C:1,D:1):2):2,((E:1,F:1):2,(G:1,H:1):2):2):1,((I:1,J:1):2,(K:1,L:1):2):3);}
\end{enumerate}

For each of the six trees, we simulate single-cell data with \texttt{scMultiSim} using the following parameters~\citep{liScMultiSimSimulationSinglecell2025}: $N=500$ cells spread as evenly as possible over the 12 tips of the tree using the discrete cell population model, 200 genes, 20 discrete CIFs of which 18 are differential between cell types, intrinsic noise weight $\sigma_i=1$, no gene regulatory network, and \emph{gene identity vectors} (GIVs) redrawn independently for each simulation. We keep the true counts without adding technical noise. We choose 31 different values for the standard deviation of the CIF $\sigma_\mathrm{cif}\in[0,1.5]$, which controls the deviation of an individual cell from its cell type. For each tree and $\sigma_\mathrm{cif}$ combination we then simulate 27 different single-cell datasets for different seeds.

Note that $\sigma_\mathrm{cif}$ alone does not determine how well separated the cell types are in a given dataset because the variance in the gene expression space also depends on GIV~\citep{liScMultiSimSimulationSinglecell2025}. Hence, two datasets with the same $\sigma_\mathrm{cif}$ but different GIVs can differ in cluster separation. %
We therefore regress on the FVU (Equation~\ref{eq:fvu_ss}), which is measured from the data itself, and treat $\sigma_\mathrm{cif}$ only as the simulation parameter. %

\subsubsection{Details on Regression Task}\label{appx:sc_regression_task}

Let $Z_i\in\mathbbm{R}^{30}$ be the 30-dimensional cell PCA coordinates, which were used for clustering with \texttt{SC3s}~\citep{quahSC3sEfficientScaling2022}, and $y_i\in\{1,\dots,12\}$ the ground-truth cell type. The \emph{fraction of variance unexplained} (FVU) is then defined as:
\begin{equation}\label{eq:fvu_ss}
    \mathrm{FVU}=1-\frac{SS_{\mathrm{between}}}{SS_{\mathrm{total}}}, \quad SS_{\mathrm{between}} = \sum_k  N_k \|\mu_k - \mu\|^2,  \quad SS_{\mathrm{total}} =  \sum_{i=1}^N \|Z_i - \mu\|^2,
\end{equation}
where $\mu \;=\; \frac{1}{N}\sum_{i=1}^N Z_i$ is the global centroid. $\mu_k \;=\; \frac{1}{N_k}\sum_{i=1}^N \mathbf{1}\{y_i = k\} Z_i$ is the per-cell-type centroid and $N_k$ is the number of cells of type $k$. FVU is normalised between 0 and 1, where FVU $=0$ corresponds to zero within-cell-type variance and FVU $=1$ corresponds to zero cell-type separation.

For our FVU regression task, we train Ridge regression and MLP models for each of our seven feature maps. We perform hyperparameter optimisation with the
Tree-Structured Parzen Estimator (TPE)~\citep{bergstraAlgorithmsHyperParameterOptimization2011} over 100 trials for the penalty term of Ridge (ranging from $0.001$ to $10^7$), and over 60 trials for the number of hidden layers (1 or 2), hidden layer size (8, 16, 32, 64, 128 or 256), learning rate (ranging from $0.0001$ to $0.01$) and $L^2$-penalty term (ranging from $10^{-5}$ to $0.1$) of the MLP model. We use the train split of our data for training and the validation split for evaluation and hyperparameter selection based on the mean absolute error (MAE). The best hyperparameters for each model and feature map combination are reported in Table~\ref{tab:fvu_hyperparams}. Train $R^2$ scores are reported in Table~\ref{tab:fvu_train_r2}. After hyperparameter optimisation, each model is refit on the full training split (train and validation combined) with the chosen hyperparameters and evaluated on the held-out test set (Table~\ref{tab:fvu_test_r2}).

\begin{table}[!htb]
\centering
\begin{tabular}{lcccc}
\toprule
 & Ridge & \multicolumn{3}{c}{MLP} \\
\cmidrule(lr){2-2}\cmidrule(lr){3-5}
Features & penalty term & hidden layers & penalty term & learning rate \\
\midrule
PCA & $2.65\times 10^{5}$ & (128,128) & $0.075$ & $8.25\times 10^{-3}$ \\
Raw label encoding & $9.66\times 10^{3}$ & (256,256) & $0.0994$ & $1.37\times 10^{-3}$ \\
Hamming & $238$ & (64,64) & $1.06\times 10^{-5}$ & $2.47\times 10^{-3}$ \\
ARI & $21.7$ & (64,64) & $0.0359$ & $9.98\times 10^{-3}$ \\
CE & $16.5$ & (128,128) & $0.0301$ & $5.38\times 10^{-3}$ \\
NVI & $8.41$ & (256,256) & $0.0538$ & $7.88\times 10^{-4}$ \\
MCF & $9.94$ & (32,32) & $0.063$ & $7.50\times 10^{-3}$ \\
\bottomrule
\end{tabular}
\caption{Optimal hyperparameters for different model and feature combinations in the FVU regression task.}
\label{tab:fvu_hyperparams}
\end{table}

\begin{table}[!htb]
\centering
\begin{tabular}{lccccccc}
\toprule
 & PCA & Raw label encoding & Hamming & ARI & CE & NVI & MCF \\
\midrule
Ridge & 0.070 & 0.555 & 0.690 & 0.829 & 0.814 & 0.864 & 0.934 \\
MLP & 0.957 & 0.954 & 0.737 & 0.893 & 0.919 & 0.932 & 0.951 \\
\bottomrule
\end{tabular}
\caption{Train $R^2$ scores for FVU regression task.}
\label{tab:fvu_train_r2}
\end{table}

\subsection{Application to Real-World Single-Cell Data}\label{appx:real_single_cell}

The `Baron' dataset collected by \citet{baronSingleCellTranscriptomicMap2016} is available from the Bioconductor \texttt{scRNAseq} package~\citep{amezquitaOrchestratingSinglecellAnalysis2020} and deposited under GEO accession GSE84133 (\texttt{BaronPancreasData}). The `Paul' dataset collected by \citet{paulTranscriptionalHeterogeneityLineage2015} is also available from \texttt{scRNAseq} and deposited under GEO accession GSE72857 (\texttt{PaulHSCData}). 

For single-cell clustering of both datasets we use the \texttt{Seurat} pipeline~\citep{haoIntegratedAnalysisMultimodal2021} with the Louvain algorithm~\citep{blondelFastUnfoldingCommunities2008}.
We choose $M=70$ Modularity resolutions (for $\gamma \in [10^{-4}, 1.80]$) by inverting the empirical relationship between resolution and number of clusters, so that the number of clusters decreases approximately linearly from 25 to 1. We then index partitions by their scale rank $t_m=m$, $m=1,\dots, M$, from $|\theta(t_1)|\approx 25$ to $|\theta(t_{70})|=1$. Equal steps in $t$ thus correspond to roughly equal changes in the number of clusters, whereas indexing by the non-linear resolution parameter $\gamma$ directly would give coarse scales most of the filtration length in the MCF.

We compare the partitions in the sequences of partitions obtained with \texttt{Seurat} for Baron with the ground-truth cell types using the NVI, see Figure~\ref{fig:sc_real_nvi}.

\begin{figure}[htb!]
    \centering
    \includegraphics[width=0.5\textwidth]{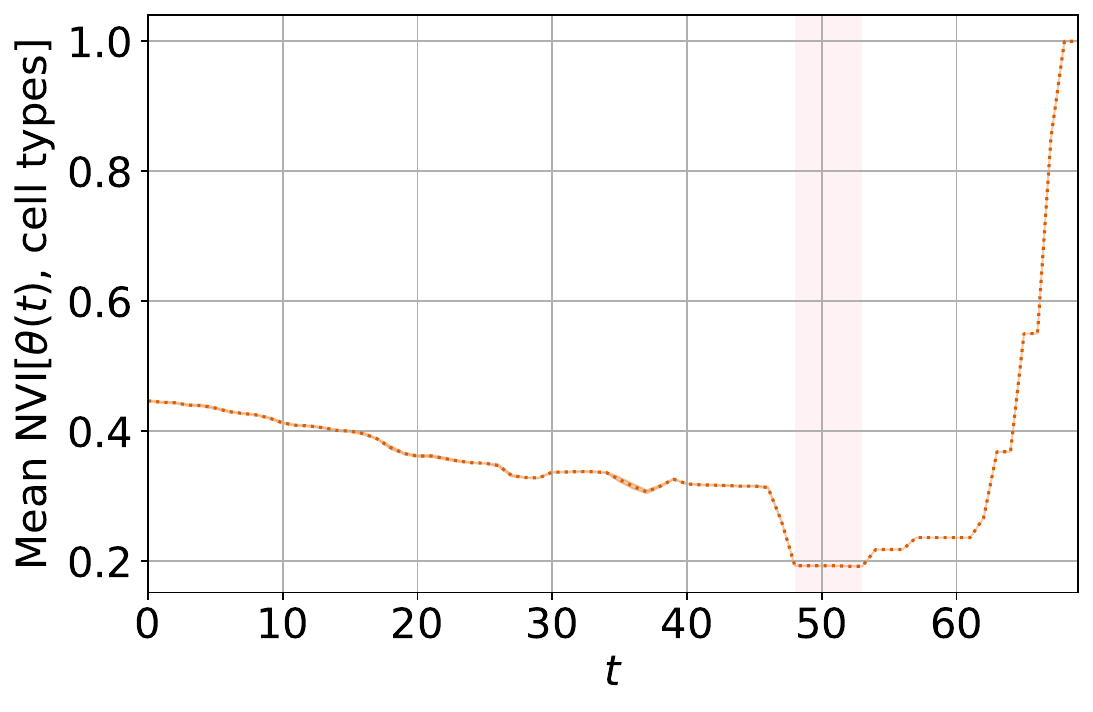}
    \vspace{-0.2cm}
    \caption{NVI between $\theta(t)$ and ground-truth cell types for Baron averaged over different realisations of $\theta$ with 95\% CIs. We shade scales $t$ pink where we recover the ground truth well (NVI within a 0.06 tolerance of the minimum).}
    \label{fig:sc_real_nvi}
    \vspace{-0.3cm}
\end{figure}

\subsection{Application to Real-World Temporal Interaction Data}\label{appx:primary_school}

The primary school dataset of temporal social interactions is available via the SocioPatterns project~\citep{gemmettoMitigationInfectiousDisease2014} at: \url{https://sociopatterns.org/datasets/primary-school-temporal-network-data/}

Let $X$ be the set of $N=242$ individuals featured in the primary school dataset comprising 232 students divided into 10 classes and 10 teachers. For a cluster $C\subseteq X$ we can express its class decomposition as a probability vector $p_C(i)=\lvert\{x\in C:\; x\; \text{is in class}\; i\}\rvert / \lvert C\rvert$ for classes $i=1,\dots,11$, where we treated the teachers as an 11th class for simplicity following \citet{stehleHighResolutionMeasurementsFacetoFace2011}. We can then compute the class concordance  $\kappa(C,C')$ between two clusters $C,C'\subseteq X$ as the inner product $0\le \kappa(C,C'):=p_C\cdot p_{C'} \le 1$, which is 1 when the individuals from both clusters are all in the same class and 0 when they share no class. Note that $\kappa(C,C')$ is the probability that a randomly selected individual from $C$ is in the same class as a randomly selected individual from $C'$.

We can use the class concordance measure to decompose our average 1-conflict $\bar c_1$ measure into the contribution of 1-conflicts due to higher-order inconsistencies within classes versus across classes. Recall Equation~\eqref{eq:total_persistence} that expressed the average 1-conflict as $\bar c_1=\frac{1}{T-t_1}\sum _{s < t}\mu^{s,t}_1 (\min(t,T)-s)$ and note that we can alternatively write $\sum _{s < t}\mu^{s,t}_1 (\min(t,T)-s) = \sum_\alpha l_\alpha$, where the sum runs over all 1-conflicts $\alpha$ and $l_\alpha$ their truncated lifetime. Each 1-conflict $\alpha$ corresponds to a non-bounding cycle in the MCF and as part of the PH computations, \texttt{GUDHI}~\citep{mariaGudhiLibrarySimplicial2014} gives us a 1-dimensional simplex $\sigma$ that closed the cycle. While this simplex is not unique we can use it for a statistical characterisation. As we use the nerve-based MCF for our computations, $\sigma$ corresponds to the intersection $C\cap C'$ of two clusters $C, C'\subseteq X$. Hence, we can compute the class concordance $\kappa_\alpha=\kappa(C,C')$ for each 1-conflict $\alpha$ and decompose $\bar c_1$ into a within-class and a cross-class term as follows:
\begin{equation}\label{eq:c_1_class_decomposition}
    \bar c_1 = \bar c_1^{\mathrm{within}} + \bar c_1^{\mathrm{cross}}, \quad \bar c_1^{\mathrm{within}}:=\frac{1}{T-t_1}\sum_\alpha l_\alpha\;\kappa_\alpha, \quad \bar c_1^{\mathrm{cross}}:=\frac{1}{T-t_1}\sum_\alpha l_\alpha\;(1-\kappa_\alpha),
\end{equation}
where we attribute a fraction $\kappa_\alpha$ of each conflict's lifetime to the within-class term and the remainder
to the cross-class term. Equation~\eqref{eq:c_1_class_decomposition} is then the
expected decomposition obtained by resolving each 1-conflict to a single pair of
individuals, one sampled from each of its two clusters.

We also compare for each window of 20 minutes the mean class concordance of cluster pairs that generate 1-conflicts against the mean class concordance of arbitrary intersecting cluster pairs during that window. We test the resulting per-window differences across non-overlapping windows using a one-sided Wilcoxon signed-rank test.

\subsection{Computational Efficiency of MCF}\label{appx:computational_complexity}

We present in Table~\ref{tab:mcf_empirical_scale} the empirical scaling of the MCF compute time on synthetic non-hierarchical sequences of partitions with increasing number of elements $N$ (equal here to the maximum cluster size) and number of change points $M$.

\begin{table}[!htb]
  \centering
  
  \begin{tabular}{llcccc}
    \toprule
    Method & Phase & $a$ (exp.\ of $N$) & $b$ (exp.\ of $M$) & $R^2$ & $n$ \\
    \midrule
    \multirow{3}{*}{Standard MCF} & Filtration & $2.49 \pm 0.08$ & $0.68 \pm 0.04$ & $0.969$ & $34$ \\
     & PH & $3.98 \pm 0.24$ & $0.74 \pm 0.13$ & $0.910$ & $31$ \\
     & Total & $3.32 \pm 0.11$ & $0.86 \pm 0.06$ & $0.971$ & $31$ \\
    \midrule
    \multirow{3}{*}{Nerve-based MCF} & Filtration & $1.66 \pm 0.03$ & $2.76 \pm 0.06$ & $0.996$ & $19$ \\
     & PH & $1.83 \pm 0.14$ & $3.48 \pm 0.26$ & $0.939$ & $19$ \\
     & Total & $1.75 \pm 0.04$ & $2.99 \pm 0.07$ & $0.994$ & $19$ \\
    \bottomrule
  \end{tabular}
    \caption{Empirical scaling of the MCF compute time, fitted jointly across the $(N,M)$ grid as $t \propto N^{a} M^{b}$ (ordinary least squares of $\log_{10} t$ on $\log_{10} N$ and $\log_{10} M$ over the $n$ completed runs with $N,M \ge 100$; $\pm$ one standard error). $R^2$ is for the log--log fit.}
    \label{tab:mcf_empirical_scale}
\end{table}

\section{Theoretical Background}

\subsection{Sankey Diagrams}\label{appx:Sankey}

A non-hierarchical sequence of partitions $\theta$ with change points $t_1< t_2< \dots < t_M$ can be visualised by a Sankey diagram $S(\theta)=(V,E)$~\citep{zarateOptimalSankeyDiagrams2018}, which is an $M$-layered graph with vertices $V=V_1\uplus \dots \uplus V_M$ and edges $E=E_1\uplus \dots \uplus E_{M-1}$ such that the nodes $V_m:=\theta(t_m)$ correspond to the clusters and edges $E_m=\{(C,C^\prime)\;|\; C\in\theta(t_m), C^\prime\in\theta(t_{m+1}): C\cap C^\prime\neq\emptyset \}$ connect clusters from consecutive partitions if they overlap. If one defines the weights $W:V\times V\to[0,\infty)$ such that $W(C,C^\prime)=\lvert C\cap C^\prime \rvert$, then $S(\theta)$ is weight-preserving, i.e., the weights from one level to the next always sum up to $\lvert X\rvert$.

Two Sankey diagrams are considered isomorphic when they are isomorphic as unweighted layered graphs:

\begin{definition}[Isomorphism between Sankey diagrams]\label{def:sankey_isomorphism}
    Two Sankey diagrams $S=(V_1\uplus \dots \uplus V_M,E_1\uplus \dots \uplus E_{M-1})$ and $S'=(V_1'\uplus \dots \uplus V_M',E_1'\uplus \dots \uplus E_{M-1}')$ are \emph{isomorphic} when there
    are bijections $\varphi_m \colon V_m \to V_m'$ for every layer $m = 1, \dots, M$ such that $(C,C')\in E_m$ iff
    $(\varphi_m(C),\varphi_{m+1}(C'))\in E_m'$.
\end{definition}

Note that a Sankey diagram isomorphism implies that the $m$-th bipartite layers $S_m:=(V_m\uplus V_{m+1}, E_m)$ and $S_m':=(V_m'\uplus V_{m+1}', E_m')$ are isomorphic as graphs for all $m=1,\dots,M-1$.

\subsection{Persistent Homology}\label{sec:background_ph}

Persistent homology (PH) reveals emergent topological properties of point cloud data (connectedness, holes, voids, etc.) in a robust manner by defining a filtered simplicial complex of the data and computing simplicial homology groups at different scales to track persistent topological features~\citep{edelsbrunnerTopologicalPersistenceSimplification2002}. Here we provide a summary of key concepts of PH---for detailed definitions see~\citet{otterRoadmapComputationPersistent2017,zomorodianComputingPersistentHomology2005,edelsbrunnerComputationalTopologyIntroduction2010,deyComputationalTopologyData2022}.

\subsubsection{Simplicial Complex}\label{background:complex}

Given a finite set %
$V$, a \emph{simplicial complex} $K$ is a subset of the power set $2^V$ (without the empty set) that is closed under the operation of building subsets. Its elements $\sigma\in K$ are called \emph{abstract simplices} and for a subset $\tau\subset\sigma$ we thus have $\tau\in K$ and $\tau$ is called a \emph{face} of $\sigma$.  A simplex $\sigma\in K$ is called $k$-dimensional if the cardinality of $\sigma$ is $k+1$ and the subset of $k$-dimensional simplices is denoted by $K_k\subset K$. The dimension  $\dim K$ of the complex $K$ is defined as the maximal dimension of its simplices.

We present three important examples of simplicial complexes used in this paper: (i) the \emph{solid simplex} $\Delta V:=2^V\setminus\{\emptyset\}$ is given by all non-empty subsets of $V$; (ii) the \emph{link complex} $\Lk \tau$ of a simplex $\tau\in K$ is defined as %
    $\Lk \tau:=\{\sigma\in K\;|\; \sigma\cap\tau=\emptyset, \, \sigma\cup\tau \in K\}$;
and (iii) the \emph{nerve complex}~\citep[p. 81]{matousekUsingBorsukUlamTheorem2003} $\Nrv(\mathcal{A})$, given a set $X$ and a family of subsets $\mathcal{A}=(U_\alpha)_{\alpha\in V}$ with $U_\alpha\subseteq X$ for all $\alpha\in V$,  is defined as 
     $\Nrv(\mathcal{A}):=\{\emptyset\neq \sigma \subseteq V\;|\; \bigcap_{\alpha\in \sigma}U_\alpha\neq \emptyset\}$, where the simplices are subsets of the index set $V$.

\subsubsection{Simplicial Homology}\label{background:homology}

For an arbitrary field $\mathbbm{F}$ (usually a finite field $\mathbbm{Z}_p$ for a prime number $p\in\mathbbm{N}$) %
and for all dimensions $k\in\{0,1,\dots,\dim K\}$ we define the $\mathbbm{F}$-vector space $C_k(K)$ with basis vectors given by %
$K_k$. The elements $c_k\in C_k(K)$ are called \emph{$k$-chains} and can be represented by a formal sum
    $c_k = \sum_{\sigma \in K_k} a_\sigma \sigma$
with coefficients $a_\sigma\in\mathbbm{F}$. After fixing a total order on the set of vertices $V$, we can define the so-called \emph{boundary operator} as a linear map $\partial_k:C_k \longrightarrow C_{k-1}$ through its operation on the basis vectors $\sigma=[v_0,v_1,\dots,v_k]\in K_k$ given by the alternating sum
    $$\partial_k(\sigma) = \sum_{i=0}^k (-1)^i [v_0,v_1,\dots,\hat{v_i},\dots,v_k],$$
where $\hat{v_i}$ indicates that vertex $v_i$ is deleted from the simplex. It is easy to show that the boundary operator fulfils the property $\partial_k \circ \partial_{k+1} = 0$, or equivalently, $\im{\partial_{k+1}}\subset \ker{\partial_k}$. Hence, the boundary operator connects the vector spaces $C_k$ for $k\in \{0,1,\dots,\dim K\}$ in a sequence of vector spaces and linear maps
    $$\dots \xrightarrow{\partial_{k+1}} C_k \xrightarrow{\partial_k} C_{k-1} \xrightarrow{\partial_{k-1}} \dots \xrightarrow{\partial_2} C_{1} \xrightarrow{\partial_1} C_{0} \xrightarrow{\partial_0} 0,$$
which is called a \emph{chain complex}. The elements in the \emph{cycle group} $Z_k := \ker{\partial_k}$ are called \emph{$k$-cycles} and the elements in the \emph{boundary group} $B_k := \im{\partial_{k+1}}$ are called the \emph{$k$-boundaries}. To characterise topological spaces by their holes or higher-dimensional voids, homology determines the non-bounding cycles, i.e., those $k$-cycles that are not the $k$-boundaries of $(k+1)$-dimensional simplices. The \emph{$k$-th homology group} $H_k$ of the chain complex is thus defined as the quotient of vector spaces
    $H_k := Z_k/B_k,$
whose elements are equivalence classes $[z]$ of $k$-cycles $z\in Z_k$. For each $k\in\{0,\dots,\dim K\}$, the rank of $H_k$ is called the \emph{$k$-th Betti number} denoted by $\beta_k$.

\subsubsection{Filtrations}\label{background:filtration}

To analyse topological properties across different scales, one defines a \emph{filtration} $\mathcal{F}$ of the simplicial complex $K$ as a sequence of $M$ increasing simplicial subcomplexes
    $K^1 \subseteq K^2 \subseteq \dots \subseteq K^M:=K,$
and we call $K$ a \emph{filtered complex}. %
In applications, the filtration is often indexed by a finite sequence of real numbers.
A common choice for point cloud data, i.e., a finite metric space $(V,d)$, is the \emph{VR filtration} $(K^\epsilon) _{\epsilon>0}$, which is defined as
 \begin{equation}\label{eq:VR_point_cloud}
K^\epsilon = \{ \sigma\subseteq V \;\;|\;\; \forall v,w\in \sigma: \;\;d(v,w) \le \epsilon \}.
 \end{equation}
As the set of vertices $V$ is finite, there are only finitely many %
values $0<\epsilon_1<\epsilon_2<\dots<\epsilon_M$ at which the simplicial complex $K^\epsilon$ changes and so $K^i:=K^{\epsilon_i}$ is a well-defined filtration. For network data, filtrations are often based on combinatorial features %
such as cliques under different thresholding schemes~\citep{horakPersistentHomologyComplex2009,aktasPersistenceHomologyNetworks2019}. Given an undirected graph $G=(V,E)$ with %
weighted adjacency matrix $A$ we define sublevel graphs $G_\delta = (V,E_\delta)$, where $E_\delta=\left\{\{i,j\}\in E \;|\; A_{ij}\le\delta\right\}$ is the set of edges with weight smaller than or equal to $\delta>0$, and the \emph{clique complex filtration} $(K^\delta) _{\delta>0}$ of $G$ is then given by
\begin{equation}\label{eq:clique_filtration}
    K^\delta = \{ \sigma\subseteq V \;\; | \;\; \sigma\;\; \text{is a clique in}\;\; G_\delta\}.
\end{equation}
Both the VR and clique complex filtrations %
have the property of being \emph{2-determined}, i.e., if each pair of vertices in a set $\sigma\subseteq V$ is a 1-simplex in a simplicial complex $K^i$ for $i\le M$, then $\sigma$ itself is a simplex in the complex $K^i$.

\subsubsection{Persistent Homology}\label{background:persistant_homology}

The goal of PH is to find the long- or short-lasting non-bounding cycles in a filtration. %
For each subcomplex $K^i$ with filtration index $i\in\{1,\dots,M\}$ and dimension $k\in\{0,1,\dots,\dim K\}$, we get a boundary operator $\partial_k^i$ and associated groups $C_k^i,Z_k^i,B_k^i,H_k^i$. For $p\ge 0$ such that $i+p\le M$, the \emph{$p$-persistent $k$-th homology group} of $K^i$ is defined as
\begin{equation}\label{eq:persistent_homology_group}
    H_k^{i,p} := Z_k^i\;/\left (B_k^{i+p}\cap Z_k^i \right ),
\end{equation}
which is well-defined because both $B_k^{i+p}$ and $Z_k^i$ are subgroups of $C_k^{i+p}$. We also write $H_k^p(K^i):=H^{i,p}_k$. %
Note that $H_k^{i,p}$ is isomorphic to the image of the map $f:H_k(K^i)\rightarrow H_k(K^{i+p})$ induced by the canonical inclusion $K^i\hookrightarrow K^{i+p}$
~\citep[p. 151]{edelsbrunnerComputationalTopologyIntroduction2010}.
The rank of %
$H_k^{i,p}$ is called the \emph{$p$-persistent $k$-th Betti number} of $K^i$, denoted by $\beta_k^{i,p}$, %
and can be interpreted as the number of non-bounding $k$-cycles that were born at filtration index $i$ or before and persist at least $p$ filtration indices, i.e., they are still `alive' in the complex $K^{i+p}$. Finally, the \emph{Betti curve} $\beta_k(i):=\beta_k^{i,0}$ records the standard Betti number of each simplicial complex $K^i$ in the filtration.

\subsubsection{Persistence Diagrams}\label{background:PDs}

The PH measures the `lifetime' of non-bounding cycles %
across the filtration $\mathcal{F}$. If a non-bounding $k$-cycle $[z]\neq 0$ emerges at filtration index $i$, i.e., $[z]\in H_k^i$, but was absent in $H_k^l$ for $l<i$, then we say that the filtration index $i$ is the \emph{birth} of the non-bounding cycle $[z]$. The \emph{death} $j$ is now defined as the first filtration index such that the previously non-bounding $k$-cycle is turned into a $k$-boundary in $H_k^j$, i.e., $[z]=0$ in $H_k^j$. The lifetime of the non-bounding cycle $[z]$ is then given by $j-i$. If a cycle remains non-bounding throughout the filtration, its death is formally set to $\infty$. One can compute the \emph{number of independent $k$-dimensional classes} $\mu_k^{i,j}$ that are born at filtration index $i$ and die at index $j=i+p$ as 
    $\mu_k^{i,j} = (\beta_k^{i,p-1}-\beta_k^{i,p})-(\beta_k^{i-1,p}-\beta_k^{i-1,p+1}),$
where the first difference computes the number of classes that are born at $i$ or before and die at $j$ and the second difference computes the number of classes that are born at $i-1$ or before and die at $j$. Drawing the set of birth-death tuples $(i,j)$ as points in the extended plane $\bar{\mathbbm{R}}^2=(\mathbbm{R}\cup\{+\infty\})^2$ with multiplicity $\mu_k^{i,j}$ and adding points on the diagonal with infinite multiplicity produces the $k$-\emph{dimensional persistence} diagram denoted by $\Dgm_k(\mathcal{F})$. The PD encodes all information about the PH groups because the Betti numbers $\beta_k^{i,p}$ can be computed from the multiplicities $\mu_k^{i,j}$. This is the statement of the \emph{Fundamental Lemma of PH}~\cite[p. 152]{edelsbrunnerComputationalTopologyIntroduction2010}:
\begin{equation}\label{eq:Fundamental_Lemma_PH}
    \beta_k^{i,p} = \sum_{\ell\le i}\sum_{j>i+p}\mu_k^{\ell,j}.
\end{equation}

\subsubsection{Distance Measures for PDs}\label{background:PD_distance}

For a filtration $\mathcal{F}$ on $K$ indexed by a finite sequence of real numbers one can define a \emph{filtration function} $f:K \rightarrow \mathbbm{R}$ such that $f(\sigma)=\min\{t\in \mathbbm{R}\; | \; \sigma \in K^t\}$ for $\sigma \in K$, which is \emph{simplex-wise monotone}, i.e., $f(\sigma^\prime) \le f(\sigma)$ for every $\sigma^\prime \subseteq \sigma\in K$~\citep{deyComputationalTopologyData2022}. The filtration $\mathcal{F}$ can be recovered from the sublevel sets %
    $K^t = f^{-1}((-\infty,t])$.
To compare two different filtrations $\mathcal{F}$ and $\mathcal{G}$ induced by filtration functions $f,g:K \rightarrow \mathbbm{R}$, %
it is possible to measure the similarity of their respective PDs $\Dgm_k(\mathcal{F}),\Dgm_k(\mathcal{G})\subset\bar{\mathbbm{R}}^2$. %
Let $\Phi = \{\phi:\Dgm_k(\mathcal{F})\rightarrow \Dgm_k(\mathcal{G})\}$ denote the set of bijections between the two diagrams. %
For $q\ge 1$, the \emph{q-th Wasserstein distance} is a metric on the space of PDs defined as
\begin{equation}\label{eq:Wasserstein_distance}
    d_{W,q}\left(\Dgm_k(\mathcal{F}),\Dgm_k(\mathcal{G})\right) = \inf_{\phi\in \Phi}\left[\sum_{x\in \Dgm_k(\mathcal{F})}(\|x - \phi(x)\|_q)^q\right]^{1/q},
\end{equation}
where $\|\cdot\|_q$ denotes the $L_q$ norm. %
Moreover, for $q=\infty$ one recovers the \emph{bottleneck distance}
\begin{equation}\label{eq:bottleneck-distance}
    d_{W,\infty}\left(\Dgm_k(\mathcal{F}),\Dgm_k(\mathcal{G})\right) = \inf_{\phi\in \Phi}\sup_{x\in \Dgm_k(\mathcal{F})}\|x - \phi(x)\|_\infty.
\end{equation}

\subsection{Zigzag Persistent Homology}\label{mcf_sec:background_zigzag}

We provide background on \emph{zigzag persistence}, which was first introduced by \citet{carlssonZigzagPersistence2010}. For additional details, see \citet{deyComputationalTopologyData2022}.

Let $Z^m$ be a simplicial complex defined on the set $X$ for  $m=1,\dots, M$. If either $Z^{m}\subseteq Z^{{m+1}}$ or $Z^{{m+1}}\subseteq Z^{{m}}$, for all $m=1,\dots,M-1$, we call the following diagram a \emph{zigzag filtration}:
    $Z^{1} \leftrightarrow Z^{2} \leftrightarrow \dots \leftrightarrow Z^{M-1} \leftrightarrow Z^{M}$,
where $Z^m\leftrightarrow Z^{m+1}$ is either a forward inclusion $Z^m\hookrightarrow Z^{m+1}$ or a backward inclusion $Z^m \hookleftarrow Z^{m+1}$. While forward inclusion corresponds to simplex addition, backward inclusion can be interpreted as simplex deletion.
Applying the homology functor $H_k$ to the zigzag filtration leads to a so-called \emph{zigzag persistence module} given by:
    $H_k(Z^{1}) \leftrightarrow H_k(Z^{2}) \leftrightarrow \dots \leftrightarrow H_k(Z^{M-1}) \leftrightarrow H_k(Z^{M})$,
where $H_k(Z^m)\leftrightarrow H_k(Z^{m+1})$ is either a forward or backward linear map. 

It has been shown that a zigzag persistence module has a unique interval decomposition that provides a zigzag PD %
as a simple invariant~\citep{carlssonZigzagPersistence2010}. Analogous to the PD of a standard filtration (Appendix~\ref{background:PDs}), the multiplicity of the point $(i,j)$ in the zigzag PD counts the number of homological features that are born at index $i$ and die at index $j$, where we take the death index to be the first index after $i$ at which the feature is no longer present (either because it is killed by the addition of simplices along a forward inclusion $Z^{j-1} \hookrightarrow Z^j$, or because it is destroyed by the deletion of simplices along a backward inclusion $Z^{j-1} \hookleftarrow Z^j$). We say that the feature is alive at the indices $i \le l < j$.

\vskip 0.2in
\bibliography{references}

\end{document}